\documentclass[a4paper,twoside]{article}  

\usepackage{amssymb,amsmath,amsthm,dsfont,amsfonts,color,latexsym}   
\usepackage{hyperref}
\usepackage{CJK}
\usepackage{mathrsfs} 

\definecolor{blue}{rgb}{0.00,0.00,1.00}
\definecolor{red}{rgb}{1.00,0.00,0.00}
\definecolor{green}{rgb}{0,0.60,0}
\definecolor{purple}{RGB}{250,0,150}
\definecolor{orange}{RGB}{255,165,0}
\newcommand{\blue}{\color{blue}}
\newcommand{\red}{\color{red}}

\newcommand{\ub}{\bar{u}}
\newcommand{\thetab}{\bar{\theta}}

\newcommand{\Mt}{\frac{M-\mu}{\sqrt{\mu}}}

\renewcommand{\baselinestretch}{1.2}
\def\ba{\begin{array}{ccc}}
	\def\bal{\begin{array}{lll}}
		\def\ea{\end{array}}

	\def\({\left(}\def\){\right)}
	\def\[{\left[}\def\]{\right]}

	\def \B   {\mathbf{B}}
	\def \L   {\mathcal{L}}
	\def \F_i   {\mathbf{F}}

	\def\g    {\mathbf{g}_2}
	\def\h    {\mathbf{h}}
	\def\gt    {\mathbf{g}_1}

	\def\eps  {\delta}
	\def\intr {\int_{\mathbb{R}^3}}
	\def\intra {\int_{\mathbb{R}}}

	\def \p   {\partial}
	
	\def \pt   {\partial}

	\def \dt    {\partial_{t}}
	\def \dtau    {\partial_{\tau}}
	
	\def \dy    {\partial_x}

	\def \R   {\mathcal{R}}

	\def \dy    {\partial_y}
	\def \dx    {\partial_{x}}

	\def \dya   {\partial^{\alpha}}

    \def \A   {\mathbf{A}}
    \def \B   {\mathbf{B}}
	\def \F_i   {\mathbf{F}}

	\def\eps  {\epsilon}
	\def\intr {\int_{\mathbb{R}^3}}
	\def\intra {\int_{\mathbb{R}}}

	\def \p   {\partial}
	
	\def \pt   {\partial}

	\def \dt    {\partial_{t}}
	\def \dtau    {\partial_{\tau}}
	
	\def \dy    {\partial_x}

	\def \R   {\mathbb{R}}

	\def \dy    {\partial_y}
	\def \dx    {\partial_{x}}

	\def \dya   {\partial^{\alpha}}

	\newcommand{\norm}[1]{\left\lVert#1\right\rVert}

	\newcommand{\norms}[1]{\left\lVert#1\right\rVert_{\sigma}}

	\def\v{\langle v\rangle}
	\def\O{\mathbb{O}}
	\def \dya   {\p^\alpha}

	\def\Tdv   {\nabla_v}

	\def\bq{\begin{equation}}
		\def\eq{\end{equation}}
	\def\be{\begin{equation}}
		\def\ee{\end{equation}}
	\def\bma#1\ema{{\allowdisplaybreaks\begin{align}#1\end{align}}}
	\def\bmas#1\emas{{\allowdisplaybreaks\begin{align*}#1\end{align*}}}
	\def\bln#1\eln{{\allowdisplaybreaks\begin{aligned}#1\end{aligned}}}
	\def\nnm{\notag}
	\def\bgr#1\egr{\allowdisplaybreaks\begin{gather}#1\end{gather}}
	\def\bgrs#1\egrs{\allowdisplaybreaks\begin{gather*}#1\end{gather*}}
	\renewcommand{\theequation}{\arabic{section}.\arabic{equation}}

	\theoremstyle{plain}
	\newtheorem{lem}{\bf Lemma}[section]
	\newtheorem{thm}[lem]{\textbf{Theorem}}
	\newtheorem{prop}[lem]{\textbf{Proposition}}
	\newtheorem{rem}[lem]{\textbf{Remark}}

	\renewcommand{\theequation}{\arabic{section}.\arabic{equation}}

\begin{document}

		\title{Hydrodynamic limit of rarefaction wave for the Vlasov-Maxwell-Landau system with Coulomb potential}
		
		\author{ Guanghui Wang$^1$,\ Lingda Xu$^2$,\, Tong Yang$^3$,\, Mingying Zhong$^{ 4}$\\
\emph
 {\small\it $^1$School of  Mathematics,
    Guangxi University, Nanning, China}\\
      {\small\it E-mail:\ guanghui06@foxmail.com (G.-H. Wang)}\\
  {\small\it $^2$Department of Applied Mathematics, The Hong Kong Polytechnic University, Hong Kong, China}\\
      {\small\it E-mail:\ xuldmath@gmail.com (L.-D. Xu)}\\
      	{\small\it $^3$ Department of Applied Mathematics, The Hong Kong Polytechnic University, Hong Kong,  China}\\
      	{\small\it E-mail:\ t.yang@polyu.edu.hk (T. Yang)}\\
   {\small\it $^4$Center for Applied Mathematical of Guangxi (Guangxi University), Nanning,  China}\\
      {\small\it E-mail:\ zhongmingying@gxu.edu.cn (M.-Y. Zhong)}\\
    }
\date{ }
		
		\pagestyle{myheadings}
		\markboth{Rarefaction wave for the VML system}
		{G.-H. Wang, L.-D. Xu, T.Yang and M.-Y. Zhong}

		\maketitle
		
\begin{abstract}\noindent
In this paper, we investigate the hydrodynamic limit of  rarefaction wave for  the two-species Vlasov-Maxwell-Landau(VML) system with Coulomb potential. We prove that  for any given time interval, 
the solution of the Vlasov-Maxwell-Landau system with appropriate initial data converges to a rarefaction wave as the Knudsen number $\epsilon$ approaches zero.
The main difficulty in the analysis lies in the loss of dissipation in the interaction between the electromagnetic field and the microscopic component, and the weak dissipation induced by  the Lorentz force and the scaling with small parameter $\epsilon$. For this, we introduce a velocity weight function and a space-time scaling parameter 
 together  with suitable
$\epsilon$-dependent energy estimates.
			
			
		\end{abstract}
		\tableofcontents
		
		\thispagestyle{empty}	
		
\section{Introduction}

\subsection{The model and previous results}
We study the following 1D two-species Vlasov-Maxwell-Landau (VML) system
\be
\left\{\bln     \label{VML1z}
&  \dt F_++v_1\dx F_++(E+v\times B)\cdot\Tdv F_+ =\epsilon^{-1}Q(F_+,F_+)+\epsilon^{-1}Q(F_+,F_-), \\
&  \dt F_-+v_1\dx F_--(E+v\times B)\cdot\Tdv F_- =\epsilon^{-1}Q(F_-,F_-)+\epsilon^{-1}Q(F_-,F_+), \\
& \dt E_1=-\intr (F_+-F_-)v_1dv,\quad \dx E_1=\intr (F_+-F_-)dv,\\
&\dt E_2=-\dx B_3-\intr (F_+-F_-)v_2dv,\\
& \dt E_3=\dx B_2-\intr (F_+-F_-)v_3dv,\\
& \dt B_2=\dx E_3,\quad  \dt B_3=-\dx E_2,\quad B_1\equiv0,
\eln\right.
\ee
where the Knudsen number $\epsilon>0$ is a small parameter proportional to
the mean free path, $F_{\pm}=F_{\pm}(t,x,v)$ are number density distribution functions of ions $(+)$ and electrons $(-)$ having position $x\in \mathbb{R}$ and velocity $v=(v_1,v_2,v_3)\in \mathbb{R}^3$ at time $t\in \mathbb{R}^+$, and $E(t,x)$, $B(t,x)$ denote the electro and magnetic fields, respectively.
		
In this paper, we consider the Landau collision operator
\be\label{landau}		Q(G_1,G_2)(v)
=\nabla_v\cdot\int_{\R^3}\varphi(v-u)\{G_1(u)\nabla_v G_2(v)-\nabla_{u}G_1(u)G_2(v)\}du.
\ee
The non-negative matrix $\varphi$ in the above integral  is given by
\be\label{varphi}
\varphi(v)=\left(\mathbb{I}-\frac{v\otimes v}{|v|^2}\right)|v|^{\gamma+2},\quad\gamma\in[-3,-2),
\ee
where $\mathbb{I}$ is the $3\times3$ identity matrix and $v\otimes v$ is the tensor product. Note that \eqref{landau} with $\gamma=-3$ corresponds to the (Fokker-Planck)-Landau collision operator for Coulomb potential. Through the paper, we  focus on the very soft potentials case $-3\leq\gamma<-2$.
		
\medskip

		The Vlasov-Maxwell-Landau (Boltzmann) system is a fundamental model in  plasma physics that can be used for describing the time evolution of  rarefied  charged entities like electrons and ions. The particles are subject to the effect self-induced Lorentz force governed by the Maxwell's equations. One can find its physical background and explanantion in  \cite{ChapmanCowling,Markowich}.
		In term of mathematical analysis, extenstive study has been done with fruitful results,
such as those in \cite{Duan2014,Guo4,Li1,Strain,Wang} and the references therein.
In more details, the global existence of strong solutions for the VMB sytems with initial data near a global Maxwellian is established, both in periodic domains (cf. \cite{Guo4}) and spatial 3D space (cf. \cite{Strain}) for hard sphere model, and  for the very soft potentials in \cite{DLYZ}.
 And the global existence of strong solutions near a Maxwellian  in the whole space was obtained for the two-species Vlasov-Maxwell-Landau system in \cite{Duan2014,Wang}.
The large time behavior near a global Maxwellian was studied in \cite{Duan5,Duan4}.
Moreover, the authors in \cite{Li1} analyzed spectrum structure and time-decay rates of solutions to the VMB systems,
and  the authors  in \cite{Li4} also studied the Green's function and  the pointwise estimates for global solutions to the two-species VMB system. Note that the authors in \cite{Liu1,Liu3} firstly studied the  Green's function  of Boltzmann equation.
We also refer to \cite{Li2,Li5} for the pointwise behavior of Green's function to the VPB system.
		
		
Important progress has also been made on  the fluid dynamic limits of the VMB system. In the incompressible regime, Jang \cite{Jang,Jiang} investigated a diffusive expansion of the VMB system within the framework of classical solutions, noting that the magnetic effect appears only as a higher-order term. Recently, \cite{YZDL,CYZ} studied the diffusion limit of classical solutions to the VMB system and obtained the optimal convergence rate by introducing a new decomposition of the solution to identify the essential components for generating the initial layer. Ars$\acute{e}$nio and Saint-Raymond \cite{AS} conducted a thorough analysis, justifying various limits to incompressible electromagneto-hydrodynamics for the VMB system under different scalings. 
In addition, \cite{Jang1} provides formal derivations of the fluid equations in the compressible regime.
 \cite{Duan7} verified the compressible Euler-Maxwell limit for  the one-species VMB system.
		
		It is well-known that a typical hyperbolic system like the Euler equation has rich wave phenomena. The main feature of the 1D Euler system is that no matter how smooth or small the initial data is,  singularity will form in  finite time.
		 Research on basic wave patterns, such as  shock waves, rarefaction waves and contact discontinuities  provides understanding on the solution behavior to these systems.
		
		In this paper, we will study hydrodynamic limit for rarefaction wave  of the VML system \eqref{VML1z}. Even though  the wave phenomena for the Boltzmann equation have been extensively studied, cf.  \cite{LYYZ,xinYY,Ukai3,Duan6,HY,LY,HWWY,WangY} and the references therein,
		to our knowledge, there is no corresponding result on the VML system because of the complexity of the system.

For the two-species VML system \eqref{VML1z}, by setting
$$
F_1=F_++F_-\text{\quad and\quad } F_2=F_+-F_-,
$$
one has
\be
\left\{\bln     \label{VML2}
&\dt F_1+v_1\dx F_1+(E+v\times B)\cdot\Tdv F_2=\epsilon^{-1}Q(F_1,F_1),\\
&\dt F_2+v_1\dx F_2+(E+v\times B)\cdot\Tdv F_1=\epsilon^{-1}Q(F_2,F_1),\\
& \dt E=\mathbb{O}\dx  B-\intr F_2v dv,\quad \dx E_1=\intr F_2dv,\\
& \dt B=- \mathbb{O}\dx E,\quad B_1\equiv0,
\eln\right.
\ee
where $\O$ is a $3\times 3$ matrix defined by
\be\label{matrix}
\O=\left(\ba
0 & 0 & 0\\
0 & 0 & -1 \\
0 & 1 & 0
\ea\right).
\ee
We assume that the initial data satisfy
\be\label{initial0}
\left\{\bln
&F_1(0,x,v)=F_{10}(x,v)\to M_{[\rho_\pm,u_\pm,\theta_\pm]}
:=\frac{\rho_\pm}{(2\pi R\theta_\pm)^{\frac32}}e^{-\frac{|v-u_\pm|^2}{2R\theta_\pm}}
 \quad as~~x\to\pm\infty,\\
&F_2(0,x,v)=F_{20}(x,v)\to 0\quad as~~x\to\pm\infty,\\
&(E,B)(0,x)=(E_0,B_0)(x)\to 0\quad as~~x\to\pm\infty,
\eln\right.
\ee
where $(\rho_\pm,u_\pm,\theta_\pm)$ are two constant states with $u_\pm=(u_{1\pm},0,0)$ and $\rho_\pm>0,~\theta_\pm>0$, and $R>0$ is the gas constant.

\if		
		Our aim of this paper is to study the hydrodynamic limits from the VML system to the compressible Euler system around rarefaction waves. To be accurate, for any given time, we construct a family of solutions to the two-species VML system, which converges to the local Maxwellian constructed by the rarefaction wave of the Euler equation as the Knudsen number tends to zero. The rough version of the main theorem is as follows. Due to the complexity of notations, the main theorem will be given in the section 3.
		
		\begin{thm}[Rough version] Letting $\bar{M}$ be the local Maxwellian defined by the rarefaction wave of the Euler system, $\mu$ be the global Maxwellian near $\bar{M}$, $(F_1^{\epsilon},F_2^{\epsilon},E^{\epsilon},B^{\epsilon})$ be the solution of \eqref{VML2} with suitable initial data, we have the following
\begin{align}
\norm{\intr \frac{\big\lvert F^{\epsilon}_1-\bar{M}\big\lvert^2}{\mu}dv}_{L^\infty_{\R}}
+\norm{\intr \frac{\big\lvert F^{\epsilon}_2\big\lvert^2}{\mu}dv}_{L^\infty_{\R}}
+\norm{E^{\epsilon},B^{\epsilon}}_{L^{\infty}_{\R}}\to 0 \text{\quad as\quad } \epsilon\to 0.
\end{align}
\end{thm}
		\begin{rem}
			The convergence rate and comments will be listed in Theorem \ref{mt}.
		\end{rem}
	
		The main difficulty is due to the strong interaction of electromagnetic fields $(E,B)$ and $F_2$. New ideas should be applied, we introduce our strategy here.

		\begin{itemize}
			
			\item Introduction of correction term $\bar{G}_1$ for the microscopic equation.
			
			Since the profile we studied is a local Maxwellian, after applying the macro-microscopic decomposition, one finds there are terms with slow convergence rates. We borrow the idea of \cite{LYYZ}, which studied the long-time behavior of rarefaction wave for Boltzmann equation,  to construct a correction term $\bar{G}_1$. Then the errors of the microscopic equation have been improved, see Lemma \ref{Gb}.
			
			\item A carefully refined energy estimate.
			Careful calculations yield that in the lower-order estimates, the higher-order terms will involve small coefficients, for example
			$$\epsilon^{1-a}\int_{\R}\int_{\R^3}\frac{|\partial^{\alpha+1}_y{G}_2|^2}{\mu}dvdy,$$
			where $a$ is the scaling parameter $1/2<a<1$ and $y=x/\epsilon^a$. This observation is the key to designing an $\epsilon$-dependence energy form.
			
			\item An $\epsilon$-dependence energy form.
			
			The Sobolev embedding theorem and the observation in the last item indicate that one actually doesn't need to have a uniform upper bound in the highest-order estimate. An estimate controlling the increasing rates of the highest-order derivatives is enough to obtain the main theorem. That is the reason why we use an $\epsilon$-dependence energy form.
		\end{itemize}		
\fi

\subsection{Macro-micro equations and dissipation}
We now  briefly describe the micro-macro decomposition
around the local Maxwellian defined by the solution of the VML system \eqref{VML2}. For the solution $F_1(t,x,v)$ of the VML system $\eqref{VML2}_1$, as in \cite{Liu2}, we decompose it into
\be\label{mac-mic}
F_1(t,x,v)=M(t,x,v)+G_1(t,x,v),
\ee
where the local Maxwellian $M$ and microscopic component $G_1$ represent the fluid and non-fluid(kinetic) parts in the solution respectively.
Here, the local Maxwellian $M$ associated with the solution $F_1(t,x,v)$ to the VML system \eqref{VML2} is defined by the  conserved quantities, the mass density $\rho(t,x)$, momentums $\rho u_{i}(t,x)~(i=1,2,3)$,
and energy density $\rho(e+\frac{|u|^{2}}{2})(t,x)$,
\bq
 \left\{\bln \label{conserve}
 &\rho(t,x)=\int_{\R^{3}}\varphi_0(v)F_1(t,x,v)dv,\\
 &\rho u_{i}(t,x)=\int_{\R^{3}}\varphi_i(v)F_1(t,x,v)dv,\quad i=1,2,3,\\
 &\rho\(e+\frac{|u|^{2}}{2}\)(t,x)=\int_{\R^{3}}\varphi_4(v)F_1(t,x,v)dv,
 \eln\right.
 \eq
where the five collision invariants $\varphi_i(v)~(i=0,1,2,3,4)$ are given by
 $$
 \varphi_0(v)=1,\quad \varphi_i(v)=v_{i}~(i=1,2,3),\quad \varphi_4(v)=\frac{1}{2}|v|^{2}.
 $$
Then, the local Maxwellian $M$ can be defined by
\be\label{M}
M\equiv M_{[\rho,u,\theta] }(v)=\frac{\rho(t,x)}{(2\pi R\theta(t,x))^{\frac{3}{2}}}e^{-\frac{|v-u(t,x)|^{2}}{2R\theta(t,x)}},
\ee
where $\rho$ is the mass density, $u=(u_{1},u_{2},u_{3})$ is the macroscopic velocity,   and $\theta$ is the temperature which is related to the internal energy with
$$
e=\frac{3}{2}R\theta=\theta,\quad R=\frac23.
$$

For any given Maxwellian $\tilde{M}$ and any functions $f(v),g(v)\in L^2(\R^3_v)$, we define the weighted inner product space in $L^2(\R^3)$ space with
\be
\langle f,g\rangle_{\tilde{M}}=\intr \frac{1}{\tilde{M}}f(v)g(v)dv, \quad \|f\|^{2}_{L^{2}_{v}(\frac{1}{\sqrt{\tilde{M}}})}=\langle f,f\rangle_{\tilde{M}}.\nnm
\ee

With  the micro-macro decomposition in \eqref{mac-mic}, the VML system $\eqref{VML2}_1$ can be rewritten as
\be \label{F_1}
\dt(M+G_1)+v_1\dx(M+G_1)+(E+v\times B)\cdot\Tdv F_2=\epsilon^{-1}L_M G_1+\epsilon^{-1}Q(G_1,G_1),
\ee
where $L_M$ is the linearized collision operator given by
\be\label{collision1}
L_M G_1=Q(G_1,M)+Q(M,G_1).
\ee
Moreover, the null space $N_0$ of the operator $L_M$ is spanned by the following five pairwise orthogonal basis:
\bq
\bln \label{basis}
&\chi_0=\frac{1}{\sqrt{\rho}}M,\quad \chi_j=\frac{v_{i}-u_{i}}{\sqrt{R\rho\theta}}M,\quad  i=1,2,3 ,\\
&\chi_4=\frac{1}{\sqrt{6\rho}}\(\frac{|v-u|^2}{R\theta}-3\)M,\quad\langle\chi_i,\chi_j \rangle_M=\delta_{ij}, \quad i,j=0,1,2,3,4.
 \eln
 \eq
With  these orthonormal functions, we define the macroscopic projection $P_0$ and microscopic projection
$P_1$ as follows:
\be
P_0F_1=\sum_{i=0}^4\langle F_1,\chi_i \rangle_M\chi_i,\quad  P_1=I- P_0. \label{P10}
\ee
Thus,
 \be\label{M1} M=P_{0}F_1,\quad G_1=P_{1}F_1.
\ee	
Next, for $F_2$, there is only one conserved quantity $n = n(t, x)$ defined by
\be
n(t, x)=\intr F_2(t, x,v)dv.
\ee
As in \cite{WangY}, we introduce another macro-micro projection around the local Maxwellian $M$ such that
\be\label{Pd}
 P_dF_2=\langle F_2,\chi_0\rangle_M\chi_0 =\frac{n}{\rho}M,\quad P_rF_2=F_2-P_dF_2=G_2.
\ee

Then, we write $F_2(t, x, v)$  as
\be\label{F2}
F_2=\frac{n}{\rho}M+G_2.
\ee
Hence,
we can rewrite $\eqref{VML2}_2$ as
\bma\label{F_2}	
&\quad\dt(\frac{n}{\rho}M+G_2)+v_1\dx (\frac{n}{\rho}M+G_2)+ (E+v\times B)\cdot\Tdv (M+G_1) \nonumber\\
&= \epsilon^{-1}\L_M G_2+\epsilon^{-1}\frac{n}{\rho}Q(M,G_1)+\epsilon^{-1}Q(G_2,G_1),	
\ema
where $\L_M$ is the linearized collision operator given by
\be\label{collision2}
 \L_M G_2=Q(G_2,M).
\ee

In the following, we will write down the macroscopic equations. By using
\bma
&\intr \varphi_j(\dt F_1+v_1\dx F_1+(E+v\times B)\cdot\Tdv F_2)dv=0,\,\,\, j=0,1,2,3,4,\nnm\\
&\intr \varphi_0(\dt F_2+v_1\dx F_2+(E+v\times B)\cdot\Tdv F_1)dv=0,\nnm
\ema
we have the following system for $(\rho,u,\theta,n)$:
\be\label{macro-eq-nond}
\left\{\bln
&\dt \rho+\dx(\rho u_1)=0,\\
&\dt (\rho u_1)+\dx (\rho u_1^2)+\dx p-n(E_1+(u\times B)_1)=-\intr v_1^2\dx G_1dv+\intr (v\times B)_1 G_2dv,\\
&\dt (\rho u_i)+\dx (\rho u_1u_i)-n(E_i+(u\times B)_i)=-\intr v_1v_i\dx G_1dv+\intr (v\times B)_i G_2dv, ~ i=2,3,\\
&\dt \[\rho\(\theta+\frac12 |u|^2\)\]+\dx \[u_1\(\rho\(\theta+\frac12 |u|^2\)+p\)\]-nu\cdot E =-\frac12\intr v_1|v|^2\dx G_1dv+\intr v\cdot E G_2dv,\\
&\dt n+\dx(n u_1)=-\intr v_1\dx G_2dv,
\eln\right.
\ee
where $p=R\rho\theta$ with $R=\frac23$.
Moreover, the governing equations of the electromagnetic field satisfy
\be\label{EB}
\left\{\bln
& \dt E_1= -nu_1-\intr v_1G_2dv,\quad \dx E_1=n,\\
&\dt E_2=-\dx B_3-nu_2-\intr v_2G_2dv, \\
&\dt E_3=\dx B_2-nu_3-\intr v_3G_2dv,\\
& \dt B_2=\dx E_3,\quad  \dt B_3=-\dx E_2,\quad B_1\equiv0 .
\eln\right.
\ee		

Applying $P_1$ to \eqref{VML2}$_{1}$, we have
\bma\label{G0-1}
&\dt G_1+P_1(v_1\dx M)+P_1(v_1\dx G_1)+P_1[(E+v\times B)\cdot\Tdv G_2]
= \epsilon^{-1}L_MG_1+\epsilon^{-1}Q(G_1,G_1).
\ema
Applying $P_r$ to \eqref{VML2}$_{2}$, we have
\bma\label{G0-2}
&\dt G_2+P_r(v_1\dx G_2 )+\frac{n}{\rho}P_r(\dt M+v_1\dx  M)+ (E+v\times B)\cdot\Tdv G_1 \nnm\\
&+\(\dx(\frac{n}{\rho}) P_r(v_1M)-\frac{E+u\times B}{R\theta}\cdot P_r(vM)\)
=\epsilon^{-1}\L_MG_2 +\epsilon^{-1}Q(F_2,G_1),
\ema
where we have used $P_d[(E+v\times B)\cdot\Tdv G_1]=0.$
Thus,
\bmas
G_1&= \epsilon L_M^{-1} P_1(v_1\dx M)+L_M^{-1}\Lambda_1,\\
G_2&= \epsilon\dx(\frac{n}{\rho})\L_M^{-1}P_r(v_1M)-\epsilon\frac{E+u\times B}{R\theta}\cdot \L_M^{-1} P_r(vM) +\L_M^{-1}\Lambda_2,
\emas
where
\bma
\Lambda_1&= [\epsilon\dt G_1+\epsilon P_1(v_1\dx G_1)-Q(G_1,G_1)]+\epsilon P_1[(E+v\times B)\cdot\Tdv G_2] ,\nnm\\
\Lambda_2&= [\epsilon\dt G_2+\epsilon P_1(v_1\dx G_2)-Q(F_2,G_1)]-\frac{n}{\rho}\epsilon P_r(\dt M+v_1\dx M)
 +\epsilon (E+v\times B)\cdot\Tdv G_1.\nnm
\ema
		
Next, we will study the dissipation structure.
Note that (cf. \cite{WangY})
$$\left\{\begin{aligned}		
&- \intr v_1v_iL_M^{-1} P_1(v_1\dx M)dv
	=\kappa_1(\theta)\dx u_i+\frac13\delta_{1i} \kappa_1(\theta)\dx u_1,\\
&- \intr \frac12v_1|v|^2L_M^{-1} P_1(v_1\dx M)dv
    =\frac13 \kappa_1(\theta)u_1\dx u_1+\kappa_1(\theta) u\cdot\dx u+ \kappa_2(\theta)\dx \theta,\\
&- \intr v_i\L_M^{-1} P_r(v_j M)dv=\delta_{ij}\sigma(\theta),
\end{aligned}\right.
$$
where the coefficients $\kappa_1(\theta)$, $\kappa_2(\theta)>0$ and $\sigma(\theta)>0$ are smooth functions depending  only on the temperature $\theta$. The explicit formula of $\kappa_1(\theta)$ and $\kappa_2(\theta)$ are defined in \eqref{viscosity}-\eqref{heat}. 

Then we have the following system with viscosity compared to \eqref{macro-eq-nond},
\be\label{ori-sys}
\left\{\bln
&\dt \rho+\dx(\rho u_1)=0,\\
&\dt (\rho u_1)+\dx (\rho u_1^2)+\dx p-n(E_1+(u\times B)_1)=\epsilon\frac43\dx(\kappa_1(\theta)\dx u_1)\\
		&\qquad  +\epsilon\sigma(\theta)\frac{[(E+u\times B)\times B]_1}{R\theta}-\intr v_1^2\dx L_M^{-1}\Lambda_1dv+\intr (v\times B)_1 \L_M^{-1}\Lambda_2dv,\\
&\dt (\rho u_i)+\dx (\rho u_1u_i)-n(E_i+(u\times B)_i)=\epsilon\dx(\kappa_1(\theta)\dx u_i)-\epsilon\sigma(\theta)\dx(\frac{n}{\rho})(e_1\times B)_i\\
		&\qquad  +\epsilon\sigma(\theta)\frac{[(E+u\times B)\times B]_i}{R\theta}-\intr v_1v_i\dx L_M^{-1}\Lambda_1dv+\intr (v\times B)_i \L_M^{-1}\Lambda_2dv, ~ i=2,3,\\
&\dt \[\rho\(\theta+\frac12 |u|^2\)\]+\dx \[u_1\(\rho\(\theta+\frac12 |u|^2\)+p\)\]-nu\cdot E=\epsilon\dx(\kappa_2(\theta)\dx \theta)\\
		&\qquad  +\epsilon\frac13\dx(\kappa_1(\theta)u_1\dx u_1)+ \epsilon\dx(\kappa_1(\theta)u\cdot\dx u)-\epsilon\sigma(\theta)\dx(\frac{n}{\rho})E_1\\
		&\qquad  +\epsilon\sigma(\theta)\frac{(E+u\times B)\cdot E}{R\theta}-\frac12\intr v_1|v|^2\dx L_M^{-1}\Lambda_1dv+\intr v\cdot E \L_M^{-1}\Lambda_2dv,\\
&\dt n+\dx(n u_1)-\epsilon\dx\(\sigma(\theta)\dx(\frac{n}{\rho})\)+\epsilon\dx\(\sigma(\theta)\frac{E_1+(u\times B)_1}{R\theta}\)=-\intr v_1\dx \L_M^{-1}\Lambda_2dv,
		\eln\right.
		\ee  where $e_1=(1,0,0)$.

Moreover, the governing equations of the electromagnetic field satisfy
\be\label{EB-inv}
\left\{\begin{aligned}
& \dt E_1= -nu_1+\epsilon\sigma(\theta)\dx(\frac{n}{\rho})-\epsilon\sigma(\theta)\frac{E_1+(u\times B)_1}{R\theta}-\intr v_1\L_M^{-1}\Lambda_2dv, \\
&\dt E_2=-\dx B_3-nu_2-\epsilon\sigma(\theta)\frac{E_2+(u\times B)_2}{R\theta}-\intr v_2\L_M^{-1}\Lambda_2dv, \\
&\dt E_3=\dx B_2-nu_3-\epsilon\sigma(\theta)\frac{E_3+(u\times B)_3}{R\theta}-\intr v_3\L_M^{-1}\Lambda_2dv.
\end{aligned}\right.
\ee

\subsection{Rarefaction wave and  approximation}
Now we turn to describe the rarefaction wave profile to the system $\eqref{VML2}_1$ as in \cite{LYYZ}.
Consider the compressible Euler system
\be
\label{Euler1}
\left\{\begin{aligned}
&\dt \rho+\dx(\rho u_1)=0,\\
&\rho(\dt  u_1+u_1\dx u_1)+\dx p=0,\\
&\frac32R\rho(\dt \theta+u_1\dx \theta)+ p\dx u_1 =0,
\end{aligned}
\right.
\ee
 with Riemann  initial data:
\be \label{Euler-in}
(\rho,u_1,\theta)(0,x)=
\left\{\begin{aligned}
(\rho_-,u_{1-},\theta_-), \quad x<0,\\
(\rho_+,u_{1+},\theta_+), \quad x>0.
\end{aligned}\right.
\ee
Here the two states $(\rho_\pm,u_\pm,\theta_\pm)$ are the same as \eqref{initial0}.
Set the state equation
$$
p=\frac23\rho\theta=k\rho^{\frac53}\exp(S),
$$
where $k=\frac1{2\pi \mathbf{e}}$($\mathbf{e}$ is the natural constant), and $S$ is the macroscopic entropy given by
$$
S=-\frac23\ln\rho+\ln\frac43\pi\theta+1.
$$
Then we can rewrite the Euler system \eqref{Euler1} in terms of $(\rho,u,S)$ with
$u=(u_1,0,0)$ as
\bq\label{Euler2}
 \left\{\bln
&\rho_t+(\rho u_1)_x=0,\\
&\rho u_{1t}+\rho u_1u_{1x}+p_x=0,\\
&S_t+u_1S_x=0,
\eln\right.
\eq
with the following Riemann initial data:
\bq\label{initial1}
(\rho,u_1,S)(0,x)=\left\{\bln(\rho_{-},u_{1-},S^*),~~x<0,\\
(\rho_{+},u_{1+},S^*),~~x>0,
\eln\right.
\eq
where $S^*=S_+=S_-$.

The Euler system \eqref{Euler2} for $\left( {\rho ,{u}_{1},S}\right)$ has three distinct eigenvalues
$$
\lambda_i(\rho,u_1,S)=u_1+(-1)^{\frac{i+1}{2}}\sqrt{p_\rho(\rho,S)},~ i=1,3,\quad
\lambda_2(\rho,u_1,S)=u_1,
$$
where $p_\rho(\rho,S)=\frac53k\rho^{\frac23}\exp{(S)}>0$ and the corresponding right eigenvectors
$$
r_i(\rho,u_1,S)= \big((-1)^{\frac{i+1}{2}}\rho,\sqrt{p_\rho },0 \big)^T~\mathrm{for}~i=1,3,\quad r_2(\rho,u_1,S)=(p_{S},0,-p_{\rho})^T.
$$
The two $i$-Riemann invariants $R_{i,j}~(i=1,3,j=1,2)$ can be defined by (cf.\cite{Lax,Smoller})
$$
R_{i,1}=u_1+(-1)^{\frac{i-1}{2}}\int^{\rho}\frac{\sqrt{p_z(z,S)}}{z}dz,\quad R_{i,2}=S.
$$

In terms of the two $i$-Riemann ($i=1,3$) invariants, we can define the $i$-rarefaction wave curve for the given Riemann initial data (cf. \cite{Lax}).
Without loss of generality, we consider only the 3-rarefaction wave in this paper, and the case for 1-rarefaction wave can be treated similarly.
We start from the Riemann solution to the inviscid Burgers equation
\bq
\left\{\bln\label{bgs}
&w_{t}+ww_{x}=0,\quad (t,x)\in \R^{+}\times\R,\\
&w(0,x)=
 \begin{cases}
  w_{-},\quad x<0,\\
  w_{+},\quad x>0.
 \end{cases}
\eln\right.
\eq
If $w_{-}<w_{+}$, then the Riemann problem \eqref{bgs} admits a self-similar rarefaction wave fan solution
$w^{r}(t,x)=w^{r}(\frac{x}{t})$ connecting $w_{-}$ and $w_{+}$ given by 
\bq\label{rarefaction3}
w^{r}(\frac{x}{t})=
\left\{\bln
&w_{-},\quad\quad\frac{x}{t}\leq w_{-},\\
&\frac{x}{t},\quad\quad w_{-}\leq\frac{x}{t}\leq w_{+},\\
&w_{+},\quad\quad\frac{x}{t}\geq w_{+}.
\eln\right.
\eq
The 3-rarefaction wave $(\rho^r,u^r,\theta^r)(\frac{x}{t})$ with $u^r(\frac{x_1}{t})=(u_1^r,u_2^r,u_3^r)(\frac{x}{t})$ to the Riemann problem \eqref{Euler1}-\eqref{Euler-in} can be defined explicitly by (cf. \cite{Lax,Smoller})
\bq\label{rarefaction2}
\left\{\bln
&\lambda_3(\rho^{r}(x/t),u^r(x/t),S^*)=w^{r}(x/t),\quad w_{\pm}=\lambda_3(\rho_{\pm}, u_{1\pm}, S^*),\\
&u^r_1-\sqrt{15k}(\rho^r)^{\frac13}\exp{(S^*/2)}=u_{1-}-\sqrt{15k}\rho_-^{\frac13}\exp{(S^*/2)},
 \quad u^r_2(x/t)=u^r_3(x/t)=0,\\
&\theta^r(x/t)=\frac32k\exp(S^*)(\rho^r)^{\frac23}(x/t).
\eln\right.
\eq
Since the above 3-rarefaction wave is only Lipschitz continuous, we shall construct an approximate smooth rarefaction wave to the 3-rarefaction wave defined in \eqref{rarefaction2}. As in \cite{Matsumura,xin1993},  the approximate smooth rarefaction wave can be constructed by the
Burgers equation with smooth initial data
 \bq
 \left\{\bln\label{BE}
 &w_{t}+ww_{x}=0,\quad (t,x)\in \R^{+}\times\R,\\
 &w|_{t=0}=\frac{w_{+}+w_{-}}{2}+\frac{w_{+}-w_{-}}{2}\text{tanh}(\frac{x}{\delta}),
 \eln\right.
 \eq
where $\delta>0$ is a small constant chosen later.  By the method of characteristic curves,
 the solution $w_{\delta}(t, x)$ to the problem \eqref{BE} can be given by
$$
w_{\delta}(t, x)=w_{\delta}(x_{0}(t, x)),\quad x=x_{0}(t,x)+w_{\delta}(x_{0}(t, x))t.
$$
Then the approximate rarefaction wave $(\bar{\rho},\bar{u},\bar{\theta})(t,x)$ is
\bq\label{smooth}
\left\{\bln
 &\lambda_3(\bar{\rho}(t,x),\bar{u}_1(t,x),S^*)=w_{\delta}(t,x),\quad w_\pm=\lambda_3(\rho_\pm,u_{1\pm},S^*),\\
 &\bar{u}_1-\sqrt{15k}\bar{\rho}^{\frac13}\exp{(S^*/2)}=u_{1-}-\sqrt{15k}\rho_-^{\frac13}\exp{(S^*/2)},\quad \bar{u}_i(t,x)=0,~i=2,3,\\
&\bar{\theta}(t,x)=\frac32k\exp(S^*)(\bar{\rho})^{\frac23}(t,x),
   \quad \lim_{x\to\pm\infty}(\bar{\rho},\bar{u}_1,\bar{\theta})=(\rho_\pm,u_{1\pm},\theta_\pm).
\eln\right.
\eq
Moreover, the approximate smooth 3-rarefaction wave $(\bar{\rho},\bar{u},\bar{\theta})(t,x)$ satisfies the following Euler system
\be\label{Euler3}
\left\{ \begin{array}{l}
\bar{\rho}_t+(\bar{\rho}\bar{u}_1)_{x}=0,\\
(\bar{\rho}\bar{u}_1)_t+(\bar{\rho}\bar{u}_1^2)_{x}+\bar{p}_{x}=0,\\
(\bar{\rho}\bar{\theta})_t+(\bar{\rho}\bar{u}_1\bar{\theta})_{x}+\bar{p}\bar{u}_{1x}=0,
\end{array}\right.
\ee
where $\bar{p}=\frac23\bar{\rho}\bar{\theta}$.
Properties of $(\bar{\rho},\bar{u},\bar{\theta})(t,x)$ are given in Lemma \ref{rarefaction4}.

With the approximate rarefaction wave $(\bar{\rho},\bar{u},\bar{\theta})(t,x)$, we denote
\be\label{barM}
M_{[\bar{\rho},\bar{u},\bar{\theta}]}(v)
=\frac{\bar{\rho}(t,x)}{(2\pi R\bar{\theta}(t,x))^{3/2}}\exp\(-\frac{|v-\bar{u}(t,x)|^2}{2R\bar{\theta}(t,x)}\).
\ee
As in \cite{duan2021}, we choose the far-field data $(\rho_\pm,u_{1\pm},\theta_\pm)$ defined in \eqref{Euler-in} to be close enough to the constant state $(1,0,\frac32)$ such that the approximate smooth rarefaction wave further satisfies that
\bq\label{approximate}
\eta_0:=\sup_{t\ge 0,x\in\R}\Big\{|\bar{\rho}(t,x)-1|+|\bar{u}(t,x)|+|\bar{\theta}(t,x)-\frac32| \Big\}~\text{being~small}.
 \eq

\subsection{Notations and main results}
Before stating the main results, we list some notations.
Through this paper, $C$ denotes a generic positive constant. 
We denote $(\cdot,\cdot)$ the $L^2$ inner product in $\R_x$ or $\R_x\times\R^3_v$ with its corresponding $L^2$ norm $\|\cdot\|$,
and denote $\langle\cdot,\cdot\rangle$ the $L^2_v$ inner product in $\R^3_v$ with the corresponding $L^2_v$ norm $|\cdot|_2$.
For multi-indices $\alpha=(\alpha_{1},\alpha_{2})$ and $\beta=(\beta_1,\beta_2,\beta_3)$, we denote
$$
\p^{\alpha}_\beta=\p^{\alpha_1}_t\p^{\alpha_2}_x
 \p^{\beta_1}_{v_1}\p^{\beta_2}_{v_2}\p^{\beta_3}_{v_3},
$$
and the lengths of $\alpha,\beta$ are $|\alpha|=\alpha_1+\alpha_2$ and $|\beta|=\beta_1+\beta_2+\beta_3$, where $\alpha_1,\alpha_2,\beta_i,~i=1,2,3,$ are non-negative integers.
Furthermore, $\alpha\leq \bar{\alpha}$ means $\alpha_1\leq \bar{\alpha}_1$ and $\alpha_2\leq \bar{\alpha}_2$. Same as $\beta\leq \bar{\beta}$.
Motivated by \cite{Duan2014,Duan6,Guo2012,Wang}, we introduce the following velocity weight function
\be\label{weight1}
\omega (\alpha,\beta)(t,v):=\v^{2(l-|\alpha|-|\beta|)}e^{q(t)\frac{\v^2}{2}},\quad l\geq|\alpha|+|\beta|, \quad\v=\sqrt{1+|v|^2},
\ee
where
\be
q(t):=\frac{q_2}{(1+t)^{q_1}},\quad q_2\in\Big(0,\frac12\Big).
\ee
Note that $q_1>0$ is a small constant to be chosen later.
Denote the weighted $L^2$ norms as
\bq\label{L2n}
|\p_\beta^\alpha g|_\omega^2
:= |\omega (\alpha,\beta)\p_\beta^\alpha g|_2^2
=\int_{\R^3}\omega^2(\alpha,\beta)|\p_\beta^\alpha g |^2 dv,\quad \|\p_\beta^\alpha g\|_\omega^2=\big\||\p_\beta^\alpha g|_\omega\big\|^2_{L^2}.
 \eq
 Throughout the  paper, $\mu=\mu(v)$ denotes the global Maxwellian
\be\label{global}
\mu=M_{[1,0,\frac{3}{2}]}={(2\pi)^{-\frac32}}e^{-\frac{|v|^2}{2}},\quad v\in\mathbb{R}^3.
\ee
The corresponding Landau collision frequency is
\be\label{sigma}
\sigma_{i j}(v)=\varphi_{i j} *\mu=\int_{\R^3} \varphi_{i j}(v-u)\mu(u)d u, \quad 1 \leq i, j \leq 3,
\ee
where $\varphi_{ij}$ is a non-negative matrix defined in \eqref{varphi}.
Note that $[\sigma_{ij}(v)]_{1\leq i,j\leq3}$ is a positive-definite self-adjoint matrix.
By \eqref{sigma}, the weighted dissipation norms are given by
\bq
 \left\{\bln
&|g|_{\sigma}^2
 :=\sum_{i,j=1}^3\int_{\R^3}\[\sigma_{ij}\p_ig\p_jg+\sigma_{ij}\frac{v_i v_j}4g^2\]dv,\quad
\|g\|_{\sigma}:=\big\||g|_{\sigma}\big\|_{L^2},\\
&|\partial_\beta^\alpha g|_{\sigma,\omega}^2
 := \sum_{i,j=1}^3 \int_{\R^3}\omega^2(\alpha,\beta)
    \sigma_{ij} \Big[ \partial_i \partial_\beta^\alpha g\partial_j \partial_\beta^\alpha g
          +\frac{v_i v_j}4( \partial_\beta^\alpha g)^2 \Big] dv, \quad \|g\|_{\sigma,\omega}:=\big\||g|_{\sigma,\omega}\big\|_{L^2}.
\eln\right.
 \eq
From \cite{SG}, we have
\be\label{sigma1}
|g|_\sigma
\approx\left|\v^{\frac{\gamma+2}{2}} g\right|_2
 +\left|\v^{\frac{\gamma}{2}} \nabla_v g \cdot \frac{v}{|v|}\right|_2
 +\left|\v^{\frac{\gamma+2}{2}} \nabla_v g \times \frac{v}{|v|}\right|_2 .
\ee
For $ a\in [0,1]$, under the perturbation $(\phi,\psi,\zeta)=(\rho-\bar{\rho},u-\bar{u},\theta-\bar{\theta}) $
and $ \gt=\mu^{-\frac12}(G_1-\bar{G}_1)$, $\g=\mu^{-\frac12}G_2$,
we define the instant energy functional:
\bma \label{energy-a1}				
\widetilde{\mathcal{E}}(t)
:=&\sum_{|\alpha|\leq 1}\epsilon^{(2|\alpha|-1)a}\big\{\|\p^{\alpha}(\phi,\psi,\zeta,n,E,B)(t)\|^2_{L^2_x}
  +\|\p^{\alpha}(\gt,\g)(t)\|^2_\omega\big\}\nnm\\
&+\epsilon^{2+a}\sum_{|\alpha|=2}\{\|\p^{\alpha}(\phi,\psi,\zeta,n,E,B)(t)\|^2_{L^2_x}+\|\p^{\alpha}(\gt,\g)(t)\|^2_{\omega}\}\nnm\\
&+\sum_{|\alpha|+|\beta|\leq 2,|\beta|\geq1}\epsilon^{(2|\alpha|-1)a}\|\p_{\beta}^{\alpha}(\gt,\g)(t)\|_{\omega}^{2}. 		
\ema
With the above preparation, we are ready to present the main results as follows.

\begin{thm}\label{mt}
Let $(\bar{\rho},\bar{u},\bar{\theta})(t,x)$ be the smooth rarefaction wave defined in \eqref{smooth} and $M_{[\bar{\rho}, \bar{u}, \bar{\theta}]}(t, x,v)$ denotes the corresponding local Maxwellian.
Assume that $\eta_0:=\sup\limits_{t\ge 0,x\in\R}\big\{|\bar{\rho}(t,x)-1|+|\bar{u}(t,x)|+|\bar{\theta}(t,x)-\frac32| \big\}$
and the strength of the rarefaction wave $\sigma=|\rho_{+}-\rho_{-}|+|u_{+}-u_{-}|+|\theta_+-\theta_-|$ are small enough.
Then, the Cauchy problem of the VML system \eqref{VML2} and \eqref{initial0} with $-3\leq \gamma< -2$ in the Landau collosion operator \eqref{landau}--\eqref{varphi} with the following initial data
$$
\widetilde{\mathcal{E}}(0)
\leq \frac{\epsilon^{2-a}}{\delta^3}+\frac{\epsilon^{1+a}}{\delta^3},
 \quad a\in\Big(\frac13,1\Big)
$$
admits a unique solution  $\[F_1(t,x,v),F_2(t,x,v), E(t,x), B(t,x)\]$ for all
$t\in[0,T]$ and it satisfies
\be
\sup_{0\leq t\leq T}\widetilde{\mathcal{E}}(t)
\leq(1+T)^2\bigg(\frac{\epsilon^{2-a}}{\delta^3}+\frac{\epsilon^{1+a}}{\delta^3}\bigg).
\ee
Here,  $T>0$ can be arbitrary large as $\epsilon$ approaches zero,  $\delta>0$ is a small constant given in  \eqref{BE} and $l\geq2$ in \eqref{weight1}, and $\epsilon,\delta,T,q_1,q_2$ satisfy  for $a\in(\frac13,1),$
\bma\label{condition}
&\epsilon(1+T)\leq\epsilon_1,\quad
   (1+T)\bigg(\frac{\epsilon^{\frac a2}}{\delta^{\frac32}}+\frac{\epsilon^{\frac32a-\frac12}}{\delta^{\frac32}}\bigg)\leq\epsilon_1,
\quad (1+T)^3\bigg(\frac{\epsilon^{1-a}}{\delta^3}+\frac{\epsilon^a}{\delta^3}\bigg)\leq\epsilon_1,\nnm\\
& q_1\in\bigg[\frac{(1+T)^3}{q_2}\bigg(\frac{\epsilon^{1-\frac a2}}{\delta^{\frac32}}+\frac{\epsilon^{\frac12+\frac a2}}{\delta^{\frac32}}\bigg),1\bigg),
  \quad q_2\in\Big(0,\frac12\Big),
\ema
where $0\leq\epsilon_1\ll1$ is a constant independent of $\epsilon,\delta,T$, and $q_1,$ $q_2$ are constants defined in \eqref{weight1}.

Moreover, there exists a constant $C>0$ independent of $\epsilon,\delta$ and $T$, such that
\begin{equation}\label{mt1}
\sup_{t\in[0,T]}\bigg\{\bigg\|\frac{F_1-M_{[\bar{\rho},\bar{u},\bar{\theta}]}}{\sqrt{\mu}},
   \frac{F_2}{\sqrt{\mu}}\bigg\|^2_{ L_x^\infty L_v^2}
 +\|(E,B)\|_{L_x^\infty}^2\bigg\}
\leq C(1+T)^2\bigg(\frac{\epsilon^{2-a}}{\delta^3}+\frac{\epsilon^{1+a}}{\delta^3}\bigg),
\end{equation}
where $a\in(\frac13,1),$ and the global Maxwellian $\mu=M_{[1,0\frac32]}$.
\end{thm}	
	
\begin{thm}\label{mt2}
Under the same assumptions as Theorem \ref{mt},
by taking $\delta=(1+T)^{\frac25}\max\{\epsilon^{\frac15-\frac a5},\epsilon^{\frac{3a}5-\frac15}\}$ with $a\in(\frac13,1)$,
the unique  solution $\[F_1(t,x,v),F_2(t,x,v), E(t,x), B(t,x)\]$
of the system \eqref{VML2}-\eqref{initial0} satisfies
\be\label{mt2-1}
\sup_{t\in[h,T]}\bigg\{\bigg\|\frac{F_1-M_{[\rho^{r},u^{r},\theta^{r}]}}{\sqrt{\mu}},
   \frac{F_2}{\sqrt{\mu}}\bigg\|_{L_x^\infty L_v^2}
 +\|(E,B)\|_{L_x^\infty}\bigg\}
\leq C\frac1h(1+T)^{\frac25}\[\ln(1+T)+|\ln\epsilon|\]\big(\epsilon^{\frac15-\frac a5}+\epsilon^{\frac{3a}5-\frac15}\big),
\ee
 where $h>0$ is any constant  and can be arbitrarily small.
\end{thm}

\begin{rem}
Under the scaling $a\in(\frac13,1)$ in Theorem \ref{mt}, the reason for taking
$(1+T)^3(\frac{\epsilon^{1-a}}{\delta^3}+\frac{\epsilon^a}{\delta^3})$ and $(1+T)(\frac{\epsilon^{\frac a2}}{\delta^{\frac32}}+\frac{\epsilon^{\frac32a-\frac12}}{\delta^{\frac32}})$ small enough is to absorb the terms $(1+T)\epsilon^{-1}\mathcal{E}(\tau)\mathcal{D}(\tau)$ and $\epsilon^{a-1}\mathcal{E}^{\frac12}(\tau)\mathcal{D}(\tau)$
 in \eqref{In1} and \eqref{first2} by  the energy dissipation functional $\mathcal{D}(\tau)$.
In \cite{duan2021}, the scaling is $\frac{2}{3}\leq a\leq 1$, and
in \cite{xinzeng2010} and \cite{xLi2012}  the scaling are $a=1$ and $a=\frac23$  respectively.

For the range of $q_1$ given in \eqref{condition} is to make sure  the term $\mathcal{E}^{\frac12}(\tau)\mathcal{F}_\omega(\tau)$
in \eqref{convergence1} can be absorbed by the energy dissipation  $\frac{q_1q_2}{(1+T)^{1+q_1}}\mathcal{F}_\omega(\tau)$ as shown in  \eqref{q1q2}. And the reason for taking $q_2\in(0,\frac12)$ is due to the terms $ |\v^b \omega(0,\beta)\p_{\beta} (\mu^{-\frac12}\bar{G}_1)|_2 $ in \eqref{barG7-1}
and $|\v^b \omega(\alpha,0) \mu^{-\frac12} \p^{\alpha}M|_2$ in \eqref{norm-m-1}, in which we need
$|\v^b \omega(\alpha,\beta)\mu^{-\frac12} M^{1-\varepsilon_0}|_2\leq C$ for any $b\geq 0$ and small $\varepsilon_0>0$ where
$ q_2<\frac1{R\theta}-\frac12-\frac{\varepsilon_0}{R\theta}<\frac12$ is used.

\end{rem}

\begin{rem}{ Two key points  in long-time analysis of rarefaction waves.}

\textbf{1.} Second-order electromagnetic-fluid coupling terms in  \eqref{If-3},
$$
I^{32}_f=-\epsilon^a\int_{\R}\int_{\R^3}\frac{\dya F_2}{\sqrt{\mu}}  \underbrace{\dya (E+v\times B)}_{no~dissipation}
                       \cdot\underbrace{\Tdv M\sqrt{\mu}\big(\frac{1}{\mu}-\frac{1}{M}\big)}_{only~smallness}dvdy.
$$

$\bullet$
Within the energy framework of the VML system, the electromagnetic field itself lacks an inherent  dissipation of $\sum\limits_{|\alpha|=2}\|\p^\alpha(E,B)\|^2_{L^2_y}$
(unlike the viscous term in the Navier-Stokes equations \cite{xin1993} or the dissipative part of the collision operator in the Boltzmann equation \cite{LYYZ}). Consequently, it is impossible to directly control or ``absorb" the magnitude of this factor through the natural dissipation of the system.

$\bullet$ Within the framework of this paper, the term $\Tdv M\sqrt{\mu}\big(\frac{1}{\mu}-\frac{1}{M}\big)$ is characterized only by its small magnitude, but without  the dissipative structure and time decay rate.

Thus, we cannot control these terms unless they have  the decay rate.

\textbf{2.}  The small coefficient of dissipation induced by the Lorentz force.

Under the scaling $(\tau,y)=(\epsilon^{-a}t,\epsilon^{-a}x)$ with the small Knudsen number $\epsilon\ll 1$, a  challenge arises from the inherent mismatch between the dissipation generated by the zeroth-order Lorentz force $($see  \eqref{lower6}$)$
$$
\epsilon^{1+a}\intra  \frac{\sigma(\theta)\bar{\theta}}{R\theta^2}(E+u\times B)^2dy,
$$
and the source terms produced by the coupling of the smooth rarefaction wave with the electromagnetic field $($see  \eqref{EB0-2}-\eqref{EB0-3}$)$
\be\label{difficult-1}
-\intra \dtau\bar u_1B_3 (E_2-u_1 B_3) dy,\quad
-\intra \dtau\bar u_1 B_2 (E_3+u_1 B_2)dy.
\ee
 This coefficient $\epsilon^{1+a}$ is  weak.
This  weak dissipation is  insufficient to ``dominate" or ``absorb" the persistent energy input from these coupling terms in the energy inequality. This scaling disparity creates an  obstacle for closing the energy estimates on a long-time scale.
This ultimately precludes a global-in-time stability result within the current analytical framework.

\begin{flushleft}
    \textbf{Strategies:} A refined energy estimate
\end{flushleft}

$\bullet$
 For the first difficulty, the term denoted by  $I^{32}_f$ in \eqref{If-3} which has ``only smallness but no dissipation" can not be  ignored.
Due to the fact that  $ \epsilon(1+T) \le  1$  and $\frac\delta\sigma\leq 1$, $I^{32}_f$ can be tread as
\bma\label{difficult-0}
\epsilon^{2-2a}I^{32}_f
&=-\epsilon^{2-a}\int_{\R}\int_{\R^3}  \dya (E+v\times B)\cdot\Tdv M\dya F_2\big(\frac{1}{\mu}-\frac{1}{M}\big)dvdy\nnm\\
&\leq C\eta_0\(\epsilon^{3-a}\|\dya(E,B)\|^2_{L^2_y}
 +\epsilon^{1-a}\|\mu^{-\frac12}\dya F_2\|^2_\sigma\)\nnm\\
&\leq C\eta_0\epsilon(1+T)\epsilon^{2-2a}\frac{\epsilon^a}{\frac\delta\sigma+\epsilon^a\tau}\mathcal{E}(\tau)
  +\eta_0\epsilon^{1-a}\|\mu^{-\frac12}\dya F_2\|^2_\sigma\nnm\\
&\leq C\eta_0\frac{\epsilon^a}{\frac\delta\sigma+\epsilon^a\tau}\mathcal{E}(\tau)+ C\eta_0\mathcal{D}(\tau),
\ema
where $\eta_0>0$ is defined in \eqref{approximate}.  

 $\bullet$ As for the second difficulty (the ``weak dissipation"), by Lemma \ref{rarefaction4},
 the terms in \eqref{difficult-1} can be bounded by
 \be\label{difficult-2}
  C\sum_{|\alpha|=1}\|\p^\alpha\bar{u}_1\|_{L^\infty_y}\|(E,B)\|^2_{L^2_y}\le C\frac{\epsilon^a}{\frac{\delta}{\sigma}+\epsilon^a\tau}\mathcal{E}(\tau).
 \ee
 The coefficient of \eqref{difficult-2} has time decay $\frac{\epsilon^a}{\frac{\delta}{\sigma}+\epsilon^a\tau}$, which mirrors the slow decay of the background rarefaction wave. This transforms the slow decay property of the background wave into a time-decreasing suppression of the error term.

 $\bullet$ After incorporating the estimates for the  terms \eqref{difficult-0}-\eqref{difficult-2} into the complete energy inequality, one arrives at a standard differential inequality of the form:
 $$
 \mathcal{E}(\tau)+\int^{\tau_1}_0\mathcal{D}(\tau)d\tau
 \leq \int^{\tau_1}_0\frac{\epsilon^a}{\frac{\delta}{\sigma}+\epsilon^a\tau}\mathcal{E}(\tau)d\tau+O(\epsilon ).
 $$
 Here, $\tau_1=T/\epsilon^a$ and the dissipation functional $\mathcal{D}(\tau)$ contains the weak Lorentz force dissipation
 $\|E+u\times B\|^2_{L^2_y}$ with the small coefficient $\epsilon^{1+a}$ $($see \eqref{energy-D}$).$
 The final step for closing the estimate is the application of Gronwall's inequality in Lemma \ref{Gronwall} to obtain
 $$
 \mathcal{E}(\tau)
 \leq O(\epsilon )\exp\(\int^{\tau_1}_0\frac{\epsilon^a}{\frac{\delta}{\sigma}+\epsilon^a\tau}d\tau\)
 =O(\epsilon )\sigma\(\frac{\delta}{\sigma}+T\)
 \leq O(\epsilon )\sigma(1+T).
 $$
In summary, as the Knudsen number $\epsilon\to 0$, then $T\to+\infty$.
This implies the almost global-in-time convergence from the solutions of the VML system to the 3-rarefaction wave.
\end{rem}

The rest of this paper will be organized as follows:
In section 2, we will introduce a scaling and reformulate the system \eqref{VML2}-\eqref{initial0}.
In Section 3, we will give the lower order energy estimates of the macroscopic component $(\phi,\psi,\zeta,n)$ and the electromagnetic field $(E,B)$ by the entropy and the microscopic component $(\gt,\g)$ by using the properties of the linearized operator.
In Sections 4 and 5, we will establish the high order energy estimates and the weighted energy estimates, respectively.
In Section 6, we will establish the existence of local-in-time solutions as well as the convergence to the local Maxwellian defined by the rarefaction wave of the Euler system  as $\epsilon\to0$.
In the  Appendix, we will state certain properties for the smooth rarefaction wave and give some basic estimates frequently used in the previous sections.

\section{Reformulation and a priori estimate}
\setcounter{equation}{0}
\subsection{Scaling and reformulation}
In this section, we will introduce a scaling for the independent variable
and the perturbation and then reformulate the system. Firstly, we define the scaled independent variables
\be\label{Scaling}
y = \frac{x}{\epsilon^a},\quad \tau=\frac{t}{\epsilon^a},\quad
\text{ for }~ a\in\Big(\frac13,1\Big),\quad (t,x)\in\R^+\times\R.
\ee
Under the above scaling, for multi-index $\alpha=(\alpha_{1},\alpha_{2})$ and $\beta=(\beta_1,\beta_2,\beta_3)$, we denote
$$
\p^{\alpha}_\beta=\p^{\alpha_1}_{\tau}\p^{\alpha_2}_y
 \p^{\beta_1}_{v_1}\p^{\beta_2}_{v_2}\p^{\beta_3}_{v_3},
$$
where $\alpha_1,\alpha_2,\beta_i,~i=1,2,3,$ are non-negative integers.
Furthermore, the velocity weight function defined in \eqref{weight1} can be rewritten as
\be\label{weight}
\omega(\alpha,\beta)(\tau,v):=\v^{2(l-|\alpha|-|\beta|)}e^{q(\tau)\frac{\v^2}{2}},\quad l\geq|\alpha|+|\beta|, \quad\v=\sqrt{1+|v|^2},
\ee
where
\be
q(\tau):=\frac{q_2}{(1+\epsilon^a\tau)^{q_1}},\quad q_2\in\Big(0,\frac12\Big).
\ee
Then under the scaling \eqref{Scaling}, the system \eqref{VML2} can be rewritten as
\be
\left\{\bln \label{VML3}
&\dtau F_1+v_1\dy F_1+\epsilon^a(E+v\times B)\cdot\Tdv F_2={\epsilon^{a-1}}Q(F_1,F_1),\\
&\dtau F_2+v_1\dy F_2+\epsilon^a(E+v\times B)\cdot\Tdv F_1=\epsilon^{a-1}Q(F_2,F_1),\\
&\dtau E=\mathbb{O}\dy  B-\epsilon^a\intr F_2v dv,\quad \dy E_1=\epsilon^a\intr F_2dv,\\
&\dtau B=- \mathbb{O}\dy E,\quad B_1\equiv0,
\eln\right.
\ee
with the initial data
\be\label{initial3}
F_i(0,y,v) =F_{i0}(\epsilon^a y,v),~i=1,2,\quad (E,B)(0,y) =(E_0,B_0)(\epsilon^a y).
\ee
Secondly, we set the scaled perturbation as $(\phi,\psi,\zeta)=(\phi,\psi,\zeta)(\tau,y)$
 as
\bq
\left\{\bln
&\phi=\rho(\tau,y)-\bar{\rho}(\tau,y),\\
&\psi=(\psi_{1},\psi_{2},\psi_{3})^{T}(\tau,y) =(u_1-\bar{u}_{1},u_2,u_3)^{T}(\tau,y) ,\\
&\zeta=\theta(\tau,y)-\bar{\theta}(\tau,y).
\eln\right.
\eq
Then, subtracting \eqref{Euler3} from \eqref{ori-sys},
$(\phi,\psi,\zeta)(\tau,y)$   satisfy the following equations
\bma\label{pert-sys}
\left\{\begin{aligned}
&\dtau \phi+u_1\dy\phi+\rho\dy\psi_1+\phi\dy \bar{u}_1+\psi_1\dy \bar{\rho}=0,\\
&\rho\dtau \psi_1+\rho u_1\dy \psi_1+\rho\psi_1\dy \bar{u}_1+R\rho\dy\zeta+R{\theta}\dy\phi+R\({\theta}
        -\frac{\bar{\theta}\rho}{\bar{\rho}}\)\dy \bar{\rho}\\
&\qquad\qquad=\epsilon^{1-a}\frac4{3}\dy(\kappa_1(\theta)\dy \psi_1)+H_e^1+H^1_{f}+H^1_{n},\\
&\rho\dtau \psi_i+\rho u_1\dy \psi_i={\epsilon}^{1-a}\dy(\kappa_1(\theta)\dy \psi_i)+H_e^i +H_f^i +H^i_n,\quad i=2,3,\\
&\rho\dtau \zeta+\rho u_1\dy \zeta+\rho\psi_1\dy \bar{\theta}+R\rho\theta\dy\psi_1+R \rho\zeta\dy \bar{u}_1
   ={\epsilon}^{1-a}\dy(\kappa_2(\theta)\dy \zeta)+H_e^4+H^4_f+H^4_n,\\
&\dtau n+\dy(n u_1)-\epsilon^{1-a}\dy\(\sigma(\theta)\dy(\frac{n}{\rho})\)+\epsilon\dy\(\sigma(\theta)\frac{E_1+(u\times B)_1}{R\theta}\)
   =-\intr v_1\dy \L_M^{-1}\Lambda_2dv,
\end{aligned}\right.
\ema
where
\be\label{Hf}
\left\{\begin{aligned}
H^1_{f}&:=\epsilon^{1-a}\frac4{3}\dy(\kappa_1(\theta)\dy \bar{u}_1),\quad 
H^i_{f}:=\epsilon^{1-a}\dy(\kappa_1(\theta)\dy \bar{u}_i)=0, \quad i=2,3,\\
H^4_f&:=\epsilon^{1-a}\dy(\kappa_2(\theta)\dy \bar{\theta})
 +\frac43\epsilon^{1-a}{\kappa_1(\theta)}(\dy u_1)^2
 +\epsilon^{1-a}\sum^3_{i=2}\kappa_1(\theta)(\dy \psi_i)^2,\\
\end{aligned}\right.
\ee
\be\label{Hn}
\left\{\begin{aligned}
H^i_n&:=-\intr v_1v_i\dy ( L_M^{-1}\Lambda_1)dv+\epsilon^a\intr (v\times B)_i \L_M^{-1}\Lambda_2dv, \quad i=1,2,3,\\
H^4_n&:=\intr v_1\[(v\cdot u)-\frac12 |v|^2\]\dy ( L_M^{-1}\Lambda_1)dv
          +\epsilon^a\intr (E+u\times B)\cdot v \L_M^{-1}\Lambda_2dv,\\
\end{aligned}\right.
\ee
\be\label{He}
\left\{\begin{aligned}
H_e^i&:=\epsilon^a{n}(E+u\times B)_i
  -{\epsilon}\sigma(\theta)\dy(\frac{n}{\rho})(e_1\times B)_i
  +\frac{{\epsilon}^{1+a}\sigma(\theta)}{R\theta}[(E+u\times B)\times B]_i, \\
H^4_e&:=\epsilon^{1+a}\sigma(\theta)\frac{|E+u\times B|^2}{R\theta}-\epsilon{\sigma(\theta)}\dy(\frac{n}{\rho})(E+u\times B)_1.
\end{aligned}\right.
\ee
For the electronic-magnetic fields $(E,B)$, from \eqref{EB-inv}, we have
\be\label{EB-inv-s}
\left\{\begin{aligned}
& \dtau E_1= -\epsilon^{a}nu_1+\epsilon\sigma(\theta)\dy(\frac{n}{\rho})
   -\epsilon^{1+a}\sigma(\theta)\frac{E_1+(u\times B)_1}{R\theta}
   -\epsilon^{a}\intr v_1\L_M^{-1}\Lambda_2dv, \\
&\dtau E_2=-\dy B_3-\epsilon^{a}nu_2
   -\epsilon^{1+a}\sigma(\theta)\frac{E_2+(u\times B)_2}{R\theta}
   -\epsilon^{a}\intr v_2\L_M^{-1}\Lambda_2dv, \\
&\dtau E_3=\dy B_2-\epsilon^{a}nu_3
   -\epsilon^{1+a}\sigma(\theta)\frac{E_3+(u\times B)_3}{R\theta}
   -\epsilon^{a}\intr v_3\L_M^{-1}\Lambda_2dv.
\end{aligned}\right.
\ee
Next, for the non-fluid part $(G_1,G_2)$, from \eqref{G0-1}-\eqref{G0-2}, we have
\bma
\dtau G_1&+P_1(v_1\dy M)+P_1(v_1\dy G_1)+\epsilon^{a}P_1[(E+v\times B)\cdot\Tdv G_2]
  =\epsilon^{a-1}L_MG_1+\epsilon^{a-1}Q(G_1,G_1),\label{G1}\\
\dtau G_2&+P_r(v_1\dy G_2 )+\frac{n}{\rho}P_r(\dtau M+v_1\dy M)+\epsilon^a  (E+v\times B)\cdot\Tdv G_1 \nonumber\\
  &+\(\dy(\frac{n}{\rho})P_r(v_1M)-\epsilon^a\frac{E+u\times B}{R\theta}\cdot P_r(vM)\)
  =\epsilon^{a-1}\L_MG_2 +\epsilon^{a-1}Q(F_2,G_1).\label{G2}
\ema
Thus,
\bmas
G_1&=\epsilon^{1-a} L_M^{-1} P_1(v_1\dy M)+L_M^{-1}\Lambda_1,\\
G_2&= \epsilon^{1-a}\dy(\frac{n}{\rho})\L_M^{-1}P_r(v_1M)
     -\epsilon\frac{E+u\times B}{R\theta}\cdot \L_M^{-1}P_r(vM)+\L_M^{-1}\Lambda_2,
\emas
where
\bma
\Lambda_1=&\,\epsilon^{1-a}[\dtau G_1+ P_1(v_1\dy G_1)]-Q(G_1,G_1)+\epsilon P_1[(E+v\times B)\cdot\Tdv G_2],\label{lambda1}\\
\Lambda_2=&\,\epsilon^{1-a}[\dtau G_2+P_r(v_1\dy G_2)]
      -Q(F_2,G_1)-\epsilon^{1-a}\frac{n}{\rho}P_r(\dtau M+v_1\dy M)\nnm\\
      &+\epsilon (E+v\times B)\cdot\Tdv G_1.\label{lambda2}
\ema

Since the slow decay rate with respect to the  Knudsen number of the term $\|\dy(\bar{u},\bar{\theta})\|^{2}$ in the term $P_1(v_1 \dy M)$ in \eqref{G1}, we have to subtract from $G_1(t,x,v)$ the term $\bar{G}_1$:
\be\label{barG1-1}
\bar{G}_1=\frac{1}{R\theta}\epsilon^{1-a}L^{-1}_{M}{P_{1}\[v_{1}\(\frac{|v-u|^{2}}{2\theta}\dy\bar{\theta}+v\cdot\nabla_y\bar{u}\)M\]}.
\ee
Noticing that the correction function $\bar{G}_1$ was first introduced in \cite{LYYZ} for the stability of the rarefaction wave to the 1D Boltzmann equation.

Inspired by \cite{Guo1}, we denote
\be\label{decompose}
\gt:=\frac{1}{\sqrt{\mu}}(G_1-\bar{G}_1),\quad \g:=\frac{1}{\sqrt{\mu}} {G}_2,
\ee
which satisfy
\bma\label{g1}
&\quad\p_\tau\gt+v_1\p_y\gt+\epsilon^a\frac{1}{\sqrt{\mu}}(E+v\times B)\cdot\nabla_v(\sqrt{\mu}\g)-{\epsilon}^{a-1} L \gt \nnm\\
&=\frac{1}{\sqrt{\mu}}P_0\big[v_1\sqrt{\mu}\p_y\gt+\epsilon^a(E+v\times B)\cdot\nabla_v(\sqrt{\mu}\g)\big]\nnm\\		
&\quad-\frac1{\sqrt{\mu}}P_1\bigg\{v_1\(\frac{|v-u|^2\dy\zeta}{2R\theta^2}+\frac{(v-u)\cdot\dy\psi}{R\theta}\)M\bigg\}+R_{g1}+R_{\Gamma1},
\ema
and
\bma
\label{g2}
&\quad\dtau \g+v_1\dy \g+\epsilon^a \frac{1}{\sqrt{\mu}}\big[(E+v\times B)\cdot\Tdv (\sqrt{\mu}\gt)\big]-\epsilon^{a-1} \L \g\nonumber\\
&= \frac{1}{\sqrt{\mu}}P_d (v_1\sqrt{\mu}\dy\g) -\epsilon^a \frac{1}{\sqrt{\mu}}(E+v\times B)\cdot\Tdv \bar{G}_1 +R_{g2}+R_{\Gamma2},
\ema	
where
\be
\left\{\bln  \label{remainder}
R_{\Gamma1}:=&\,{\epsilon}^{a-1}\bigg[\Gamma\(\frac{M-\mu}{\sqrt{\mu}},\gt\)
  +\Gamma\(\gt,\frac{M-\mu}{\sqrt{\mu}}\)+\Gamma\(\frac {G_1}{\sqrt{\mu}},\frac {G_1}{\sqrt{\mu}}\)\bigg],\\		
R_{\Gamma2}:=&\,\epsilon^{a-1}\bigg[\Gamma\(\g,\frac{M-\mu}{\sqrt{\mu}}\)+\Gamma\(\frac{F_2}{\sqrt{\mu}},\frac{G_1}{\sqrt{\mu}}\)\bigg],\\
R_{g1}:=&\,-\frac{P_1(v_1\p_y\bar{G}_1)}{\sqrt{\mu}}-\frac{\p_\tau\bar{G}_1}{\sqrt{\mu}},\\
R_{g2}:=&\,-\frac{1}{\sqrt{\mu}}\(\dy(\frac{n}{\rho})P_r(v_1M)-\epsilon^a\frac{E+u\times B}{R\theta}\cdot P_r(vM)
                                       +\frac{n}{\rho}P_r (\dtau M+v_1\dy M)\).
\eln\right.
\ee	
Here, for any well-defined $\h_1,\h_2$,
\be\label{L12}
\Gamma(\h_1,\h_2):=\frac1{\sqrt\mu}Q(\sqrt\mu\h_1,\sqrt\mu\h_2),\quad
L \h_1:=\Gamma(\sqrt\mu,\h_1)+\Gamma(\h_1,\sqrt\mu),\quad
\L \h_1:=\Gamma(\h_1,\sqrt\mu),
\ee
then a direct calculation yields
\begin{align*}		
&\frac1{\sqrt{\mu}}L_M(\sqrt{\mu}\h_1)=L \h_1+\Gamma\(\h_1,\frac{M-\mu}{\sqrt{\mu}}\)+\Gamma\(\frac{M-\mu}{\sqrt{\mu}},\h_1\),\\
&\frac1{\sqrt{\mu}}\L_M(\sqrt{\mu}\h_1)=\L \h_1+\Gamma\(\h_1,\frac{M-\mu}{\sqrt{\mu}}\).
		\end{align*}

\subsection{A Priori estimates}
In this subsection, we will introduce the a priori estimates for the system \eqref{VML3} around the smooth rarefaction wave.
For this, we define the following instant energy functional:
\bma \label{energy-a}				
\mathcal{E}(\tau)
:=&\,\sum_{|\alpha|\leq 1}\big\{\|\p^{\alpha}(\phi,\psi,\zeta,n,{E},{B})(\tau)\|^2_{L^2_y}+\|\p^{\alpha}(\gt,\g)(\tau)\|^2_{\omega}\big\}\nnm\\
&+\epsilon^{2-2a}\sum_{|\alpha|=2}\{\|\p^{\alpha}(\phi,\psi,\zeta,n,E,B)(\tau)\|^2_{L^2_y}+\|\p^{\alpha}(\gt,\g)(\tau)\|_{\omega}^{2}\}\nnm\\
 &+\sum_{|\alpha|+|\beta|\leq 2,|\beta|\geq1}\|\p_{\beta}^{\alpha}(\gt,\g)(\tau)\|_{\omega}^{2},  		
\ema
and the corresponding energy dissipation functional:
\bma\label{energy-D}
\mathcal{D}(\tau)
:=&\,
  \epsilon^{1+a}\sum_{|\alpha|=1}\|\p^\alpha(E,B)(\tau)\|_{L^2_y}^2
  +\epsilon^{1+a}\|(n,E+u\times B)(\tau)\|^2_{L^2_y}\nnm\\
&+\epsilon^{1-a}\sum_{1\leq|\alpha|\leq 2}\|\p^\alpha(\phi,\psi,\zeta,n)(\tau)\|^2_{L^2_y}
 +\epsilon^{1-a}\sum_{|\alpha|=2}\|\p^\alpha (\gt,\g)(\tau)\|_{\sigma,\omega}^2\nnm\\
&+\epsilon^{a-1}\sum_{|\alpha|\leq 1}\|\p^{\alpha}(\gt,\g)(\tau)\|_{\sigma,\omega}^{2}
 +\epsilon^{a-1}\sum_{|\alpha|+|\beta|\leq2,|\beta|\geq1}\|\p_{\beta}^{\alpha}(\gt,\g)(\tau)\|_{\sigma,\omega}^2.
\ema
Note that $\mathcal{E}(\tau)=\widetilde{\mathcal{E}}(\eps^{-a}t)$. Furthermore, we also denote
\bma\label{energy-F}
\mathcal{F}_{\omega}(\tau)
:=&\,\epsilon^a\sum_{|\alpha|\leq 1}\|\v\p^{\alpha}(\gt,\g)(\tau)\|_{\omega}^{2}
  +\epsilon^{2-a}\sum_{|\alpha|=2}\|\v\mu^{-\frac12}\p^{\alpha}(F_1,F_2)(\tau)\|_\omega^{2}\nnm\\
&+\epsilon^a\sum_{|\alpha|+|\beta|\leq 2,|\beta|\geq1}\|\v\p^{\alpha}_{\beta}(\gt,\g)(\tau)\|_\omega^{2}.
\ema

Next, we will prove the following a prior estimate.
\begin{prop}(A priori estimate)\label{priori3}
Suppose that the Cauchy problem \eqref{VML3}-\eqref{initial3} has a solution $(F_1,F_2,E,B)$.
For  $\tau_1=\frac{T}{\epsilon^a}>0,~T<+\infty$ and $l\geq2$ in \eqref{weight},
$\epsilon,T,q_1,q_2$ satisfy \eqref{condition}, if
\be\label{priori1}
\sup_{0\leq\tau\leq \tau_1}\mathcal{E}(\tau)
 \leq(1+T)^2\bigg(\frac{\epsilon^{2-a}}{\delta^3}+\frac{\epsilon^{1+a}}{\delta^3}\bigg),~~a\in\Big(\frac13,1\Big),
\ee
then
\be\label{priori4}
\sup_{0\leq\tau\leq\tau_1}\mathcal{E}(\tau)
 \leq \frac12(1+T)^2\bigg(\frac{\epsilon^{2-a}}{\delta^3}+\frac{\epsilon^{1+a}}{\delta^3}\bigg),
\ee
where $\delta>0$ is a small constant given by \eqref{BE}.
\end{prop}

Moreover, using the a priori estimate \eqref{priori1} and \eqref{approximate}, we have
\be
|\rho(t,x)-1|\leq|\rho(t,x)-\bar{\rho}(t,x)|+|\bar{\rho}(t,x)-1|\leq|\phi|+\eta_0.\nnm
\ee
Then, $(\rho,u,\theta)(t,x)$ satisfies for all $(t,x)\in\R^+\times\R$,
\bq\label{local}
\eta_1:=\sup_{ t\ge 0,x\in\R}\Big\{|\rho(t,x)-1|+|u(t,x)|+|\theta(t,x)-\frac32| \Big\}
 \leq\sup_{t\ge 0,x\in\R}|(\phi,\psi,\zeta)|+\eta_0.
 \eq
\section{Lower order energy estimates}\setcounter{equation}{0}
In this section, we will consider the energy estimate of zero-th order derivative in the following lemma.
\begin{lem}\label{lolem}
Under the same assumptions as in Proposition \ref{priori3} and by letting $\sigma\frac{\epsilon^a}{\delta}$ small enough
and $\frac{\delta}{\sigma}\leq 1$, we have the following estimate:
\bma
&\frac{d}{d\tau}\int_{\R}(\eta+E^2+2E\cdot\bar u \times B+B^2+n^2)dy+\frac{d}{d\tau}\|(\gt,\g)\|^2
  +\epsilon^{1-a}\frac{d}{d\tau}\int_{\R}\psi_1\p_y\phi dy\nnm\\
& 
  + \epsilon^{a-1}\|(\gt,\g)\|^2_\sigma
  +\epsilon^{1-a}\sum_{|\alpha|=1}\|\p^\alpha(\phi,\psi,\zeta,n)\|^2_{L^2_y}
  +\epsilon^{1+a}\|(n,E+u\times B)\|^2_{L^2_y}\nnm\\
\leq&\,\lambda\mathcal{D}(\tau)
   +C\lambda^{-1}\(\eta^2_1+\mathcal{E}^{\frac12}(\tau)+\sigma^2\frac{\epsilon^{2a}}{\delta^2}+(1+T)\epsilon^{-1}\mathcal{E}(\tau)\)\mathcal{D}(\tau)
   +C\mathcal{E}^{\frac12}(\tau)\mathcal{F}_\omega(\tau)\nnm\\
& +C(1+T)\frac{\epsilon^2}{\delta(\frac{\delta}{\sigma}+\epsilon^a\tau)^2}
  +C\frac{\epsilon^a}{\frac{\delta}{\sigma}+\epsilon^a\tau}\mathcal{E}(\tau)
  +C\lambda^{-1}\epsilon^{1-a}\sum_{|\alpha|=1}\|\p^\alpha(\gt,\g)\|^2_\sigma,
\ema
where $\lambda>0$ is a small constant, and $\eta_1$ is defined in \eqref{local} and $\eta$ is entropy function defined by \eqref{entropy}.
\end{lem}

The proof of Lemma \ref{lolem} is divided in the following four subsections.
\subsection{Estimates of the macro-components and electromagnetic field}
\begin{lem}\label{lower1-1}
Under the same assumptions as  in Proposition \ref{priori3} and by letting $\sigma\frac{\epsilon}{\delta}$ small enough
and $\frac{\delta}{\sigma}\leq 1$, we have
\bma
&\frac{d}{d\tau}\int_{\R}(\eta+E^2+2E\cdot\bar u \times B+B^2)dy
 +\epsilon^{1-a}\|\dy(\psi,\zeta)\|^2_{L^2_y}
 +\epsilon^{1+a}\|(n,E+u\times B)\|^2_{L^2_y}\nnm\\
& +\|\sqrt{\dy\bar{u}_1}(\phi,\psi_1,\zeta)\|^2_{L^2_y}
  +\sum^3_{i=2} \|\sqrt{\dy\bar u_1}(E_i^2+B^2_i)\|^2_{L^2_y}\nnm\\
\leq&\lambda\mathcal{D}(\tau)
   +C\lambda^{-1}\(\sigma^2\frac{\epsilon^2}{\delta^2}+\mathcal{E}^{\frac12}(\tau)+(1+T)\epsilon\mathcal{E}(\tau)\)\mathcal{D}(\tau)
   +C\frac{\epsilon^a}{\frac{\delta}{\sigma}+\epsilon^a\tau}\mathcal{E}(\tau)\nnm\\
&+C(1+T)\frac{\epsilon^2}{\delta(\frac{\delta}{\sigma}+\epsilon^a\tau)^2}
 +C\lambda^{-1}\epsilon^{1-a}\sum_{|\alpha|=1}\|\p^\alpha(\gt,\g)\|^2_\sigma,\nnm
\ema
where $\lambda>0$ is a small constant and $\eta$ is entropy function defined by \eqref{entropy}.
\end{lem}
\begin{proof}
\begin{flushleft}
\textbf{Step 1. } The entropy-entropy flux pair.
\end{flushleft}
In the following, we will consider the lower order energy estimates of the macroscopic component $(\phi,\psi,\zeta)$
by the entropy-entropy flux.
Define the  entropy-entropy flux pair $(\eta,q)$ by
\bma
\eta
&=R\rho\bar{\theta}\Phi\(\frac{\bar{\rho}}{\rho}\)
  +\rho\bar{\theta}\Phi\(\frac{\theta}{\bar{\theta}}\)
  +\rho\frac{|u-\bar{u}|^2}{2}, \label{entropy}
  \\
q&=u_1\eta+R(u_1-\bar{u}_1)(\rho\theta-\bar{\rho}\bar{\theta})
=u_1\eta  +R(\phi\psi_1\bar{\theta}+\rho\psi_1\zeta),\nnm
\ema
where $\Phi(s)$ is a convex function defined by
$$
\Phi(s)=s-\ln s-1.
$$
We observe from $\Phi(1)=\Phi'(1)=0,\ \Phi''(1)=1$ and $\Phi'''(s)=-2/s^3$ that
$$
\Phi\(\frac{\bar{\rho}}{\rho}\)=\frac12\frac{\phi^2}{\rho^2}+O(1)\phi^3,\quad
\Phi\(\frac{\theta}{\bar{\theta}}\)=\frac12\frac{\zeta^2}{\bar{\theta}^2}+O(1)\zeta^3.
$$
Firstly, by direct computations, we have
\bma\label{ee-1}
&\dtau\bigg(R\rho\bar{\theta}\Phi(\frac{\bar{\rho}}{\rho})+\rho\bar{\theta}\Phi(\frac{\theta}{\bar{\theta}})\bigg)
  +\dy\(R\rho u_1\bar{\theta}\Phi(\frac{\bar{\rho}}{\rho})+\rho u_1\bar{\theta}\Phi(\frac{\theta}{\bar{\theta}})\)\nnm\\
=&\rho\(\dtau\bar{\theta}+u_1\p_y\bar{\theta}\)\(R\Phi(\frac{\bar{\rho}}{\rho})+\Phi(\frac{\theta}{\bar{\theta}})\)
  +R\frac{\bar{\theta}}{\rho\bar{\rho}}\(\bar{\rho}\p_y\bar{u}_1-\psi_1\p_y\bar{\rho}\)\phi^2			\nonumber\\
   &+R\frac{\bar{\theta}}{\rho}\phi\(\dtau\phi+ u_1\p_y\phi\)-\frac{\rho}{\bar{\theta}\theta}\(\dtau\bar{\theta}+u_1\p_y\bar{\theta}\)\zeta^2+\frac{\rho}{\theta}\zeta\(\dtau\zeta+u_1\dy\zeta\).
\ema
Next, multiplying $\eqref{pert-sys}_{1}$ by $\frac{R\bar{\theta}}{\rho}\phi$, 
$\eqref{pert-sys}_5$ by $\frac{\zeta}{\theta}$, we have
\bma
R\frac{\bar{\theta}}{\rho}\phi\(\dtau\phi+u_1\p_y\phi\)
=&-R\bar{\theta}\p_y\psi_1\phi-R\frac{\bar{\theta}}{\rho}\phi\psi_1\p_y\bar{\rho}
  -R\frac{\bar{\theta}}{\rho}\p_y\bar{u}_1\phi^2,\label{ee1}\\
\frac{\rho}{\theta}\zeta\(\dtau\zeta+u_1\p_y\zeta\)
=&-R\rho\zeta\p_y\psi_1-\frac{\rho}{\theta}\zeta\psi_1\p_y\bar{\theta}-R\frac{\rho}{\theta} \p_y\bar{u}_1\zeta^2
  +\epsilon^{1-a}\p_y(\kappa_2(\theta)\p_y\zeta)\frac{\zeta}{\theta}\nnm\\
&+\frac{\zeta}{\theta}(H_e^4+H^4_f+H^4_n).\label{ee2}
\ema
Multiplying \eqref{pert-sys}$_{i+1}$ by $\psi_i$, $i=1,2,3$ and using $\eqref{macro-eq-nond}_1$,
the summation of the resulting equations yields
\bma
&\quad\dtau\(\rho\frac{\psi^2}{2}\)+\dy\(\rho u_1\frac{\psi^2}{2}\)+\rho\dy \bar{u}_1 \psi^2_1
 +R\rho\psi_1\dy\zeta+R\theta\psi_1\dy\phi+R\psi_1\(\theta-\frac{\bar{\theta}\rho}{\rho} \)\p_y\bar{\rho}\nnm\\
&=\epsilon^{1-a}\bigg(\frac13\p_y(\kappa_1(\theta)\dy\psi_1)\psi_1+\sum^3_{i=1}\p_y(\kappa_1(\theta)\dy \psi_i)\psi_i\bigg)
 +\sum_{i=1}^3\psi_i\(H^i_e+H^i_f+H^i_n\).
\ema
Moreover, by $\eqref{Euler3}_4$, we have
\be \label{ee3}
\dtau\bar{\theta}+u_{1}\p_{y}\bar{\theta}
=-R\bar{\theta}\p_{y}\bar{u}_1+\psi_1\p_{y}\bar{\theta}.
\ee

Combining \eqref{ee-1}-\eqref{ee3}, we obtain
\bma\label{lower2}
&\dtau\(R\rho\bar{\theta}\Phi(\frac{\bar{\rho}}{\rho})+\rho\frac{\psi^2}{2}+\rho\bar{\theta}\Phi(\frac{\theta}{\bar{\theta}})\)
+\p_y\(R\rho u_1\bar{\theta}\Phi(\frac{\bar{\rho}}{\rho})+\rho u_1\frac{\psi^2}{2}
			 +\rho u_1\bar{\theta}\Phi(\frac{\theta}{\bar{\theta}})\)\nnm\\
=&- \(R\rho\bar{\theta} \p_y\bar{u}_1\(R\Phi(\frac{\bar{\rho}}{\rho})+\Phi(\frac{\theta}{\bar{\theta}})\)
              +\rho \p_y\bar{u}_1 \psi^2_1\)
  -\rho\psi_1 \p_y\bar{\theta}\(R\ln\frac{\bar{\rho}}{\rho}+\ln\frac{\theta}{\bar{\theta}}\) \nnm\\
&-\big[R\p_y(\rho\zeta\psi_1)+R\p_y(\bar{\theta}\phi\psi_1)\big]+ \sum_{i=1}^3\psi_i\(H^i_e+H^i_f+H^i_n\)+\frac{\zeta}{\theta}\(H^4_e+H^4_f+H^4_n\) \nnm\\
&+ \epsilon^{1-a}\p_y(\kappa_2(\theta)\p_y\zeta)\frac{\zeta}{\theta}
  +\epsilon^{1-a}\bigg(\frac13\p_y(\kappa_1(\theta)\dy\psi_1)\psi_1+\sum^3_{i=1}\p_y(\kappa_1(\theta)\dy \psi_i)\psi_i\bigg) ,
\ema
where we have used the fact that
\bma
&\rho\psi_1 \p_y\bar{\theta}\(R\Phi(\frac{\bar{\rho}}{\rho})+\Phi(\frac{\theta}{\bar{\theta}})\)
 -\frac{\rho}{\bar{\theta}}\zeta\psi_1 \p_y\bar{\theta}
 -R\bar{\theta} \p_y(\psi_1\phi)=-R\dy(\bar{\theta}\phi\psi_1)-\rho\psi_1 \p_y\bar{\theta}\(R\ln\frac{\bar{\rho}}{\rho}+\ln\frac{\theta}{\bar{\theta}}\).\nnm
\ema
Furthermore, there exists a positive constant $C>0$ such that (cf. \cite{duan2021,WangY,LYYZ})
$$
\(R\rho\bar{\theta} \p_y\bar{u}_1\(R\Phi(\frac{\bar{\rho}}{\rho})+\Phi(\frac{\theta}{\bar{\theta}})\)
   +\rho\p_y\bar{u}_1 \psi^2_1\)+\rho\psi_1 \p_y\bar{\theta}\(R\ln\frac{\bar{\rho}}{\rho}+\ln\frac{\theta}{\bar{\theta}}\)\geq C^{-1}\p_y\bar{u}_{1}(\phi^2+\psi_1^2+\zeta^2).
$$
Then integrating \eqref{lower2} over $\R_y$, we can obtain
\bma\label{lower}
&\frac{d}{d\tau}\int_{\R}\eta(\tau, y)dy+\|\sqrt{\p_y\bar{u}_1}(\phi,\psi_1,\zeta)\|^2_{L^2_y}\nnm\\
\leq&\int_\R\[\epsilon^{1-a}\p_y(\kappa_2(\theta)\p_y\zeta)\frac{\zeta}{\theta}
  +\epsilon^{1-a}\bigg(\frac13\p_y(\kappa_1(\theta)\dy\psi_1)\psi_1+\sum^3_{i=1}\p_y(\kappa_1(\theta)\dy \psi_i)\psi_i\bigg)\]dy\nnm\\
&+\int_\R\[\sum_{i=1}^3\psi_iH^i_f+\frac{\zeta}{\theta}H^4_f+H^4_n\frac{\zeta}{\theta}\]dy
 -\int_\R\psi\cdot\intr v_1v\dy ( L_M^{-1}\Lambda_1)dvdy\nnm\\
&+\epsilon^a\int_\R\psi\cdot\intr (v\times B)\L_M^{-1}\Lambda_2dvdx
 +\epsilon^a\int_\R n\psi\cdot(E+u\times B)dy\nnm\\
&+{\epsilon}^{1+a}\int_{\R}\[-\frac{\sigma(\theta)}{R\theta}(E+u\times B)\cdot(\psi\times B)
                           +\sigma(\theta)\frac{|E+u\times B|^2}{R\theta}\frac{\zeta}{\theta}\]dy\nnm\\
&-{\epsilon}\int_{\R}\[\sigma(\theta)\dy(\frac{n}{\rho})(e_1\times B)\cdot\psi
                    +{\sigma(\theta)}\dy(\frac{n}{\rho})(E+u\times B)_1\frac{\zeta}{\theta}\]dy,
\ema
where $H^i_f~(i=1,2,3,4)$ and $H^4_n$ are defined in \eqref{Hf}, \eqref{Hn} and $\Lambda_1,\Lambda_2$ are defined in \eqref{lambda1}-\eqref{lambda1}.
\begin{flushleft}
\textbf{Step 2. } Low order energy estimates on  $(E,B)$.
\end{flushleft}
We first study $E_1$. Taking the inner product between \eqref{EB-inv-s}$_1$  and $E_1$, then using the integration by parts
and the fact that $\dy E_1=\epsilon^a n$, we have
\bma
&\quad\frac12\frac{d}{d\tau}\intra E_1^2dy-\epsilon\intra  \sigma(\theta) \dy(\frac{n}{\rho})E_1dy
  +\epsilon^{1+a}\intra \frac{\sigma(\theta)}{R\theta}(E_1+(u\times B)_1)E_1dy\nnm\\
&=-\int_{\R}\epsilon^au_1nE_1 dy-\epsilon^a\int_{\R}E_1\int_{\R^3} v_1 \L_M^{-1}\Lambda_2 dvdy\nnm\\
&=\frac12\intr \dy \bar{u}_1 E_1^2dy-\epsilon^a\intra n\psi_1 E_1dy-\epsilon^a\int_{\R}E_1\int_{\R^3} v_1\L_M^{-1}\Lambda_2 dvdy.
\ema
And for $(E_2,B_3)$, we take the inner product between \eqref{EB-inv-s}$_2$ and $E_2-\bar{u}_1B_3$ to have
\bmas
&\quad\frac12 \frac{d}{d\tau}\intra  E_2^2dy-\int_{\R}\p_\tau E_2\bar{u}_1 B_3 dy+\intra \dy B_3 (E_2-\bar{u}_1B_3)dy \\
&\quad+\epsilon^a \intra n\psi_2(E_2-\bar u_1B_3)dy +\epsilon^{1+a}\intra \frac{\sigma(\theta)}{R\theta}(E_2+(u\times B)_2)(E_2-\bar u_1 B_3)dy\\
&= -\epsilon^a\int_{\R}(E_2-\bar u_1 B_3)\int_{\R^3} v_2 \L_M^{-1}\Lambda_2 dvdy.
\emas
By using $\dtau B_3=-\dy E_2$ and the fact that
\bmas
\int_{\R}\dy B_3 E_2dy=&-\int_{\R}B_3 \dy E_2dy=\int_{\R}B_3\p_\tau B_3dy=\frac12\frac{d}{d\tau}\int_{\R}B^2_3dy,\\
-\intra \dtau E_2\bar u_1 B_3dy
 =&-\frac{d}{d\tau}\intra E_2\bar u_1 B_3dy+\intra \dtau\bar u_1 E_2 B_3dy+\frac12\intra \dy \bar u_1 E_2^2dy,
\emas
we have
\bma\label{EB0-2}
&\frac12 \frac{d}{d\tau}\intra (E_2^2-2E_2\bar u_1 B_3+B_3^2)dy+\frac12\intra \dy\bar u_1 (E_2^2+B^3_3)dy\nnm \\
&+\epsilon^a \intra n\psi_2(E_2-\bar u_1B_3)dy
 +\epsilon^{1+a}\intra  \frac{\sigma(\theta)}{R\theta}(E_2+(u\times B)_2)(E_2-\bar u_1 B_3)dy\nnm\\
=&-\intra \dtau\bar u_1 (E_2-u_1 B_3) B_3dy-\intra \dtau\bar u_1 u_1 B_3^2dy
  -\epsilon^a\int_{\R}(E_2-\bar u_1 B_3)\int_{\R^3} v_2 \L_M^{-1}\Lambda_2 dvdy.
\ema
Similarly, taking the inner product between \eqref{EB-inv-s}$_3$ and $E_3+\bar u_1 B_2$, we have
\bma\label{EB0-3}
&\frac12 \frac{d}{d\tau} \intra (E_3^2+2E_3\bar u_1 B_2+B_2^2)dy+\frac12\intra \dy\bar u_1 (E_3^2+B^2_2)dy\nnm \\
&+ \epsilon^a\intra n\psi_3(E_3+\bar u_1B_2)dy
 +\epsilon^{1+a}\intra  \frac{\sigma(\theta)}{R\theta}(E_3+(u\times B)_3)(E_3+\bar u_1 B_2)dy\nnm\\
=&-\intra \dtau\bar u_1 (E_3+u_1 B_2) B_2dy+\intra \dtau\bar u_1 u_1 B_2^2dy
 -\epsilon^a\int_{\R}(E_3+\bar u_1 B_2)\int_{\R^3} v_3\L_M^{-1} \Lambda_2 dvdy.
\ema
Note that $B_1=0$ and $\bar{u}\times B=(0,-\bar{u}_1B_3,\bar{u}_1B_2).$
Thus,
\bma\label{EUB}
&\frac12\frac{d}{d\tau} \intra (E^2+2E\cdot\bar u \times B+B^2)dy+\frac12\sum^3_{i=2}\intra \dy\bar u_1 (E_i^2+B^2_i)dy \nnm\\
&+\epsilon^{1+a}\intra \frac{\sigma(\theta)}{\rho}|n|^2  dy
  +\epsilon^{1+a}\intra  \frac{\sigma(\theta)}{R\theta}(E+u\times B)\cdot(E+\bar u \times B)dy\nnm\\
\leq& C\sum_{|\alpha|=1}\|\p^\alpha\bar{u}_1\|_{L^\infty_y}\|(E,B)\|^2_{L^2_y}
  -\epsilon\int_{\R}\dy\sigma(\theta) (\frac{n}{\rho})E_1 dy\nnm\\
&-\epsilon^a\intra n\psi\cdot(E+\psi\times B)dy
 -\epsilon^a\int_{\R} (E+\bar u \times B)\cdot\int_{\R^3} v \L_M^{-1}\Lambda_2 dvdy,
\ema
where we have used the fact that
\bma
 \psi\cdot(E+\bar u \times B)&=\psi\cdot(E+ u \times B)-\psi\cdot(\psi\times B)=\psi\cdot(E+ u \times B),\nnm\\
-\epsilon\sigma(\theta) \dy (\frac{n}{\rho})E_1
&=-\epsilon\dy \(\sigma(\theta) (\frac{n}{\rho})E_1\)
 +\epsilon\dy\sigma(\theta)(\frac{n}{\rho})E_1
 +\epsilon^{1+a}\frac{n^2\sigma(\theta)}{\rho}.
\ema
By summing of \eqref{EUB} and \eqref{lower}, and using
\bma
&-(E+\bar u \times B)\cdot v+\psi\cdot(v\times B)=-(E+u\times B)\cdot v,\\
&\frac{\sigma(\theta)}{R\theta}(E+u\times B)\cdot(E+\bar u \times B)
+\frac{\sigma(\theta)}{R\theta}(E+u\times B)\cdot(\psi\times B)\nnm\\
&\qquad-\sigma(\theta)\frac{|E+u\times B|^2}{R\theta}\frac{\zeta}{\theta}
 =\frac{\sigma(\theta)\bar{\theta}}{R\theta^2}|E+u\times B|^2,
\ema
we have
\bma\label{lower6}
&\frac{d}{d\tau}\int_{\R}\eta+\frac12(E^2+2E\cdot\bar u \times B+B^2)dy
 +\|\sqrt{\p_y\bar{u}_1}(\phi,\psi_1,\zeta)\|^2_{L^2_y}
 +\frac12\sum^3_{i=2}\intra \dy\bar u_1 (E_i^2+B^2_i)dy\nnm\\
&+\epsilon^{1+a}\intra \frac{\sigma(\theta)}{\rho}|n|^2  dy
 +\epsilon^{1+a}\intra  \frac{\sigma(\theta)\bar{\theta}}{R\theta^2}(E+u\times B)^2dy\nnm\\
\leq&\int_{\R}\[\epsilon^{1-a}\p_y(\kappa_2(\theta)\p_y\zeta)\frac{\zeta}{\theta}
  +\epsilon^{1-a}\bigg(\frac13\p_y(\kappa_1(\theta)\dy\psi_1)\psi_1+\sum^3_{i=1}\p_y(\kappa_1(\theta)\dy \psi_i)\psi_i\bigg)\]dy\nnm\\
&-\int_{\R}\[\epsilon\sigma(\theta)\dy(\frac{n}{\rho})(e_1\times B)\cdot\psi
 +\epsilon\dy\sigma(\theta) (\frac{n}{\rho})E_1
 +\epsilon\sigma(\theta)\dy(\frac{n}{\rho})(E+u\times B)_1\frac{\zeta}{\theta}\]dy\nnm\\
&+\int_{\R}\[\psi\cdot\intr v_1v\dy ( L_M^{-1}\Lambda_1)dv
   +H^4_n\frac{\zeta}{\theta}
 +\epsilon^a(E+u\times B)\cdot\int_{\R^3} v \L_M^{-1}\Lambda_2 dv\]dy\nnm\\
&+\int_{\R}\(\sum_{i=1}^3\psi_iH^i_f+\frac{\zeta}{\theta}H^4_f\)dy
 +C\sum_{|\alpha|=1}\|\p^\alpha\bar{u}_1\|_{L^\infty_y}\|(E,B)\|^2_{L^2_y}
:=\sum^5_{i=1}I_i,
\ema
where $H^i_f~(i=1,2,3,4)$, $H^4_n$ are defined in \eqref{Hf}-\eqref{Hn}
and $\Lambda_1,\Lambda_2$ are defined in \eqref{lambda1}-\eqref{lambda1}.

Firstly, for $I_5$, by Lemma \ref{rarefaction4}, we have
\be
I_5=C\sum_{|\alpha|=1}\|\p^\alpha\bar{u}_1\|_{L^\infty_y}\|(E,B)\|^2_{L^2_y}\leq C\frac{\epsilon^a}{\frac{\delta}{\sigma}+\epsilon^a\tau}\mathcal{E}(\tau).
\ee
Next, we will estimate $I_1$ which is the dissipation term. Note that
\bma
&\quad\int_{\R}\epsilon^{1-a}\(\frac13\p_y(\kappa_1(\theta)\dy\psi_1)\psi_1+\sum^3_{i=1}\p_y(\kappa_1(\theta)\dy \psi_i)\psi_i\)dy\nnm\\
&=-\epsilon^{1-a}\int_{\R}\frac13\kappa_1(\theta)|\dy\psi_1|^2+\sum^3_{i=1}\kappa_1(\theta)|\dy \psi_i|^2dy.
\ema
Then by Lemma \ref{rarefaction4}, we have
\bma
\int_{\R}\epsilon^{1-a}\p_y(\kappa_2(\theta)\p_y\zeta)\frac{\zeta}{\theta}dy
=&-\epsilon^{1-a}\int_{\R}\frac{\kappa_2(\theta)}{\theta}|\p_y\zeta|^2dy
  +\epsilon^{1-a}\int_{\R}\frac{\kappa_2(\theta)\zeta\p_y\zeta\p_y\theta}{\theta^2} dy\nnm\\
\leq&-\epsilon^{1-a}\int_{\R}\frac{\kappa_2(\theta)}{\theta}|\p_y\zeta|^2dy
  +C\epsilon^{1-a}\int_{\R}\zeta|\p_y\zeta|^2+\zeta\p_y\zeta\p_y\bar{\theta} dy\nnm\\
\leq&-\epsilon^{1-a}\int_{\R}\frac{\kappa_2(\theta)}{\theta}|\p_y\zeta|^2dy
  +C\mathcal{E}^{\frac12}(\tau)\mathcal{D}(\tau)\nnm\\
&+\lambda\epsilon^{1-a}\|\p_y\zeta\|^2_{L^2_y}
  +C\lambda^{-1}\frac{\epsilon^{1+a}}{(\frac{\delta}{\sigma}+\epsilon^a\tau)^2}\|\zeta\|^2_{L^2_y},
\ema
where $\lambda>0$ is a small constant.

For $I_2$, by Lemma \ref{rarefaction4}, we have for $\frac{\delta}{\sigma}\leq 1$,
\bma
I_2
\leq&\,C\int_{\R}\epsilon|\p_y(n,\phi)||(E,B)||(\psi,\zeta)|
  +\epsilon|\p_y\bar{\rho}||n||(E,B)||(\psi,\zeta)|+\epsilon|\p_y\theta||n||E|dy\nnm\\
\leq&\,C(\frac{\delta}{\sigma}+\epsilon^a\tau)\epsilon^{2-a}\int_{\R}|\p_y(n,\phi)|^2|(E,B)|^2dy
  +C\frac{\epsilon^a}{(\frac{\delta}{\sigma}+\epsilon^a\tau)}\|(\psi,\zeta)\|^2_{L^2_y}\nnm\\
 &+C\epsilon^{1+a}\int_{\R}|n|^2|(E,B)|^2dy
  +C\epsilon^{1-a}\int_{\R}|\p_y(\bar{\rho},\bar{\theta})|^2|(\psi,\zeta)|^2dy\nnm\\
 &+\lambda\epsilon^{1-a}\|\p_y\zeta\|^2_{L^2_y}+C\lambda^{-1}\epsilon^{1+a}\int_{\R}|n|^2|E|^2dy
  +\lambda\epsilon^{1+a}\|n\|^2_{L^2_y}+C\lambda^{-1}\epsilon^{1-a}\int_{\R}|\dy\bar{\theta}|^2|E|^2dy\nnm\\
\leq&\,{ C(1+T)\epsilon\mathcal{E}(\tau)\mathcal{D}(\tau)}
  +C\lambda^{-1}\mathcal{E}(\tau)\mathcal{D}(\tau)
  +\lambda\mathcal{D}(\tau)\nnm\\
&+C\frac{\epsilon^a}{\frac{\delta}{\sigma}+\epsilon^a\tau}\|(\psi,\zeta)\|^2_{L^2_y}
  +C\lambda^{-1}\frac{\epsilon^{1+a}}{(\frac{\delta}{\sigma}+\epsilon^a\tau)^2}\|(\psi,\zeta,E)\|^2_{L^2_y}.
\ema
Next, we will estimate $I_4$ in \eqref{lower6}. From the definition of $H^i_f~(i=1,2,3,4)$ in \eqref{Hf} and Lemma \ref{rarefaction4},
we have for $\frac{\delta}{\sigma}\leq 1$,
\bma
I_4\leq&\,C\epsilon^{1-a}\int_{\R}\(|\p^2_y(\bar{u}_1,\bar{\theta})|+|\p_y(\bar{u}_1,\bar{\theta})|^2\)|(\psi_1,\zeta)|
 +|\p_y(\bar{u}_1,\bar{\theta})||(\psi_1,\zeta)||\p_y\zeta|
 +|\zeta||\p_y\psi|^2dy\nnm\\
\leq&\,C\frac{\epsilon^a}{(\frac{\delta}{\sigma}+\epsilon^a\tau)}\|(\psi_1,\zeta)\|^2_{L^2_y}
  +C(\frac{\delta}{\sigma}+\epsilon^a\tau)\epsilon^{2-3a}\int_{\R}|\p^2_y(\bar{u}_1,\bar{\theta})|^2+|\p_y(\bar{u}_1,\bar{\theta})|^4dy\nnm\\
&+\lambda\epsilon^{1-a}\|\p_y\zeta\|^2_{L^2_y}
 +C\lambda^{-1}\epsilon^{1-a}\int_{\R}|\p_y(\bar{u}_1,\bar{\theta})|^2|(\psi_1,\zeta)|^2dy
 +C\lambda^{-1}\mathcal{E}^{\frac12}(\tau)\mathcal{D}(\tau)\nnm\\
\leq&\,C\frac{\epsilon^a}{(\frac{\delta}{\sigma}+\epsilon^a\tau)}\|(\psi_1,\zeta)\|^2_{L^2_y}
 +C(1+T)\frac{\epsilon^2}{\delta(\frac{\delta}{\sigma}+\epsilon^a\tau)^2}
 +\lambda\epsilon^{1-a}\|\p_y\zeta\|^2_{L^2_y}\nnm\\
&+C\lambda^{-1}\frac{\epsilon^{1+a}}{(\frac{\delta}{\sigma}+\epsilon^a\tau)^2}\|(\psi_1,\zeta)\|^2_{L^2_y}
 +C\lambda^{-1}\mathcal{E}^{\frac12}(\tau)\mathcal{D}(\tau).
\ema
Then, we need to estimate the remaining term $I_3$ in \eqref{lower6}.
For the terms involving $H^4_n$, by using the self-adjoint property of $L_M^{-1}$ and the Burnett functions in \eqref{Burnett1}-\eqref{Burnett2}, we have
\bma\label{lambda1.1-1}
&\int_{\R}\frac{\zeta}{\theta}\int_{\R^3}\(\frac{1}{2}v_{1}|v|^{2}-v_{1}u\cdot v\)\p_y\(L_{M}^{-1}\Lambda_1\) dvdy\nnm\\
=&-\int_{\R} \p_y\(\frac{\zeta}{\theta}\)\int_{\R^3}\(\frac{1}{2}v_{1}|v|^{2}-v_{1}u\cdot v\)L_{M}^{-1}\Lambda_1 dvdy
 +\sum^3_{i=1}\int_{\R}\p_y u_i\frac{\zeta}{\theta}\int_{\R^3}v_{1}v_iL_{M}^{-1}\Lambda_1 dvdy\nnm\\
=&-\int_{\R} \p_y\(\frac{\zeta}{\theta}\)(R\theta)^{\frac32}\int_{\R^3}\A_{1}\left(\frac{v-u}{\sqrt{R\theta}}\right)\frac{\Lambda_1}{M}dv
 +\sum^3_{i=1}\int_{\R}\p_y u_i\frac{\zeta}{\theta}R\theta\int_{\R^3}\B_{1i}\(\frac{v-u}{\sqrt{R\theta}}\)\frac{\Lambda_1}{M}dvdy\nnm\\
:=&\,I_{\A}+I_{\B}.
\ema
In the following we only estimate $I_{\A}$, the term $I_{\B}$ can be treated  in the same way.
From the definition of $\Lambda_1$ in \eqref{lambda1}, we have
\bma\label{lambda1.1-2}
I_{\A}
=&-\int_{\R} \p_y\(\frac{\zeta}{\theta}\)(R\theta)^{\frac32}\int_{\R^3}\A_{1}\left(\frac{v-u}{\sqrt{R\theta}}\right)
  \frac{\epsilon^{1-a}\dtau G_1}{M}dvdy\nnm\\
&-\int_{\R} \p_y\(\frac{\zeta}{\theta}\)(R\theta)^{\frac32}\int_{\R^3}\A_{1}\left(\frac{v-u}{\sqrt{R\theta}}\right)
  \frac{\epsilon^{1-a}P_1(v_1\dy G_1)}{M}dvdy\nnm\\
&+\int_{\R} \p_y\(\frac{\zeta}{\theta}\)(R\theta)^{\frac32}\int_{\R^3}\A_{1}\left(\frac{v-u}{\sqrt{R\theta}}\right)
  \frac{Q(G_1,G_1)}{M}dvdy\nnm\\
&-\int_{\R} \p_y\(\frac{\zeta}{\theta}\)(R\theta)^{\frac32}\int_{\R^3}\A_{1}\left(\frac{v-u}{\sqrt{R\theta}}\right)
  \frac{\epsilon P_1[(E+v\times B)\cdot\Tdv G_2]}{M}dvdy\nnm\\
:=&\,I^1_{\A}+I^2_{\A}+I^3_{\A}+I^4_{\A}.
\ema
For $I^1_{\A}$, recalling that $G_1=\bar{G}_1+\sqrt{\mu}\gt$, by  Lemma  \ref{Burnett3} and \eqref{Gb2}, we have
\bma
I^1_{\A}
&= \int_{\R} \p_y\(\frac{\zeta}{\theta}\)
 (R\theta)^{\frac32}\int_{\R^3}\A_{1}\left(\frac{v-u}{\sqrt{R\theta}}\right)\frac{\epsilon^{1-a}(\dtau\bar{G}_1+\sqrt{\mu}\dtau\gt)}{M}dvdy\nnm\\
&\leq \lambda\epsilon^{1-a}(\|\p_y\zeta\|^2_{L^2_y}+\|\zeta \p_y\theta\|^2_{L^2_y})
  +C\lambda^{-1}\epsilon^{1-a}\|\mu^{-\frac12}\dtau\bar{G}_1\|^2
  +C\lambda^{-1}\epsilon^{1-a}\|\v^{-\frac12}\dtau\gt\|^2\nnm\\
&\leq \lambda\mathcal{D}(\tau)
   +\lambda\mathcal{E}(\tau)\mathcal{D}(\tau)
   +\lambda\frac{\epsilon^{1+a}}{(\frac{\delta}{\sigma}+\epsilon^a\tau)^2}\|\zeta\|^2_{L^2_y}
   +C\lambda^{-1}\epsilon^{1-a}\|\p_\tau\gt\|^2_\sigma\nnm\\
&\quad+C\lambda^{-1}\frac{\epsilon^3}{\delta(\frac{\delta}{\sigma}+\epsilon^a\tau)^2}
   +C\lambda^{-1}\sigma^2\frac{\epsilon^2}{\delta^2}\mathcal{D}(\tau).
\ema
Similarly, we have
\bma
I^2_{\A}=&\,\int_{\R} \p_y\(\frac{\zeta}{\theta}\)
 (R\theta)^{\frac32}\int_{\R^3}\A_{1}\left(\frac{v-u}{\sqrt{R\theta}}\right)\frac{\epsilon^{1-a}P_1(\p_y G_1)}{M}dvdy\nnm\\
\leq&\,\lambda\mathcal{D}(\tau)
   +\lambda\mathcal{E}(\tau)\mathcal{D}(\tau)
   +\lambda\frac{\epsilon^{1+a}}{(\frac{\delta}{\sigma}+\epsilon^a\tau)^2}\|\zeta\|^2_{L^2_y}
   +C\lambda^{-1}\epsilon^{1-a}\|\p_y\gt\|^2_\sigma\nnm\\
&+C\lambda^{-1}\frac{\epsilon^3}{\delta(\frac{\delta}{\sigma}+\epsilon^a\tau)^2}
   +C\lambda^{-1}\sigma^2\frac{\epsilon^2}{\delta^2}\mathcal{D}(\tau).
\ema
And then by using \eqref{decompose}, \eqref{fast decay} and Lemma \ref{QGG}, we have
\bma
I^3_{\A}&=\int_{\R} \p_y\(\frac{\zeta}{\theta}\)
 (R\theta)^{\frac32}\int_{\R^3}\A_{1}\(\frac{v-u}{\sqrt{R\theta}}\)\frac{Q(G_1,G_1)}{M}dvdy\nnm\\
&\leq \lambda\mathcal{D}(\tau)
   +\lambda\frac{\epsilon^{1+a}}{(\frac{\delta}{\sigma}+\epsilon^a\tau)^2}\|\zeta\|^2_{L^2_y}
   +C\lambda^{-1}\frac{\epsilon^3}{\delta(\frac{\delta}{\sigma}+\epsilon^a\tau)^2}
   +C\lambda^{-1}\(\sigma^2\frac{\epsilon^2}{\delta^2}+\mathcal{E}(\tau)\)\mathcal{D}(\tau).
\ema
For $I^4_{\A}$, by \eqref{decompose}, \eqref{fast decay}, using the fact $\nabla_v\cdot(v\times B)=0$
and $\p_{v_i}M=-\frac{v_i-u_i}{R\theta}M(i=1,2,3)$, we have
\bma\label{lambda1.1-5}
I^4_{\A}&= \int_{\R} \p_y\(\frac{\zeta}{\theta}\)
 (R\theta)^{\frac32}\int_{\R^3}\A_{1}\left(\frac{v-u}{\sqrt{R\theta}}\right)\frac{\epsilon P_1[(E+v\times B)\cdot\Tdv G_2]}{M}dvdy\nnm\\
&= \epsilon \int_{\R} \p_y\(\frac{\zeta}{\theta}\)
 (R\theta)^{\frac32}\int_{\R^3}\A_{1}\left(\frac{v-u}{\sqrt{R\theta}}\right)\frac{(E+v\times B)\cdot\Tdv G_2}{M}dvdy\nnm\\
 &\quad+\epsilon \int_{\R} \p_y\(\frac{\zeta}{\theta}\)
 (R\theta)^{\frac32}\int_{\R^3}\frac{1}{M}\A_{1}\left(\frac{v-u}{\sqrt{R\theta}}\right) \chi_i dv
   \sum^4_{i=1}\intr (E+v\times B)G_2\cdot \Tdv \(\frac{\chi_i}{M}\) dvdy\nnm\\
&\leq \lambda\mathcal{D}(\tau)
   +\lambda\mathcal{E}(\tau)\mathcal{D}(\tau)
   +\lambda\frac{\epsilon^{1+a}}{(\frac{\delta}{\sigma}+\epsilon^a\tau)^2}\|\zeta\|^2_{L^2_y}
   +C\lambda^{-1}\epsilon^2\mathcal{E}(\tau)\mathcal{D}(\tau).
\ema
Combining \eqref{lambda1.1-2}-\eqref{lambda1.1-5}, we have
\bma
I_{\A}=&\,\int_{\R} \p_y\(\frac{\zeta}{\theta}\)(R\theta)^{\frac32}\int_{\R^3}\A_1\(\frac{v-u}{\sqrt{R\theta}}\)\frac{\Lambda_1}{M}dvdy\nnm\\
\leq&\,\lambda\mathcal{D}(\tau)
   +\lambda\frac{\epsilon^{1+a}}{(\delta+\epsilon^a\tau)^2}\|\zeta\|^2_{L^2_y}
   +C\lambda^{-1}\frac{\epsilon^3}{\delta(\frac{\delta}{\sigma}+\epsilon^a\tau)^2}\nnm\\
  &+C\lambda^{-1}\(\sigma^2\frac{\epsilon^2}{\delta^2}+\mathcal{E}(\tau)\)\mathcal{D}(\tau)
   +C\lambda^{-1}\epsilon^{1-a}\sum_{|\alpha|=1}\|\p^\alpha\gt\|^2_\sigma.
\ema

Similar to the estimation on $I_{\A}$, we can estimate the second term $I_{\B}$ in \eqref{lambda1.1-1} as follows:
\bma
I_{\B}=&\,\sum^3_{i=1}\int_{\R}\p_y u_i\frac{\zeta}{\theta}R\theta\int_{\R^3}\B_{1i}\(\frac{v-u}{\sqrt{R\theta}}\)\frac{\Lambda_1}{M}dvdy\nnm\\
\leq&\,\lambda\frac{\epsilon^{1+a}}{(\frac{\delta}{\sigma}+\epsilon^a\tau)^2}\|\zeta\|^2_{L^2_y}
   +C\lambda^{-1}\frac{\epsilon^3}{\delta(\frac{\delta}{\sigma}+\epsilon^a\tau)^2}\nnm\\
 &+C\lambda^{-1}\(\sigma^2\frac{\epsilon^2}{\delta^2}+\mathcal{E}(\tau)\)\mathcal{D}(\tau)
   +C\lambda^{-1}\epsilon^{1-a}\sum_{|\alpha|=1}\|\p^\alpha\gt\|^2_\sigma.\nnm
\ema
Then we have
\bma
&\int_{\R}\frac{\zeta}{\theta}\int_{\R^3}\(\frac{1}{2}v_{1}|v|^{2}-v_{1}u\cdot v\)\p_y\(L_{M}^{-1}\Lambda_1\) dvdy\nnm\\
\leq&\,\lambda\mathcal{D}(\tau)
   +\lambda\frac{\epsilon^{1+a}}{(\frac{\delta}{\sigma}+\epsilon^a\tau)^2}\|\zeta\|^2_{L^2_y}
   +C\lambda^{-1}\frac{\epsilon^3}{\delta(\frac{\delta}{\sigma}+\epsilon^a\tau)^2}\nnm\\
  &+C\lambda^{-1}\(\sigma^2\frac{\epsilon^2}{\delta^2}+\mathcal{E}(\tau)\)\mathcal{D}(\tau)
   +C\lambda^{-1}\epsilon^{1-a}\sum_{|\alpha|=1}\|\p^\alpha\gt\|^2_\sigma.\nnm
\ema
Similarly, we have
\bma
&-\sum^3_{i=1}\int_{\R}\psi_i\intr v_1v_i\dy (L_M^{-1}\Lambda_1)dvdy
 =\sum^3_{i=1}\int_{\R}\dy\psi_i\intr v_1v_i L_M^{-1}\Lambda_1 dvdy\nnm\\
\leq&\,\lambda\mathcal{D}(\tau)
   +C\lambda^{-1}\frac{\epsilon^3}{\delta(\frac{\delta}{\sigma}+\epsilon^a\tau)^2}
   +C\lambda^{-1}\(\sigma^2\frac{\epsilon^2}{\delta^2}+\mathcal{E}(\tau)\)\mathcal{D}(\tau)
  +C\lambda^{-1}\epsilon^{1-a}\sum_{|\alpha|=1}\|\p^\alpha\gt\|^2_\sigma.\nnm
\ema
Then by \eqref{lambda2}, Lemma \ref{Burnett5}, Lemma \ref{Gamma}, we have
\bma
&\quad\epsilon^a\int_{\R} (E+u\times B)\cdot\int_{\R^3} v \L_M^{-1}\Lambda_2 dvdy
 +\epsilon^a\int_{\R}\frac{\zeta}{\theta}\intr [E+(u\times B)]\cdot v \L_M^{-1}\Lambda_2dvdy\nnm\\
&\leq\lambda\mathcal{D}(\tau)
 +C\lambda^{-1}\epsilon^{1-a}\sum_{|\alpha|=1}\|\p^\alpha\g\|^2_\sigma
 +C\lambda^{-1}\mathcal{E}(\tau)\mathcal{D}(\tau)
 +C\lambda^{-1}\frac{\epsilon^{1+a}}{(\frac{\delta}{\sigma}+\epsilon^a\tau)^2}\|(n,E,B)\|^2_{L^2_y}.
\ema
Combining the estimates for $I_i$, $i=1,2,...,5$ and by choosing $\sigma\frac{\epsilon}{\delta}$ small enough,
we complete the proof of Lemma \ref{lower1-1}.
\end{proof}

\subsection{Low order energy estimates on $n$}
\begin{lem}\label{n1}
Under the same assumptions as in Proposition \ref{priori3} and letting $\sigma\frac{\epsilon}{\delta}$ small enough,
$\frac{\delta}{\sigma}\leq 1$, we have
\bma
&\frac{d}{d\tau}\|n\|^2_{L^2_y} +\|\sqrt{\dy\bar{u}_1}n\|^2_{L^2_y}
 +\epsilon^{1-a}\|\dy n\|^2_{L^2_y}
 +\epsilon^{1+a}\|n\|^2_{L^2_y}\nnm\\
\leq&\,\lambda\mathcal{D}(\tau)
 +C\lambda^{-1}\epsilon^{1-a}\sum_{|\alpha|=1}\|\p^\alpha\g\|^2_\sigma
 +C\frac{\epsilon^a}{(\frac{\delta}{\sigma}+\epsilon^a\tau)}\mathcal{E}(\tau)
 +C(1+T)\epsilon^{-1}\mathcal{E}(\tau)\mathcal{D}(\tau),\nnm
\ema
where $\lambda>0$ is a small constant.
\end{lem}
\begin{proof}
For $n$, multiplying \eqref{pert-sys}$_6$ by $n$, we have
\bma
&\frac12 \frac{d}{d\tau}\intra n^2dy+\frac12\intr \dy \bar{u}_1 n^2dy
 +\epsilon^{1-a}\intra \frac{\sigma(\theta)}{\rho}|\dy n|^2dy
 +\epsilon^{1+a}\int_{\R}\frac{\sigma(\theta)}{R\theta}|n|^2dy\nnm\\
=&-\frac12\intra \dy \psi_1 n^2dy+\epsilon^{1-a}\intra  \frac{\sigma(\theta)}{\rho^2}n\dy  n\dy\rho dy
 -\epsilon\int_{\R}\dy\(\frac{\sigma(\theta)}{R\theta}\)\(E_1+(u\times B)_1\)ndy\nnm\\
&+\epsilon\int_{\R}\frac{\sigma(\theta)}{R\theta}n\dy(\psi_2B_3-\psi_3B_2)dy
 +\int_{\R}\dy n\int_{\R^3} v_1\L_M^{-1}\Lambda_2 dvdy\nnm\\
=&\,I_{n}^1+I_{n}^2+I_{n}^3+I_{n}^4+I^5_n,
\ema
where we have used the fact $\dy E_1=\epsilon^a n$ and
\bma
&-\epsilon\dy\(\sigma(\theta)\frac{E_1+(u\times B)_1}{R\theta}\)n\nnm\\
=&-\epsilon\dy\(\frac{\sigma(\theta)}{R\theta}\)\(E_1+(u\times B)_1\)n
 -\epsilon^{1+a}\frac{\sigma(\theta)}{R\theta}n^2
 +\epsilon\frac{\sigma(\theta)}{R\theta}n\dy(\psi_2B_3-\psi_3B_2).\nnm
\ema
By Lemma \ref{rarefaction4}, it holds that for $\frac{\delta}{\sigma}\leq 1$,
\bma
I_n^1
&\leq C\|n\|_{L^\infty_y}\|\dy\psi_1\|_{L^2_y}\|n\|_{L^2_y}
\leq C\frac{\epsilon^a}{(\frac{\delta}{\sigma}+\epsilon^a\tau)}\|n\|^2_{L^2_y}
     +C(\frac{\delta}{\sigma}+\epsilon^a\tau)\epsilon^{-a}\|n\|^2_{L^\infty_y}\epsilon^{a-1}\epsilon^{1-a}\|\dy\psi_1\|^2_{L^2_y}\nnm\\
&\leq C\frac{\epsilon^a}{(\frac{\delta}{\sigma}+\epsilon^a\tau)}\|n\|^2_{L^2_y}
 + C(1+T)\epsilon^{-1}\mathcal{E}(\tau)\mathcal{D}(\tau),\label{In1}\\
I_n^2
&\leq C\epsilon^{1-a}\int_{\R}(|\dy\bar\rho|+|\dy \phi|)|n||\dy n|dy
 \leq\lambda\epsilon^{1-a}\|\dy n\|^2_{L^2_y}+C\lambda^{-1}\epsilon^{1-a}\int_{\R}|\dy\bar\rho|^2|n|^2+|\dy\phi|^2|n|^2dy\nnm\\
&\leq\lambda\mathcal{D}(\tau)
 +C\lambda^{-1}\frac{\epsilon^{1+a}}{(\frac{\delta}{\sigma}+\epsilon^a\tau)^2}\|n\|^2_{L^2_y}
 +C\lambda^{-1}\mathcal{E}(\tau)\mathcal{D}(\tau),\nnm\\
I_n^3
&=-\epsilon\intra \dy(\frac{\sigma(\theta)}{R\theta})(E_1+(u\times B)_1)ndy
 \leq C\epsilon\int_{\R}(|\dy \bar{\theta}|+|\dy \zeta|)|E+u\times B||n|dy\nnm\\
&\leq\lambda\mathcal{D}(\tau)
   +C\lambda^{-1}\frac{\epsilon^{1+a}}{(\frac{\delta}{\sigma}+\epsilon^a\tau)^2}\|n\|^2_{L^2_y}
   +C\lambda^{-1}\mathcal{E}(\tau)\mathcal{D}(\tau),\nnm\\
I^4_n
&\leq C\epsilon\int_{\R}|n||B||\p_y\psi|+|n||\p_yB||\psi|dy\nnm\\
&\leq \lambda\epsilon^{1-a}\|\p_y\psi\|^2_{L^2_y}+C\lambda^{-1}\epsilon^{1+a}\int_{\R}|n|^2|B|^2dy
 +C\frac{\epsilon^a}{(\frac{\delta}{\sigma}+\epsilon^a\tau)}\|n\|^2_{L^2_y}
 +C(1+T)\epsilon^{2-a}\int_{\R}|\psi|^2|\p_yB|^2dy\nnm\\
&\leq\lambda\mathcal{D}(\tau)+C\lambda^{-1}\mathcal{E}(\tau)\mathcal{D}(\tau)
 +C\frac{\epsilon^a}{(\frac{\delta}{\sigma}+\epsilon^a\tau)}\mathcal{E}(\tau)
 +C(1+T)\epsilon^{1-a}\mathcal{E}(\tau)\mathcal{D}(\tau).\nnm
\ema
And then  by \eqref{lambda2} and Lemma \ref{Burnett5}, we have
$$
I_n^5
\leq  \lambda\epsilon^{1-a}\|\dy n\|^2_{L^2_y}
 +C\lambda^{-1}\epsilon^{1-a}\sum_{|\alpha|=1}\|\p^\alpha\g\|^2_\sigma
 +C\lambda^{-1}\mathcal{E}(\tau)\mathcal{D}(\tau)
 +C\lambda^{-1}\frac{\epsilon^{1+a}}{(\frac{\delta}{\sigma}+\epsilon^a\tau)^2}\|(n,E,B)\|^2.
$$
Thus combining all estimates for $I_n^i$, $i=1,2,...,5$ and letting $\sigma\frac{\epsilon}{\delta}$ small enough, we have proved Lemma \ref{n1}.
\end{proof}
\subsection{Estimates on $ \dy\phi $ and  $ \dtau(\phi,\psi,\zeta,n) $}

In this subsection, we estimate the terms involving time-derivatives and $\phi_y$.
Under the scaling transformation \eqref{Scaling}, subtracting \eqref{Euler3} from \eqref{macro-eq-nond},
$(\phi,\psi,\zeta)(\tau,y)$ satisfy the following system
\bma\label{pert-sys-n}
\left\{\begin{aligned}
&\dtau \phi+u_1\dy\phi+\rho\dy\psi_1+\phi\dy \bar{u}_1+\psi_1\dy \bar{\rho}=0,\\
&\dtau \psi_1+ u_1\dy \psi_1+\psi_1\dy \bar{u}_1+R\dy\zeta+R\frac{\theta}{\rho}\dy\phi
   +R\(\frac{\theta}{\rho}-\frac{\bar{\theta}}{\bar{\rho}}\)\dy \bar{\rho}\\
&\qquad\qquad=\epsilon^a\frac{n}{\rho}(E_1+(u\times B)_1)-\frac{1}{\rho}\intr v_1^2\dy G_1dv
   +\frac{\epsilon^a}{\rho}\intr (v\times B)_1 G_2dv,\\
&\dtau \psi_i+ u_1\dy \psi_i=\epsilon^a\frac{n}{\rho}(E_i+(u\times B)_i)-\frac{1}{\rho}\intr v_1v_i\dy G_1dv
   +\frac{\epsilon^a}{\rho}\intr (v\times B)_i G_2dv,~ i=2,3,\\
&\dtau \zeta+ u_1\dy \zeta+\psi_1\dy \bar{\theta}+R\theta\dy\psi_1+R \zeta\dy \bar{u}_1\\
&\qquad\qquad=\frac{1}{\rho}\intr v_1\[(v\cdot u)-\frac12 |v|^2\]\dy G_1dv+\frac{\epsilon^a}{\rho}\intr [E+(u\times B)]\cdot vG_2dv.
\end{aligned}\right.
\ema

\begin{flushleft}
\textbf{Step 1.} Estimates on $\epsilon^{1-a}\|\dy\phi\|^2_{L^2_y}$.
\end{flushleft}
 We will prove the following lemma.
\begin{lem}\label{pphi1}
Under the same assumptions as in Proposition \ref{priori3}, we have
\bmas
&\quad \epsilon^{1-a}\|\dy\phi\|^2_{L^2_y}
 +\epsilon^{1-a}\frac{d}{d\tau}\int_{\R}\psi_1\p_y\phi dy\nnm\\
&\leq C\epsilon^{1-a}\|\p_y(\psi_1,\zeta)\|^2_{L^2_y}
 +C\frac{\epsilon^{1+a}}{(\frac{\delta}{\sigma}+\epsilon^a\tau)^2}\|(\phi,\psi_1,\zeta)\|^2_{L^2_y}
 +C\frac{\epsilon^3}{\delta(\frac{\delta}{\sigma}+\epsilon^a\tau)^2}\nnm\\
&\quad+ C\(\mathcal{E}(\tau)+\sigma^2\frac{\epsilon^2}{\delta^2}\)\mathcal{D}(\tau)
 +C\epsilon^{1-a}\|\dy\gt\|^2_\sigma.
\emas
\end{lem}
\begin{proof}
Multiplying \eqref{pert-sys-n}$_2$ by $\epsilon^{1-a}\p_y \phi$ and integrating over $\R_y$, we have
\bma\label{phi1}
&\epsilon^{1-a}\int_{\R}\p_\tau \psi_1\p_y\phi dy+\epsilon^{1-a}\int_{\R}\frac{R\theta}{\rho}|\p_y\phi|^2dy\nnm\\
=&-\epsilon^{1-a}\int_{\R}\(u_1\dy\psi_1+R \p_y\zeta+\psi_1 \dy\bar{u}_1
  +R\(\frac{\theta}{\rho}-\frac{\bar{\theta}}{\bar{\rho}}\) \dy \bar{\rho}\)\p_y\phi dy\nnm\\
&+\epsilon\int_{\R}\frac{n}{\rho}(E_1+(u\times B)_1)\p_y\phi dy
 -\epsilon^{1-a}\int_{\R}\frac{1}{\rho}\p_y\phi\int_{\R^3}v_1^2 \dy G_1 dvdy\nnm\\
& +\epsilon\int_{\R}\frac{1}{\rho}\dy\phi\intr (v\times B)_1 G_2dvdy.
\ema
By \eqref{pert-sys-n}$_1$ and integration by parts, we have
\bma
\int_{\R}\p_\tau \psi_1\p_y\phi dy
=&\frac{d}{d\tau}\int_{\R}\psi_1\p_y\phi dy-\int_{\R}\psi_1\dy\p_\tau\phi dy
 =\frac{d}{d\tau}\int_{\R}\psi_1\p_y\phi dy+\int_{\R}\dy\psi_1\p_\tau\phi dy\nnm\\
=&\frac{d}{d\tau}\int_{\R}\psi_1\p_y\phi dy
  -\int_{\R}\p_y\psi_1\(u_1\dy\phi+\rho\dy\psi_1+\psi_1\dy\bar{\rho}+\phi\dy\bar{u}_1\)dy.\nnm
\ema
Substituting the above equality into \eqref{phi1}, we have
\bma
&\epsilon^{1-a}\frac{d}{d\tau}\int_{\R}\psi_1\p_y\phi dy
 +\epsilon^{1-a}\int_{\R}\frac{R\theta}{\rho}|\p_y\phi|^2dy\nnm\\
=&\,\epsilon^{1-a}\int_{\R}\p_y\psi_1\(u_1\dy\phi+\rho\dy\psi_1+\psi_1\dy\bar{\rho}+\phi\dy\bar{u}_1\)dy\nnm\\
&-\epsilon^{1-a}\int_{\R}\(u_1\dy\psi_1+R \p_y\zeta+\psi_1 \dy\bar{u}_1
  +R\(\frac{\theta}{\rho}-\frac{\bar{\theta}}{\bar{\rho}}\) \dy \bar{\rho}\)\p_y\phi dy\nnm\\
&+\epsilon\int_{\R}\frac{n}{\rho}(E_1+(u\times B)_1)\p_y\phi dy
 -\epsilon^{1-a}\int_{\R}\frac{1}{\rho}\p_y\phi\int_{\R^3}v_1^2 \dy G_1 dvdy\nnm\\
&+\epsilon\int_{\R}\frac{1}{\rho}\dy\phi\intr (v\times B)_1 G_2dvdy
:=\sum_{i=1}^{5}I_{\phi}^{i}.\nnm
\ema
For $I_{\phi}^1,I_{\phi}^2$, by Lemma \ref{rarefaction4} and using the fact that  $(\dy\psi_1u_1\dy\phi-u_1\dy\psi_1\dy\phi)=0$, we have
\bma
I_{\phi}^1+I_{\phi}^2
&\leq C\epsilon^{1-a}\int_{\R}|\p_y\zeta||\p_y\phi|+|\p_y\psi_1|^2
    +|\p_y(\phi,\psi_1)||(\phi,\psi_1,\zeta)||\p_y(\bar{\rho},\bar{u}_1)|dy\nnm\\
&\leq \lambda\epsilon^{1-a}\|\p_y\phi\|^2_{L^2_y}
    +C\lambda^{-1}\epsilon^{1-a}\|\p_y(\psi_1,\zeta)\|^2_{L^2_y}
    +C\lambda^{-1}\frac{\epsilon^{1+a}}{(\frac{\delta}{\sigma}+\epsilon^a\tau)^2}\|(\phi,\psi_1,\zeta)\|^2_{L^2_y}.\nnm
\ema		
For $I_{\phi}^3$, we have,
\bma
\int_{\R}\epsilon\frac{n}{\rho}(E_1+(u\times B)_1)\p_y\phi dy
&\leq \lambda\epsilon^{1-a}\|\p_y\phi\|^2_{L^2_y}+C\lambda^{-1}\epsilon^{1+a}\int_{\R}|n|^2|E+u\times B|^2dy\nnm\\
&\leq \lambda\epsilon^{1-a}\|\p_y\phi\|^2_{L^2_y}+C\lambda^{-1}\mathcal{E}(\tau)\mathcal{D}(\tau).\nnm
\ema
For $I_{\phi}^4$, by $G_1=\bar{G}+\sqrt{\mu}\gt$, we have from \eqref{Gb2} that
\bma
I_{\phi}^4&=-\epsilon^{1-a}\int_{\R}\frac{1}{\rho}\p_y\phi\int_{\R^3}v_1^2 \dy(\bar{G}+\sqrt{\mu}\gt) dvdy\nnm\\
&\leq \lambda\epsilon^{1-a}\|\p_y\phi\|^2_{L^2_y}
   +C\lambda^{-1}\frac{\epsilon^3}{\delta(\delta+\epsilon^a\tau)^2}
   +C\lambda^{-1}\sigma^2\frac{\epsilon^2}{\delta^2}\mathcal{D}(\tau)
   +C\lambda^{-1}\epsilon^{1-a}\|\dy\gt\|^2_\sigma.\nnm
\ema
Similarly, by $G_2=\sqrt{\mu}\g$, we have
\be
\I_{\phi}^5=\epsilon\int_{\R}\frac{1}{\rho}\dy\phi\intr (v\times B)_1(\sqrt{\mu}\g)dvdy
\leq\lambda\epsilon^{1-a}\|\p_y\phi\|^2_{L^2_y}
 +C\lambda^{-1}\epsilon^2\mathcal{E}(\tau)\mathcal{D}(\tau).\nnm
\ee
Combining the  estimates of $I^i_{\phi}~(i=1,2,3,4,5)$, and by letting $\lambda$, $\epsilon$ small enough,
we complete the proof of the  lemma.
\end{proof}

\begin{flushleft}
\textbf{Step 2. } Estimate on  $\epsilon^{1-a}\|\p_\tau(\phi,\psi,\zeta,n)\|^2_{L^2_y}$.
\end{flushleft}		
\begin{lem}\label{time1}
Under the same assumptions as in Proposition \ref{priori3}, we have
\begin{align*}
&\epsilon^{1-a}\|\p_\tau(\psi,\phi,\zeta,n)\|_{L^2_y}\nnm\\
\leq&\,C\epsilon^{1-a}\|\p_y(\psi,\phi,\zeta,n)\|^2_{L^2_y}
  +C\frac{\epsilon^{1+a}}{(\frac{\delta}{\sigma}+\epsilon^a\tau)^2}\|(\phi,\psi,\zeta,n)\|^2_{L^2_y}
  +C\frac{\epsilon^3}{\delta(\frac{\delta}{\sigma}+\epsilon^a\tau)^2}\\
&+C\(\mathcal{E}(\tau)+\sigma^2\frac{\epsilon^2}{\delta^2}\)\mathcal{D}(\tau)
   +C\epsilon^{1-a}\|\dy(\gt,\g)\|^2_\sigma.
\end{align*}
\end{lem}
\begin{proof}
Multiplying the equation $\eqref{pert-sys-n}_1$ by $\epsilon^{1-a}\p_\tau\phi$,
$\eqref{pert-sys-n}_{i+1}$ by $\epsilon^{1-a}\p_\tau\psi_i~(i=1,2,3)$,
$\eqref{pert-sys-n}_5$ by $\epsilon^{1-a}\p_\tau\zeta$,
and integrating  over $\R_y$,
we have
\bma
\epsilon^{1-a}\|\p_{\tau}\phi\|^2_{L^2_y}
 =&-\epsilon^{1-a}\int_{\R}u_1\p_y\phi\p_\tau\phi dy
   -\epsilon^{1-a}\int_{\R}\(\rho\dy\psi_1+\phi\dy \bar{u}_1+\psi_1\dy \bar{\rho}\)\p_{\tau}\phi\,dy,\nnm\\
\epsilon^{1-a}\|\p_{\tau}\psi\|^2_{L^2_y}
 =&-\epsilon^{1-a}\int_{\R}u_1\p_y\psi\cdot\p_\tau\psi dy+\epsilon\int_{\R}\frac{n}{\rho}(E+u\times B)\cdot\p_\tau\psi dy
  \nnm\\
 & -\epsilon^{1-a}\int_{\R}\[\frac{R\theta}{\rho} \p_y \phi+R\p_y \zeta+\psi_1\dy\bar{u}_1
     +R\(\frac{\theta}{\rho}-\frac{\bar{\theta}}{\bar{\rho}}\)\p_y\bar{\rho}\]\p_{\tau}\psi_1 dy\nnm\\
 &-\epsilon^{1-a}\int_{\R}\frac{1}{\rho}\p_\tau\psi\cdot\intr v_1v\dy G_1dvdy
  +\epsilon\int_{\R}\frac{1}{\rho}\p_\tau\psi\cdot\intr (v\times B)G_2dvdy,\nnm\\
\epsilon^{1-a}\|\p_{\tau}\zeta\|^2_{L^2_y}
 =&-\epsilon^{1-a}\int_{\R}u_1\p_y\zeta\p_\tau\zeta dy
  -\epsilon^{1-a}\int_{\R}\(R\theta\dy\psi_1+\psi_1\dy \bar{\theta}+R \zeta\dy \bar{u}_1\)\p_{\tau}\zeta dy\nnm\\
 &+\epsilon^{1-a}\int_{\R}\frac{1}{\rho}\p_\tau\zeta\intr v_1\[(v\cdot u)-\frac12 |v|^2\]\dy G_1dvdy\nnm\\
  & +\epsilon\int_{\R}\frac{1}{\rho}\p_\tau\zeta\intr [E+(u\times B)]\cdot vG_2dvdy.
\ema
Moreover, taking inner product between $\eqref{macro-eq-nond}_6$ and $\epsilon^{1-a}\p_\tau n$ to yield
$$
\epsilon^{1-a}\|\p_{\tau}n\|^2_{L^2_y}
= -\epsilon^{1-a}\int_{\R}\dy(nu_1)\dtau n dy-\int_{\R}\dtau n\int_{\R^3}v_1\dy G_2dvdy.
$$
By Cauchy-Schwarz inequality and the Sobolev embedding inequality, one gets
\begin{align*}
\epsilon^{1-a}\|\p_\tau(\psi,\phi,\zeta,n)\|^2_{L^2_y}
\leq&\, \lambda\epsilon^{1-a}\|\p_\tau(\psi,\phi,\zeta,n)\|^2_{L^2_y}
 +C\lambda^{-1}\(\mathcal{E}(\tau)+\sigma^2\frac{\epsilon^2}{\delta^2}\)\mathcal{D}(\tau)\\
&+C\lambda^{-1}\epsilon^{1-a}\|\p_y(\psi,\phi,\zeta,n)\|^2_{L^2_y}
 +C\lambda^{-1}\epsilon^{1-a}\|\dy(\gt,\g)\|^2_\sigma\\
&+C\lambda^{-1}\frac{\epsilon^{1+a}}{(\frac{\delta}{\sigma}+\epsilon^a\tau)^2}\|(\phi,\psi,\zeta,n)\|^2_{L^2_y}
  +C\lambda^{-1}\frac{\epsilon^3}{\delta(\frac{\delta}{\sigma}+\epsilon^a\tau)^2}.
\end{align*}
Hence, by choosing $\lambda>0$ small enough, we cmpletes the proof of the lemma.
\end{proof}

\subsection{Estimate on  $(\gt,\g)$}
In this subsection, we will estimate the microscopic part starting from  the following lemma.
\begin{lem}\label{gg1}
Under the same assumptions as in Proposition \ref{priori3}, we have
\bma
&\frac{d}{d\tau}\|(\gt,\g)\|^2+\sigma_1\epsilon^{a-1}\|(\gt,\g)\|_\sigma^2\nnm \\
\leq&\, C\epsilon^{1-a}\|\dy(\psi,\zeta,n)\|^2_{L^2_y}
  +C\epsilon^{1-a}\norm{\p_y(\gt,\g)}_\sigma^2
  +C\epsilon^{1+a}\|E+u\times B\|^2_{L^2_y}
   +C\frac{\epsilon^{1+a}}{(\frac{\delta}{\sigma}+\epsilon^a\tau)^2}\mathcal{E}(\tau)\nnm\\
 &+C\(\eta^2_1+\mathcal{E}(\tau)+\sigma^2\frac{\epsilon^{2a}}{\delta^2}\)\mathcal{D}(\tau)
  +C\frac{\epsilon^3}{\delta(\delta+\epsilon^a\tau)^2}
  + C\mathcal{E}^{\frac12}(\tau)\mathcal{F}_\omega(\tau),\nnm
\ema
where $\eta_1>0$ is defined in \eqref{local}.
\end{lem}
\begin{proof}
Multiplying \eqref{g1} by $\gt$, \eqref{g2} by $\g$  and integrating  over $\R^3_v\times\R_y$,
we have the following by using \eqref{H1}
\bma
&\frac{1}{2}\frac{d}{d\tau}\norm{(\gt,\g)}^2+\sigma_1\epsilon^{a-1}\norms{(\gt,\g)}^2\nnm\\
\leq&\,-\epsilon^a\int_{\R}\int_{\R^3}\frac{1}{\mu}(E+v\times B)\cdot\nabla_v\big(\mu\gt\g \big)dvdy\nnm\\
&+\int_{\R}\int_{\R^3}\frac{1}{\sqrt{\mu}}P_0\big[v_1 \sqrt{\mu}\dy\gt
 +\epsilon^a(E+v\times B)\cdot\nabla_v(\sqrt{\mu}\g)\big]\gt dvdy\nnm\\
&+\int_{\R}\int_{\R^3}\frac{1}{\sqrt{\mu}}\[P_d(v_1 \sqrt{\mu}\p_y\g)- \epsilon^a (E+v\times B)\cdot\Tdv \bar{G}_1\]\g dvdy\nnm\\
&-\int_{\R}\int_{\R^3}\frac1{\sqrt{\mu}}P_1\[v_1\(\frac{|v-u|^2\p_y\zeta}{2R\theta^2}
 +\frac{(v-u)\p_y\psi}{R\theta}\)M\]\gt dvdy\nnm\\
&+\int_{\R}\int_{\R^3} (R_{g1}+R_{\Gamma1})\gt
 + (R_{g2}+R_{\Gamma2} )\g dvdy\nnm\\
:=&\,J_{g}^1+J_{g}^2+J_{g}^3+J_{g}^4+J_{g}^5,
\ema
where $R_{g1},R_{g2},R_{\Gamma1},R_{\Gamma2}$ are defined in \eqref{remainder}.

For $J_{g}^1$, by the integration by parts and using the fact that $\nabla_v\cdot(E+v\times B)=0,~v\cdot(v\times B)=0$, we have
\bma\label{Jg1}
J_g^1
&=\epsilon^a\int_{\R}\int_{\R^3}\nabla_v\(\frac{1}{\mu}\)\cdot(E+v\times B)\big(\mu\gt\g \big)dvdy\nnm\\
&\leq C\epsilon^{a}\|E\|_{L^\infty_y}\|\langle v\rangle^{\frac12}(\gt,\g)\|^2
\leq  C\mathcal{E}^{\frac12}(\tau)\mathcal{F}_\omega(\tau).
\ema
For $J_g^2$, $J_g^3$ and $J_g^4$,
from \eqref{P10}, \eqref{Pd} and using the fact that $\nabla_v\cdot(E+v\times B)=0$, we have
\be\label{pov}
P_0\[ (E+v\times B)\cdot\nabla_v(\sqrt{\mu}\g)\]=- \sum_{j=1}^4\chi_{j}\int_{\R^3}(E+v\times B)(\sqrt{\mu}\g)\cdot\nabla_{v}\(\frac{\chi_j}{M}\)dv.
\ee	
Hence
\bma\label{Jg234}				
J_g^2+J_g^3+J_g^4
\leq&\,\lambda\epsilon^{a-1}\norm{(\gt,\g)}_\sigma^2
     +C\lambda^{-1}\epsilon^{1+a}\|(E,B)\|^2_{L^\infty_y}\|\g\|^2_\sigma\nnm\\
&+C\lambda^{-1}\epsilon^{1-a}\norm{\p_y(\gt,\g)}_\sigma^2
     +C\lambda^{-1}\epsilon^{1-a}\|\dy(\psi,\zeta)\|^2_{L^2_y}\nnm\\
&+C\lambda^{-1}\epsilon^{1+3a}\|(E,B)\|^2_{L^2_y} \| |\mu^{-\frac12}\v^{\frac32}\nabla_v\bar{G}_1|_2 \|^2_{L^\infty_y}\nnm\\
\leq&\,\lambda\epsilon^{a-1}\norm{(\gt,\g)}_\sigma^2
     +C\lambda^{-1}\epsilon^2\mathcal{E}(\tau)\mathcal{D}(\tau)
     +C\lambda^{-1}\frac{\epsilon^{3+a}}{(\frac{\delta}{\sigma}+\epsilon^a\tau)^2}\|(E,B)\|^2_{L^2_y}\nnm\\
&+C\lambda^{-1}\epsilon^{1-a}\norm{\p_y(\gt,\g)}_\sigma^2
     +C\lambda^{-1}\epsilon^{1-a}\|\dy(\psi,\zeta)\|^2_{L^2_y}.
\ema
Next, we focus on $J_g^5$. For terms involving $R_{g1}$ and $R_{g2}$, we use \eqref{Gb2} to have
\bma
\int_{\R}\int_{\R^3} R_{g1}\gt dvdy
&\leq \lambda\epsilon^{a-1}\norms{\gt}^2
 +C\lambda^{-1}\epsilon^{1-a}\sum_{|\alpha|=1}\left\|\langle v\rangle\frac{\p^\alpha\bar{G}_1}{\sqrt{\mu}}\right\|_{2}^2\nnm\\
&\leq \lambda\epsilon^{a-1}\norms{\gt}^2
 +C\lambda^{-1}\sigma^2\frac{\epsilon^2}{\delta^2}\mathcal{D}(\tau)
 +C\lambda^{-1}\frac{\epsilon^3}{\delta(\frac{\delta}{\sigma}+\epsilon^a\tau)^2},
\ema
and
\bma
\int_{\R}\int_{\R^3} R_{g2}\g dvdy
&\leq \lambda\epsilon^{a-1}\norms{\g}^2
 +C\lambda^{-1}\epsilon^{1-a}\|\p_y n\|^2_{L^2_y}
 +C\lambda^{-1}\frac{\epsilon^{1+a}}{(\frac{\delta}{\sigma}+\epsilon^a\tau)^2}\|n\|^2_{L^2_y}\nnm\\
&\quad+C\lambda^{-1}\epsilon^{1+a}\|E+u\times B\|^2_{L^2_y}
  +C\lambda^{-1}\(\sigma^2\frac{\epsilon^{2a}}{\delta^2}+\mathcal{E}(\tau)\)\mathcal{D}(\tau).
\ema
For terms involving $R_{\Gamma1}$ and $R_{\Gamma2}$, recalling that $F_2=\frac n\rho M+\sqrt{\mu}\g,G_1=\bar{G}_1+\sqrt{\mu}\gt$, by using Lemma \ref{GMG} and Lemma \ref{QGG}, we have
\bma
\int_{\R}\int_{\R^3}R_{\Gamma1}\gt+R_{\Gamma2}\g dvdy
\leq&\lambda\epsilon^{a-1}\|(\gt,\g)\|^2_\sigma
 +C\lambda^{-1}\(\eta^2_1+\mathcal{E}(\tau)+\sigma^2\frac{\epsilon^{2a}}{\delta^2}\)\mathcal{D}(\tau)\nnm\\
&+C\lambda^{-1}\frac{\epsilon^{1+a}}{(\frac{\delta}{\sigma}+\epsilon^a\tau)^2}\|n\|^2_{L^2_y}
 +C\lambda^{-1}\frac{\epsilon^3}{\delta(\frac{\delta}{\sigma}+\epsilon^a\tau)^2}.
\ema
Then combining all estimates for $J_g^i$, $i=1,2,...,5$ and letting $\lambda$ small enough, we complete the proof of the lemma.
\end{proof}
Combining Lemma \ref{lower1-1}-Lemma \ref{gg1} and by letting $\sigma\frac{\epsilon^a}{\delta}$ small enough,
we obtain the conclusion in Lemma \ref{lolem}.

\section{High order energy estimates}
\setcounter{equation}{0}
In this section, we will derive some energy estimates of high order derivatives.

\begin{lem}\label{lolem1}
Under the same assumptions as  in Proposition \ref{priori3} and  by letting $\frac{\epsilon^a}{\delta}$ small enough, we have the following estimate for $|\alpha|=1$ and $|\alpha_1|=2$,
\bma
& \epsilon^{2-2a}\frac{d}{d\tau}\intra \frac{|\p^{\alpha_1}(F_1,F_2)|^2}{\mu}+\frac1{R\theta}(|\p^{\alpha_1}E|^2+2\p^{\alpha_1} E\cdot u\times\p^{\alpha_1}B+|\p^{\alpha_1}B|^2)dy
\nnm\\
&~+\frac{d}{d\tau}\|\p^\alpha(\phi,\psi,\zeta,n,E,B)\|^2_{L^2_y}
 -\epsilon^{1+a}\frac{d}{d\tau}\int_{\R}E\cdot\O\dy B dy+\frac{d}{d\tau}\|\p^{\alpha} (\gt,\g)\|^2
\nnm\\
&~+\epsilon^{1+a}\|\p^{\alpha}(E,B)\|^2_{L^2_y}+ \epsilon^{a-1}\|\p^{\alpha} (\gt,\g)\|^2_{\sigma}
 +\epsilon^{1-a} (\|\p^{\alpha_1}(\phi,\psi,\zeta,n)\|^2_{L^2_y}  + \|\p^{\alpha_1} (\gt,\g)\|_{\sigma}^2)\nnm\\
&\leq  \lambda \mathcal{D}(\tau)
 +C\lambda^{-1}\(\eta^2_1+\epsilon^a+\sigma^2\frac{\epsilon^{2a}}{\delta^2}+\epsilon^{a-1}\mathcal{E}^{\frac12}(\tau)
          +\epsilon^{-a}\mathcal{E}(\tau)
          +(1+T)\epsilon^{a-1}\mathcal{E}(\tau)\)\mathcal{D}(\tau)\nnm\\
&~~+C\frac{\epsilon^{1+2a}}{\delta(\frac{\delta}{\sigma}+\epsilon^a\tau)^2}
  +C\frac{\epsilon^a}{\frac{\delta}{\sigma}+\epsilon^a\tau}\mathcal{E}(\tau)
 +C\mathcal{E}^{\frac12}(\tau)\mathcal{F}_\omega(\tau),\nnm
\ema
where $\lambda>0$ is a small constant, and $\eta_1>0$ is defined in \eqref{local}.
\end{lem}
The proof of Lemma \ref{lolem1} is divided in  the following five subsections.

\subsection{Estimate on derivatives of  $ (\phi,\psi,\zeta)$}
 For  the macroscopic component $ (\phi,\psi,\zeta)$, applying the operator $\p^\alpha~(|\alpha|=1)$ to $\eqref{pert-sys}_1$ together with
$\frac1\rho\eqref{pert-sys}_i~(i=2,3,4,5)$ yields
\be\label{first1}
\left\{\begin{aligned}
&\dtau \p^\alpha\phi+u_1\dy\p^\alpha\phi+\rho\dy\p^\alpha\psi_1=N_1+W_1,\\
&\dtau \p^\alpha\psi_1+u_1\dy \p^\alpha\psi_1+R\dy\p^\alpha\zeta+R\frac{\theta}{\rho}\dy\p^\alpha\phi
=N_2+W_2\\
&\qquad+\epsilon^{1-a}\frac43\p^\alpha\[\frac1\rho\dy(\kappa_1(\theta)\dy \psi_1)\]
  +\p^\alpha\[\frac1\rho(H_e^1+H^1_f+H^1_n)\],\\
&\dtau \p^\alpha\psi_i+u_1\dy \p^\alpha\psi_i
=W_{i+1}+\epsilon^{1-a}\p^\alpha\[\frac1\rho\dy(\kappa_1(\theta)\dy \psi_i)\]
  +\p^\alpha\[\frac1\rho({ H_f^i}+H_e^i+H^i_n)\],\quad i=2,3,\\
&\dtau \p^\alpha\zeta+u_1\dy \p^\alpha\zeta+R\theta\dy\p^\alpha\psi_1
   =N_5+W_5+\epsilon^{1-a}\p^\alpha\[\frac1\rho\dy(\kappa_2(\theta)\dy \zeta)\]+\p^\alpha\[\frac1\rho(H_e^4+H^4_f+H^4_n)\],
\end{aligned}\right.
\ee
where
\bq\label{N1}
\left\{\bln
N_{1}=&-\(\p^\alpha\psi_1\p_{y}\bar{\rho}+\p^\alpha\phi\p_{y}\bar{u}_1\)
       -\(\psi_1\p_{y}\p^\alpha\bar{\rho}+\phi\p_{y}\p^\alpha\bar{u}_1\),\\
N_{2}=&-\Big(\p^\alpha\psi_1\p_y\bar{u}_1+R\p^\alpha\Big(\frac{\theta}{\rho}-\frac{\bar{\theta}}{\bar{\rho}}\Big)\p_y\bar{\rho}\Big)
       -\Big(\psi_1\p_{y}\p^\alpha\bar{u}_1+R\Big(\frac{\theta}{\rho}-\frac{\bar{\theta}}{\bar{\rho}}\Big)\p_y\p^\alpha\bar{\rho}\Big),\\
N_5=&-\(\p^\alpha\psi_1\p_{y}\bar{\theta}+R\p^\alpha\zeta\p_{y}\bar{u}_1\)
       -\(\psi_1\p_{y}\p^\alpha\bar{\theta}+R\zeta\p_{y}\p^\alpha\bar{u}_1\),
\eln\right.
\eq
and
\bma\label{W1}
    \left\{\begin{aligned}
        &W_1=-\(\p^\alpha u_1\p_{y}\phi+\p^\alpha\rho\p_{y}\psi_1\),\quad
        W_2=-\(\p^\alpha u_1\p_{y}\psi_1+\p^\alpha\(\frac{R\theta}{\rho}\)\p_{y}\phi\),\\
        &W_{i+1}=-\p^\alpha u_1\p_{y}\psi_i,\quad i=2,3,\quad
        W_5=-\(\p^\alpha u_1\p_{y}\zeta+R\p^\alpha\theta\p_{y}\psi_1\).
    \end{aligned}
    \right.
\ema
\begin{lem}\label{mac1}
Under the same assumptions as in Proposition \ref{priori3} and by letting $\sigma\frac{\epsilon^a}{\delta}$ small enough,
it holds that for $|\alpha|=1$,
\bma
&\frac{d}{d\tau}\|\p^\alpha(\phi,\psi,\zeta)\|^2_{L^2_y}
 +\epsilon^{1-a}\|\p^\alpha\dy(\psi,\zeta)\|^2_{L^2_y}\nnm\\
\leq&\,\lambda\mathcal{D}(\tau)
 +C\lambda^{-1}\epsilon^{1-a}\sum_{|\alpha_1|=2}\|\p^{\alpha_1}(\gt,\g)\|^2_{\sigma}
 +C\frac{\epsilon^a}{(\frac{\delta}{\sigma}+\epsilon^a\tau)}\mathcal{E}(\tau)\nnm\\
&+C\lambda^{-1}\frac{\epsilon^{1+4a}}{\delta^3(\frac{\delta}{\sigma}+\epsilon^a\tau)^2}
 +C\(\sigma^2\frac{\epsilon^{1+a}}{\delta^2}+\epsilon^{a-1}\mathcal{E}^{\frac12}(\tau)\)\mathcal{D}(\tau),\nnm
\ema
where $\lambda>0$ is a small constant.
\end{lem}
		
\begin{proof}
For $|\alpha|=1$, multiplying $\eqref{first1}_{1}$ by $\frac{R\theta}{\rho^2}\p^\alpha\phi$,
$\eqref{first1}_{i+1}$ by $\p^\alpha\psi_i~(i=1,2,3)$,
$\eqref{first1}_5$ by $\frac{\p^\alpha\zeta}{\theta}$, respectively,
then  integrating together over $\R_y$ yields
\bma\label{First}
&\frac{d}{d\tau}\int_{\R}\frac{R\theta}{2\rho^2}|\p^\alpha\dy\phi|^2+\frac12|\p^\alpha\dy\psi|^2+\frac{1}{2\theta}|\p^\alpha\dy\zeta|^2dy\nnm\\
&\quad+\int_{\R}\[\frac{\mu(\theta)}{3\rho}|\p^\alpha\dy\psi_1|^2+\frac{\mu(\theta)}{\rho}|\p^\alpha\dy\psi|^2+\frac{\kappa(\theta)}{\rho\theta}|\p^\alpha\dy\zeta|^2\]dy\nnm\\
=&\, \int_{\R}\[\p_\tau\(\frac{R\theta}{2\rho^2}\)
     +\p_y\(\frac{Ru_1\theta}{2\rho^2}\)\]|\p^\alpha\phi|^2+\frac{\p_y u_1}{2}|\p^\alpha\psi|^2
     +\[\p_\tau\(\frac{1}{2\theta}\)+\p_y\(\frac{u_1}{2\theta}\)\]|\p^\alpha\zeta|^2dy\nnm\\
&-\int_{\R} \frac{R\theta}{\rho}\Big(\p_y\p^\alpha\psi_1\p^\alpha\phi +\p_y\p^\alpha\phi\p^\alpha\psi_1\Big)
       +R\Big(\p_y\p^\alpha\zeta\p^\alpha\psi_1+\p_y\p^\alpha\psi_1\p^\alpha\zeta\Big)dy\nnm\\
&+\epsilon^{1-a}\bigg[\int_{\R} \frac43\p^\alpha\(\frac1\rho\dy(\kappa_1(\theta)\dy \psi_1)\)\p^{\alpha}\psi_1+ \p^\alpha\(\frac1\rho\dy(\kappa_1(\theta)\dy \psi_i)\)\p^{\alpha}\psi_i\nnm\\
&+ \p^\alpha\(\frac1\rho\dy(\kappa_2(\theta)\dy \zeta)\)\frac{\p^\alpha\zeta}{\theta}dy\bigg] + \[\int_{\R}N_{1}\frac{R\theta}{\rho^2}\p^\alpha\phi
       +N_2\p^\alpha\psi_1+N_3\frac{\p^\alpha\zeta}{\theta}dy\] \nnm\\
&+\int_{\R} W_1\frac{R\theta}{\rho^2}\p^\alpha\phi+W_2\p^\alpha\psi_1
                +W_3\p^\alpha\psi_2+W_4\p^\alpha\psi_3+W_5\frac{\p^\alpha\zeta}{\theta} dy\nnm\\
&+\intra\sum_{i=1}^3\p^\alpha\[\frac1\rho(H_e^i+H_f^i+H_n^i)\]\p^\alpha\psi_i
   + \p^\alpha\[\frac1\rho(H_e^4+H_f^4+H_n^4)\]\frac{\p^\alpha\zeta}{\theta}dy
:=\sum_{i=1}^6I_{\alpha}^i.
\ema
For $I^2_{\alpha}$ we have
\bma
 &R\Big(\p_y\p^\alpha\zeta\p^\alpha\psi+\p_y\p^\alpha\psi\p^\alpha\zeta\Big)
  =R\p_y(\p^\alpha\psi\p^\alpha\zeta),\nnm\\
&\frac{R\theta}{\rho}\Big(\p_y\p^\alpha\psi\p^\alpha\phi +\p_y\p^\alpha\phi\p^\alpha\psi\Big)
  =\frac{R\theta}{\rho}\p_y(\p^\alpha\psi\p^\alpha\phi)
  =-\p_y\(\frac{R\theta}{\rho}\)(\p^\alpha\psi\p^\alpha\phi)
    +\p_y(\frac{R\theta}{\rho}\p^\alpha\psi\p^\alpha\phi).\nnm
 \ema
Then,
 for $I_{\alpha}^1,I^2_\alpha$ and $I_{\alpha}^5$, by Lemma \ref{rarefaction4}, we have for $|\alpha|=1$,
\bma\label{first2}
I_\alpha^1+I^2_\alpha+I_{\alpha}^5
&\leq C\int_{\R}|\p^{\alpha}(\bar{\rho},\bar{u}_1,\bar{\theta})||\p^\alpha(\phi,\psi,\zeta)|^2
  +|\p^{\alpha}(\phi,\psi,\zeta)||\p^\alpha(\phi,\psi,\zeta)|^2dy\nnm\\
&\leq C\frac{\epsilon^a}{\delta+\epsilon^a\tau}\|\p^\alpha(\phi,\psi,\zeta)\|^2_{L^2_y}
  +C\|\p^{\alpha}(\phi,\psi,\zeta)\|_{L^2_y}\|\p^\alpha(\phi,\psi,\zeta)\|^2_{L^4_y}\nnm\\
&\leq C\frac{\epsilon^a}{\frac{\delta}{\sigma}+\epsilon^a\tau}\mathcal{E}(\tau)
  +C\epsilon^{a-1}\mathcal{E}^{\frac12}(\tau)\mathcal{D}(\tau).
\ema
For $I_{\alpha}^3$ and $I_{\alpha}^4$, by Lemma \ref{rarefaction4}, we have for $|\alpha|=1$,
\bma\label{ia3}
I_\alpha^3
&\leq C\epsilon^{1-a}\int_{\R}|\p^\alpha(\rho,\theta)|^2||\p^\alpha(\psi,\zeta)|^2
 +|\p^\alpha\dy\theta||\p^\alpha(\psi,\zeta)|^2
 +|\p^\alpha(\rho,\theta)||\p^\alpha\dy(\psi,\zeta)||\p^\alpha(\psi,\zeta)|dy\nnm\\
&\leq \lambda\epsilon^{1-a}\|\p^\alpha\dy(\psi,\zeta)\|^2_{L^2_y}
 +C\lambda^{-1}\(\mathcal{E}(\tau)+\sigma^2\frac{\epsilon^{2a}}{\delta^2}\)\mathcal{D}(\tau),
\ema
and
\bma\label{first3}
I_{\alpha}^4
&\leq C\int_{\R}|N_1||\p^\alpha\phi|+|N_2||\p^\alpha\psi|+|N_3||\p^\alpha\zeta| dy\nnm\\
&\leq C\int_{\R}|\p_y(\bar{\rho},\bar{u}_1,\bar{\theta})||\p^\alpha(\phi,\psi,\zeta)|^2
 +|(\phi,\psi,\zeta)|\(|\p^\alpha\p_y(\bar{\rho},\bar{u}_1,\bar{\theta})|
        +|\p^\alpha(\bar{\rho},\bar{\theta})||\p_y\bar{\rho}|\)|\p^\alpha(\phi,\psi,\zeta)|dy\nnm\\
&\leq C\frac{\epsilon^a}{\frac{\delta}{\sigma}+\epsilon^a\tau}\mathcal{E}(\tau).
\ema
To estimate $I_\alpha^6$, we first study terms involving $H^i_f~(i=1,2,3,4)$ in \eqref{Hf}. By Lemma \ref{rarefaction4}, we have
\bma
&\sum_{i=1}^3\intra\p^\alpha\(\frac1\rho H_f^i\)\p^\alpha\psi_i dy
   +\intra\p^\alpha\(\frac1\rho H_f^4\)\frac{\p^\alpha\zeta}{\theta}dy\nnm\\
\leq&\,C\epsilon^{1-a}\int_{\R}|\p^\alpha(\psi,\zeta)|
   (|\p^{\alpha}\p_{yy}(\bar{u}_1,\bar{\theta})|
     +|\p^{\alpha}(\bar u_1,\bar\theta)||\p^{\alpha}\dy(\bar u_1,\bar\theta)|
     +|\p^{\alpha}\bar\theta||\dy\bar{u}_1|^2)dy\nnm\\
&+C\epsilon^{1-a}\int_{\R}|\p^\alpha(\psi,\zeta)||\p^{\alpha}\bar\rho|
     (|\p_{yy}\bar{u}_1|+|\dy(\bar{u}_1,\bar{\theta})|^2)\nnm\\
&+C\epsilon^{1-a}\int_{\R}|\p^\alpha(\psi,\zeta)||\p^{\alpha}(\phi,\zeta)|
     (|\p^{\alpha}\dy(\bar u_1,\bar\theta)|+|\p^{\alpha}(\bar u_1,\bar\theta)|^2)\nnm\\
&+C\epsilon^{1-a}\int_{\R}|\p^\alpha(\psi,\zeta)|(|\p^{\alpha}(\psi,\zeta)||\p^{\alpha}\dy\psi|+|\p^{\alpha}\phi||\dy\psi|^2)dy\nnm\\
\leq&\,
\lambda\epsilon^{1-a}\|\p^\alpha(\psi,\zeta)\|^2_{L^2_y}
 +C\lambda^{-1}\frac{\epsilon^{1+4a}}{\delta^3(\frac{\delta}{\sigma}+\epsilon^a\tau)^2}
 +C\(\sigma^2\frac{\epsilon^{1+a}}{\delta^2}+\mathcal{E}^{\frac12}(\tau)\)\mathcal{D}(\tau).
\ema
Secondly, we estimate $H^j_e~(j=1,2,3,4)$ in \eqref{He} which interact with electromagnetic fields.
\be
\sum_{i=1}^3\intra\p^\alpha\left(\frac{1}{\rho} H_e^i\right)\p^\alpha \psi_i dy
 +\intra\p^\alpha\left(\frac{1}{\rho} H_e^4\right)\frac{\p^\alpha \zeta}{\theta} dy
\leq\lambda\mathcal{D}(\tau)
 +C\lambda^{-1}\epsilon^{2a-2}\mathcal{E}(\tau)\mathcal{D}(\tau).
\ee
Lastly, we study terms involving $\Lambda_1$ in $H^j_n~(j=1,2,3,4)$ in \eqref{Hn} and the terms involving $\Lambda_2$ is similar. Noth that
\be
\int_{\R^3}v_i v_1 L_M^{-1} \Lambda_1 d v
 =\int_{\R^3} L_M^{-1}\left\{R \theta \hat{\B}_{i1}\(\frac{v-u}{\sqrt{R \theta}}\)M\right\} \frac{\Lambda_1}{M} dv
 =R\theta\int_{\R^3} \B_{i1}\(\frac{v-u}{\sqrt{R \theta}}\)\frac{\Lambda_1}{M} dv,\nnm
\ee
for $i=1,2,3$, where we have used the self-adjoint property of $L_M^{-1}$ and the definition of Burnett functions in \eqref{Burnett1}-\eqref{Burnett2}. Then,
\bma
&-\intra\p^\alpha\(\frac{1}{\rho}\int_{\R^3} v_i v_1 \dy L_M^{-1}\Lambda_1 dv\)\p^\alpha\psi_i dy
 =-\int_{\R}\p^\alpha\[\frac1{\rho} \dy\(\int_{\R^3} R\theta \B_{i1}\(\frac{v-u}{\sqrt{R\theta}}\)\frac{\Lambda_1}{M} dv\)\]\p^\alpha\psi_i dy\nnm\\
=&\,\int_{\R}\p^\alpha\[\frac{1}{\rho}\int_{\R^3} R\theta \B_{i1}\(\frac{v-u}{\sqrt{R \theta}}\)\frac{\Lambda_1}{M} dv\]\p^\alpha\dy\psi_i dy
  +\int_{\R}\p^\alpha\[\dy\(\frac{1}{\rho}\)\int_{\R^3} R\theta \B_{i1}\(\frac{v-u}{\sqrt{R\theta}}\)\frac{\Lambda_1}{M} dv\]\p^\alpha\psi_i dy\nnm\\
:=&\,I_{\Lambda^1_1}+I_{\Lambda^2_1}.
\ema
For $I_{\Lambda^1_1}$, by the definition of $\Lambda_1$ in \eqref{lambda1}, we have
\bma
I_{\Lambda^1_1}
=&\int_{\R^3}\p^\alpha\[\int_{\R^3}\frac{1}{\rho} R\theta \B_{i1}\(\frac{v-u}{\sqrt{R\theta}}\)\frac{\epsilon^{1-a}\dtau G_1}{M} dv\]
  \p^\alpha\dy\psi_i dy\nnm\\
&+\int_{\R^3}\p^\alpha\[\int_{\R^3}\frac{1}{\rho} R\theta \B_{i1}\(\frac{v-u}{\sqrt{R \theta}}\)\frac{\epsilon^{1-a} P_1(v\dy G_1)}{M} d v\]
  \p^\alpha \dy\psi_i dy\nnm\\
&-\int_{\R^3}\p^\alpha\[\int_{\R^3}\frac{1}{\rho} R\theta \B_{i1}\(\frac{v-u}{\sqrt{R \theta}}\)
  \frac{\epsilon P_1[(E+v\times B)\cdot\nabla_v G_1]}{M}dv\] \p^\alpha\dy\psi_i dy\nnm\\
&-\int_{\R^3}\p^\alpha\[\int_{\R^3}\frac{1}{\rho} R\theta \B_{i1}\(\frac{v-u}{\sqrt{R \theta}}\)\frac{Q(G_1,G_1)}{M} dv\]\p^\alpha\dy\psi_i dy\nnm\\
:=&I_{\Lambda^1_1}^1+I_{\Lambda^1_1}^2+I_{\Lambda^1_1}^3+I_{\Lambda^1_1}^4.
\ema
Noting that $G_1=\bar{G}_1+\sqrt{\mu}\gt$, from Lemma \ref{Burnett3}, \eqref{Gb2}-\eqref{Gb3} and letting $\sigma\frac{\epsilon^a}{\delta}\leq1$, we have
\bma
I_{\Lambda^1_1}^1
&=\int_{\R^3}\p^\alpha\[\int_{\R^3}\frac{1}{\rho} R\theta \B_{i1}\(\frac{v-u}{\sqrt{R \theta}}\)\epsilon^{1-a}\(\frac{ \dtau\bar{G}}{M}+\frac{\sqrt{\mu}}{M}  \dtau\gt\) dv\] \p^\alpha\dy\psi_i dy\nnm\\
&\leq \lambda\epsilon^{1-a}\|\p^\alpha\dy\psi\|^2_{L^2_y}
 +C\lambda^{-1}\(\sigma^2\frac{\epsilon^2}{\delta^2}+\mathcal{E}(\tau)\)\mathcal{D}(\tau)\nnm\\
&\quad +C\lambda^{-1}\frac{\epsilon^{3+2a}}{\delta^3(\frac{\delta}{\sigma}+\epsilon^a\tau)^2}
 +C\lambda^{-1}\epsilon^{1-a}\|\p^{\alpha}\dtau \gt\|^2_{\sigma}.
\ema
Similarly, for $I_{\Lambda^1_1}^2$ and $I_{\Lambda^1_1}^3$, we have
\bma
I_{\Lambda^1_1}^2+I_{\Lambda^1_1}^3
&\leq\lambda\epsilon^{1-a}\|\p^\alpha\dy\psi\|^2_{L^2_y}
 +C\lambda^{-1}\(\sigma^2\frac{\epsilon^2}{\delta^2}+\mathcal{E}(\tau)\)\mathcal{D}(\tau)\nnm\\
&\quad +C\lambda^{-1}\frac{\epsilon^{3+2a}}{\delta^3(\frac{\delta}{\sigma}+\epsilon^a\tau)^2}
 +C\lambda^{-1}\epsilon^{1-a}\|\p^{\alpha}\dy\gt\|^2_{\sigma}.\nnm
\ema
 As for the term $I_{\Lambda^1_1}^4$, we use Lemmas \ref{Burnett3} and \ref{QGG} to write
\bma
I_{\Lambda^1_1}^4
=&\,-\int_{\R}\int_{\R^3}\p^\alpha\[\frac{1}{\rho} R\theta \B_{i1}\(\frac{v-u}{\sqrt{R \theta}}\)\frac{\sqrt{\mu}}{M}\]
   \Gamma\(\frac{G_1}{\sqrt{\mu}},\frac{G_1}{\sqrt{\mu}}\) dv \p^\alpha \dy\psi_i dy\nnm\\
&-\int_{\R}\int_{\R^3}\frac{1}{\rho} R\theta \B_{i1}\(\frac{v-u}{\sqrt{R \theta}}\)\frac{\sqrt{\mu}}{M}
   \p^\alpha\Gamma\(\frac{G_1}{\sqrt{\mu}},\frac{G_1}{\sqrt{\mu}}\) dv \p^\alpha \dy\psi_i dy\nnm\\
\leq&\, \lambda\epsilon^{1-a}\|\p^\alpha\dy\psi\|^2_{L^2_y}
    +C\lambda^{-1}\frac{\epsilon^{3+2a}}{\delta^3(\delta+\epsilon^a\tau)^2}
    +C\lambda^{-1}\(\frac{\epsilon^2}{\delta^2}+\mathcal{E}(\tau)\)\mathcal{D}(\tau).
\ema
Therefore,  by combining  the estimates $I_{\Lambda^1_1}^1-I_{\Lambda^1_1}^4$, we can get
\bma
I_{\Lambda^1_1}
\leq\lambda\epsilon^{1-a}\|\p^\alpha\dy\psi\|^2_{L^2_y}
 +C\lambda^{-1}\(\sigma^2\frac{\epsilon^2}{\delta^2}+\mathcal{E}(\tau)\)\mathcal{D}(\tau)
 +C\lambda^{-1}\frac{\epsilon^{3+2a}}{\delta^3(\frac{\delta}{\sigma}+\epsilon^a\tau)^2}
 +C\lambda^{-1}\epsilon^{1-a}\sum_{|\alpha|=2}\|\p^{\alpha}\gt\|^2_{\sigma}.\nnm
\ema
For $I_{\Lambda_1^2}$, by using  an arguments similar to  the one for $I_{\Lambda_1^1}$, we have
\bma
I_{\Lambda^2_1}
\leq\lambda\epsilon^{1-a}\|\p^\alpha\dy\psi\|^2_{L^2_y}
 +C\lambda^{-1}\(\sigma^2\frac{\epsilon^2}{\delta^2}+\mathcal{E}(\tau)\)\mathcal{D}(\tau)
 +C\lambda^{-1}\frac{\epsilon^{3+2a}}{\delta^3(\frac{\delta}{\sigma}+\epsilon^a\tau)^2}
 +C\lambda^{-1}\epsilon^{1-a}\sum_{|\alpha|=2}\|\p^{\alpha}\gt\|^2_{\sigma}.\nnm
\ema
Similarly, for $|\alpha|=1$, by \eqref{lambda2} and Lemma \ref{Burnett3} we have
{
\bma
&\quad\epsilon^a\sum^3_{i=1}\int_{\R}\p^\alpha\int_{\R^3}\frac1\rho(v\times B)_i \L_M^{-1}\Lambda_2dv\p^\alpha\psi_idy
 +\epsilon^a\int_{\R}\p^\alpha\int_{\R^3}\frac1\rho[E+(u\times B)]\cdot v \L_M^{-1}\Lambda_2dv\frac{\p^\alpha\zeta}{\theta}dy\nnm\\
&\leq\lambda\mathcal{D}(\tau)
 +C\lambda^{-1}\(\sigma^2\frac{\epsilon^2}{\delta^2}+\mathcal{E}(\tau)\)\mathcal{D}(\tau)
 +C\lambda^{-1}\frac{\epsilon^{3+2a}}{\delta^3(\frac{\delta}{\sigma}+\epsilon^a\tau)^2}
 + C\lambda^{-1}\epsilon^{1-a}\sum_{|\alpha|=2}\|\p^{\alpha}\gt\|^2_{\sigma}.
\ema}
Combining the estimates for $I_{\alpha}^i$, $i=1,2,...,5,$ we complete the proof of the  lemma.
\end{proof}

\subsection{Estimate on derivatives of $(E,B)$ and  $n$}
\begin{flushleft}
\textbf{Step 1. } Estimates on  $\|\p^\alpha(E,B)\|_{L^2_y}$ with $|\alpha|=1$.
\end{flushleft}

\begin{lem}\label{EB1}
Under the same assumptions as in Proposition \ref{priori3} and by letting $\sigma\frac{\epsilon^a}{\delta}$ small enough,
it holds that for $|\alpha|=1$,
\bma\label{EB-first}
&2\frac{d}{d\tau}\|\p^{\alpha}(E,B)\|^2_{L^2_y}
  -\epsilon^{1+a}\frac{d}{d\tau}\int_{\R}E\cdot\O\dy B dy
  + \epsilon^{1+a}\|\p^{\alpha}(E,n)\|^2_{L^2_y}
  +\epsilon^{1+a}\|\p^{\alpha} B\|^2_{L^2_y}\nnm\\
\leq&\, C\frac{\epsilon^a}{\frac{\delta}{\sigma}+\epsilon^a\tau}\mathcal{E}(\tau)
 +C\(\eta^2_0+\epsilon^{a}+\sigma^2\frac{\epsilon^{2a}}{\delta^2}+\epsilon^{a-1}\mathcal{E}^{\frac12}(\tau)\)\mathcal{D}(\tau)
 +C\epsilon^{1-a}\sum_{|\alpha_1|=2}\|\p^{\alpha_1}\g\|^2_\sigma.
\ema

\end{lem}
\begin{proof}	
Applying $\p^{\alpha}$ to \eqref{EB-inv-s}, we have
\be\label{EB-inv-1}
\left\{\begin{aligned}
\dtau \p^{\alpha}E_{1}=&-\epsilon^a\p^{\alpha}(nu_1)
  +\epsilon\frac{\sigma(\theta)}{\rho}{\p^{\alpha}\p_yn}
 -\epsilon^{1+a}\frac{\sigma(\theta)}{R\theta}\p^{\alpha}\big(E+u\times B\big)_1
 -\epsilon^a\intr v_1\p^{\alpha}\L_M^{-1}\Lambda_{2}dv+{L}^2_{E_1}, \\
\dtau\p^{\alpha} E_{2}=&-\p^{\alpha}\p_y B_3-\epsilon^a\p^{\alpha}(n\psi_2)
 -\epsilon^{1+a}\frac{\sigma(\theta)}{R\theta}\p^{\alpha}\big(E+u\times B\big)_2
 -\epsilon^a\intr v_2\p^{\alpha}\L_M^{-1}\Lambda_{2}dv+{L}^2_{E_2}, \\
\dtau \p^{\alpha}E_{3}=&\p^{\alpha}\p_y B_2-\epsilon^a\p^{\alpha}(n\psi_3)
 -\epsilon^{1+a}\frac{\sigma(\theta)}{R\theta}\p^{\alpha}\big(E+u\times B\big)_3
 -\epsilon^a\intr v_3\p^{\alpha}\L_M^{-1}\Lambda_2dv+{L}^2_{E_3},
\end{aligned}\right.
\ee
where
\begin{align*}				
L^2_{E_1}:=&-\epsilon\p^{\alpha}\(\frac{\sigma(\theta)}{\rho^2}\p_y\rho n\)			+\bigg[\epsilon\p^{\alpha}(\frac{\sigma(\theta)}{\rho})\p_yn
  -\epsilon^{1+a}\p^{\alpha}(\frac{\sigma(\theta)}{R\theta})(E+u\times B)_1\bigg],\\	
{L}^2_{E_2}:=&-\epsilon^{1+a}\p^{\alpha}(\frac{\sigma(\theta)}{R\theta})(E+u\times B)_2,\quad			{L}^2_{E_3}:=-\epsilon^{1+a}\p^{\alpha}(\frac{\sigma(\theta)}{R\theta})(E+u\times B)_3.
\end{align*}
Taking the inner product between \eqref{EB-inv-1}$_i$ and $\p^{\alpha}E_i~(i=1,2,3)$, we have
\bma\label{EB-1-1}
&\frac12\frac{d}{d\tau} \|\p^{\alpha}E\|^2_{L^2_y}
 +\epsilon^{1+a}\int_{\R}\frac{\sigma(\theta)}{R\theta}|\p^{\alpha}E|^2dy
 -\epsilon\int_{\R}\frac{\sigma(\theta)}{\rho}\p^{\alpha}\dy n\p^{\alpha}E_1dy\nnm\\
=&\int_{\R}\O\p^\alpha\dy B\cdot\p^{\alpha}Edy
  -\epsilon^a\int_{\R}\p^{\alpha}(n\bar{u}_1)\p^{\alpha}E_1dy
  -\epsilon^a\int_{\R}\p^{\alpha}(n\psi)\cdot\p^{\alpha}E dy\nnm\\
&-\epsilon^{1+a}\int_{\R}\frac{\sigma(\theta)}{R\theta}\p^\alpha(u\times B)\cdot\p^{\alpha}E dy
 -\epsilon^a\int_{\R}\p^{\alpha}E\cdot\int_{\R^3}v \p^{\alpha}\L_M^{-1}\Lambda_{2}dvdy
  +\sum^3_{i=1}\intra{L}^2_{E_i}\p^{\alpha}E_idy,
\ema
where $\O$ is a $3\times 3$ matrix defined by \eqref{matrix}. By using the fact $\dtau B=-\O\dy E$, $\dy E_1=\epsilon^{a}n$ and the integration by parts, we have
\be\label{EB-1-2a}
\left\{\begin{aligned}
&\int_{\R}\O\p^\alpha\dy B\cdot\p^{\alpha}Edy
 =-\int_{\R}\O\p^\alpha B\cdot\p^{\alpha}\dy Edy
  =\int_{\R}\p^\alpha B\cdot\O\p^{\alpha}\dy Edy
  =-\frac12\frac{d}{d\tau}\int_{\R}|\p^{\alpha}B|^2dy,\\
&-\epsilon\int_{\R}\frac{\sigma(\theta)}{\rho}\p^{\alpha}\dy n\p^{\alpha}E_1dy
  =\epsilon\int_{\R}\dy\(\frac{\sigma(\theta)}{\rho}\)\p^{\alpha}n\p^{\alpha}E_1 dy
   +\epsilon^{1+a}\int_{\R}\frac{\sigma(\theta)}{\rho}|\p^{\alpha}n|^2dy,\\
&-\epsilon^a\int_{\R}\p^{\alpha}(n\bar{u}_1)\p^{\alpha}E_1dy	
=\frac12\intr \p_y \bar{u}_1|\p^{\alpha} E_1|^2dy
  -\epsilon^a\int_{\R}n\p^\alpha\bar{u}_1\p^\alpha E_1 dy.
\end{aligned}\right.
\ee
Submitting \eqref{EB-1-2a} into \eqref{EB-1-1}, we have
\bma\label{EB-1-2}
&\frac12\frac{d}{d\tau} \|\p^{\alpha}(E,B)\|^2_{L^2_y}
 +\epsilon^{1+a}\int_{\R}\frac{\sigma(\theta)}{R\theta}|\p^{\alpha}E|^2dy
 +\epsilon^{1+a}\int_{\R}\frac{\sigma(\theta)}{\rho}|\p^{\alpha}n|^2dy\nnm\\
=&\frac12\intr \p_y \bar{u}_1|\p^{\alpha} E_1|^2dy
  -\epsilon^a\int_{\R}n\p^\alpha\bar{u}_1\p^\alpha E_1 dy
  -\epsilon^a\int_{\R}\p^{\alpha}(n\psi)\cdot\p^{\alpha}E dy\nnm\\
&-\epsilon^{1+a}\int_{\R}\frac{\sigma(\theta)}{\rho}\p^\alpha(u\times B)\cdot\p^{\alpha}E dy
 -\epsilon^a\int_{\R}\p^{\alpha}E\cdot\int_{\R^3}v \p^{\alpha}\L_M^{-1}\Lambda_{2}dvdy
  +\sum^3_{i=1}\intra{L}^2_{E_i}\p^{\alpha}E_idy\nnm\\
&-\epsilon\int_{\R}\dy\(\frac{\sigma(\theta)}{\rho}\)\p^{\alpha}n\p^{\alpha}E_1 dy
:=\sum^7_{i=1}J^i_E.
\ema

For $J^1_E,J^2_E$, by Lemma \ref{rarefaction4} and $\p_y E_1=\epsilon^a n$, we have
$$
J^1_E+J^2_E
\leq C\|\p^\alpha\bar{u}_1\|_{L^\infty_y}\|\p^{\alpha}E_1\|^2_{L^2_y}
\leq C\frac{\epsilon^a}{\frac{\delta}{\sigma}+\epsilon^a\tau}\mathcal{E}(\tau).
$$	
For $J^3_E$, it holds that
\bma\label{J3E}
J^3_E
&\leq C\epsilon^a\int_{\R}|\p^\alpha n||\psi||\p^{\alpha}E|+|n||\p^\alpha \psi||\p^{\alpha}E|dy
 \leq C\epsilon^{a-1}\mathcal{E}^{\frac12}(\tau)\mathcal{D}(\tau).
\ema
For $J^4_E$, note that $B_1=0$ and $\bar{u}\times \p^{\alpha}B=(0,-\bar{u}_1\p^{\alpha}B_3,\bar{u}_1\p^{\alpha}B_2)$.
Thus, we have
\bma
J^4_E
=&-\epsilon^{1+a}\int_{\R}\frac{\sigma(\theta)}{R\theta}\p^\alpha(\bar u\times B)\cdot\p^{\alpha}E dy
 -\epsilon^{1+a}\int_{\R}\frac{\sigma(\theta)}{R\theta}\p^\alpha(\psi\times B)\cdot\p^{\alpha}E dy\nnm\\
\leq&C\epsilon^{1+a}\int_{\R}|\bar{u}_1||\p^{\alpha}B||\p^{\alpha}E|dy
  +C\epsilon^{1+a}\int_{\R}|\p^{\alpha}\bar{u}_1||B||\p^{\alpha}E|dy\nnm\\
&+C\epsilon^{1+a}\int_{\R}|\psi||\p^{\alpha}B||\p^{\alpha}E|dy
 +C\epsilon^{1+a}\int_{\R}|\p^{\alpha}\psi||B||\p^{\alpha}E|dy\nnm\\
\leq&\lambda\epsilon^{1+a}\|\p^{\alpha}E\|^2_{L^2_x}
 +C\lambda^{-1}\Big(\eta_0^2+\mathcal{E}^{\frac12}(\tau)\Big)\mathcal{D}(\tau)
 +C\frac{\epsilon^{1+2a}}{\frac{\delta}{\sigma}+\epsilon^a\tau}\mathcal{E}(\tau).\nnm
\ema
For $J^5_E$, by the definition of $\Lambda_2$ in \eqref{lambda2}, we have
\bma\label{J2E}
J^5_E
\leq&\,\lambda\epsilon^{1+a}\|\p^\alpha E\|^2_{L^2_y}
 +C\lambda^{-1}\(\sigma^2\frac{\epsilon^{2a}}{\delta^2}+\mathcal{E}(\tau)\)\mathcal{D}(\tau)\nnm\\
&+C\lambda^{-1}\frac{\epsilon^{1+3a}}{\delta^2(\frac{\delta}{\sigma}+\epsilon^a\tau)^2}\|n\|^2_{L^2_y}
 +C\lambda^{-1}\epsilon^{1-a}\sum_{|\alpha_1|=2}\|\p^{\alpha_1}\g\|^2_\sigma.
\ema
For $J^6_E$, 
by using Lemma \ref{rarefaction4}, we have
\bma\label{J6E}
J^6_E
\leq&\,C\epsilon\int_{\R}|\p^{\alpha}E||\p^\alpha(\rho,\theta)|(|\dy\rho||n|+|\p^\alpha n|)dy
 +C\epsilon\int_{\R}|\p^{\alpha}E||\p^\alpha\dy\rho||n|dy\nnm\\
&+C\epsilon^{1+a}\int_{\R}|\p^{\alpha}E||\p^\alpha\theta||E+u\times B|dy\nnm\\
\leq&\,\lambda\epsilon^{1+a}\|\p^\alpha E\|^2_{L^2_y}
 +C\lambda^{-1}\(\sigma^2\frac{\epsilon^{2a}}{\delta^2}+\mathcal{E}(\tau)\)\mathcal{D}(\tau).
\ema
For $J^7_E$, by Lemma \ref{rarefaction4}, we have
\bma
J^7_E
&\leq C\epsilon\int_{\R}|\dy(\rho,\theta)||\p^\alpha n||\p^\alpha E_1|dy\nnm\\
&\leq \lambda\epsilon^{1+a}\|\p^\alpha E_1\|^2_{L^2_y}
 +C\lambda^{-1}\epsilon^{1-a}\int_{\R}|\dy(\rho,\theta)|^2|\p^\alpha n|^2dy\nnm\\
&\leq\lambda\epsilon^{1+a}\|\p^\alpha E_1\|^2_{L^2_y}
 +C\lambda^{-1}\(\sigma^2\frac{\epsilon^{2a}}{\delta^2}+\mathcal{E}(\tau)\)\mathcal{D}(\tau).\nnm
\ema
Submitting the estimates for $J_E^i$, $i=1,2,...,7$ into \eqref{EB-1-2} and letting $\lambda$ small enough,
and using the fact
$$
\|\dtau B\|^2_{L^2_x}=\|\O\dy E\|^2_{L^2_x},
$$
we have
\bma\label{EB-first-1}
&\frac{d}{d\tau}\|\p^{\alpha}(E,B)\|^2_{L^2_y}
  +\epsilon^{1+a}\|\p^{\alpha}(E,n)\|^2_{L^2_y}
  +\epsilon^{1+a}\|\dtau B\|^2_{L^2_y}\nnm\\
\leq&\, C\frac{\epsilon^a}{\frac{\delta}{\sigma}+\epsilon^a\tau}\mathcal{E}(\tau)
 +C\(\eta^2_0+\sigma^2\frac{\epsilon^{2a}}{\delta^2}+\epsilon^{a-1}\mathcal{E}^{\frac12}(\tau)\)\mathcal{D}(\tau)
 +C\epsilon^{1-a}\sum_{|\alpha_1|=2}\|\p^{\alpha_1}\g\|^2_\sigma.
\ema

In the following, we shall estimate the dissipation of $\epsilon^{1+a}\| \dy B\|^2_{L^2_y}$ in \eqref{EB-first}.
Under the scaling \eqref{Scaling}, we rewrite \eqref{EB} as the following:
\be\label{EB-2}
\left\{\begin{aligned}
&\dtau E-\O\dy B=-\epsilon^anu-\epsilon^a\intr vG_{2}dv, \\
&\dtau B=-\O\dy E,\quad\dy E_1=\epsilon^{a}n.
\end{aligned}\right.
\ee
Taking the inner product between \eqref{EB-2}$_1$ and $-\epsilon^{1+a}\O\dy B$ and using $\eqref{EB-2}_2$, we have
\bma\label{EB-first-2}
&-\epsilon^{1+a}\frac{d}{d\tau}\int_{\R}E\cdot\O\dy B dy+\epsilon^{1+a}\|\O\dy B\|^2_{L^2_y}\nnm\\
=&\epsilon^{1+a}\|\O\dy E\|^2_{L^2_y}+\epsilon^{1+2a}\int_{\R}nu\cdot\O\dy B dy+\epsilon^{1+2a}\int_{\R}\O\dy B\cdot\intr vG_{2}dv\nnm\\
\leq&\epsilon^{1+a}\|\O\dy E\|^2_{L^2_y}+C\epsilon^{a}\mathcal{D}(\tau).
\ema
By summing of $2*\eqref{EB-first-1}$ and \eqref{EB-first-2}, we can prove Lemma \ref{EB1}.
\end{proof}

\begin{flushleft}
\textbf{Step 2. } Estimates on $\|\p^\alpha n\|_{L^2_y}$ with $|\alpha|=1$.
\end{flushleft}	
\begin{lem}\label{n1-1}
Under the same assumptions as in Proposition \ref{priori3} and by letting $\frac{\epsilon^a}{\delta}$ small enough, it holds that for $|\alpha|=1$,
\bmas
&\frac{d}{d\tau} \intra |\p^{\alpha}n|^2dy
 +\epsilon^{1-a}\|\p^{\alpha}\dy n\|^2_{L^2_y}
 +\epsilon^{1+a}\|\p^{\alpha}n\|^2_{L^2_y}
 +\|\sqrt{\dy\bar{u}_1}\p^\alpha n\|^2_{L^2_y}\\
\leq&\,C\(\epsilon^{a-1}\mathcal{E}^{\frac12}(\tau)+\epsilon^{-a}\mathcal{E}(\tau)+\sigma^2\frac{\epsilon^{2a}}{\delta^2}\)\mathcal{D}(\tau)
 +C\epsilon^{1-a}\sum_{|\alpha_1|=2}\|\p^{\alpha_1}\g\|^2_\sigma
 +C\frac{\epsilon^a}{(\frac{\delta}{\sigma}+\epsilon^a\tau)}\mathcal{E}(\tau).
 \emas
\end{lem}	
\begin{proof}		
For $\p^\alpha n$, applying $\p^\alpha$ to $\eqref{pert-sys}_6$ and using  $\p_y E_1=\epsilon^a n$ and
\bma
&\epsilon\p^{\alpha}\dy\(\frac{\sigma(\theta)}{R\theta}E_1\)
=\epsilon\p^{\alpha}\[\dy\(\frac{\sigma(\theta)}{R\theta}\)E_1+\frac{\sigma(\theta)}{R\theta}\dy E_1\]\nnm\\
=&\epsilon\[\p^{\alpha}\dy\(\frac{\sigma(\theta)}{R\theta}\)E_1+\dy\(\frac{\sigma(\theta)}{R\theta}\)\p^{\alpha}E_1
  +\p^\alpha\(\frac{\sigma(\theta)}{R\theta}\)\dy E_1\]
  +\epsilon^{1+a}\frac{\sigma(\theta)}{R\theta}\p^{\alpha}n,\nnm
\ema
we have
\bma\label{nx}
&\dtau \p^{\alpha}n+\p^{\alpha}\dy(n u_1)
   +\epsilon^{1+a}\frac{\sigma(\theta)}{R\theta}\p^{\alpha}n
   -\epsilon^{1-a}\dy\(\frac{\sigma(\theta)}{\rho}\p^{\alpha}\dy n\)\nnm\\
&\quad=\dy R^1_n+R^2_n-\intr v_1\p^{\alpha}\dy \L_M^{-1}\Lambda_2dv,
\ema
where
\bma
R^1_n&= \epsilon^{1-a}\(\p^{\alpha}\(\frac{\sigma(\theta)}{\rho}\)\dy n\)
 -\epsilon^{1-a}\p^{\alpha}\(\sigma(\theta)\frac{n}{\rho^2}\dy\rho\),\nnm\\
R^2_n&= -\epsilon\[\p^{\alpha}\dy\(\frac{\sigma(\theta)}{R\theta}\)E_1+\dy\(\frac{\sigma(\theta)}{R\theta}\)\p^{\alpha}E_1
  +\p^\alpha\(\frac{\sigma(\theta)}{R\theta}\)\dy E_1\]
  -\epsilon\p^{\alpha}\(\sigma(\theta)\frac{(u\times B)_1}{R\theta}\).\nnm
\ema
Multiplying \eqref{nx} by $\p^{\alpha}n$ and then integrating  over $\R_y$, we have
\bmas
&\frac12\frac{d}{d\tau}\intra |\p^{\alpha}n|^2dy
 +\epsilon^{1-a}\intra \frac{\sigma(\theta)}{\rho}|\p^{\alpha}\dy n|^2dy
 +\epsilon^{1+a}\intra \frac{\sigma(\theta)}{R\theta}|\p^{\alpha}n|^2dy\\
=&\,-\intra\p^{\alpha}\dy(n u_1)\p^{\alpha}n dy
 +\intra \dy R^1_n\p^{\alpha}n+R^2_n\p^{\alpha}ndy
 -\int_{\R}\p^{\alpha}n\int_{\R^3}v_1\dy\p^{\alpha}\L_M^{-1}\Lambda_2 dvdy\\
:=&\,J_{n}^1+J_{n}^2+J_{n}^3.
\emas
For $J_n^1$, by integration by parts and choosing $\frac{\epsilon^a}{\delta}$ small enough, we have
\bma\label{Jn1}
J_n^1
=&-\intra\dy\p^{\alpha}n u_1\p^{\alpha}n+\dy u_1|\p^{\alpha}n|^2+n\dy\p^{\alpha} u_1\p^{\alpha}n +\dy n\p^{\alpha}u_1\p^{\alpha}n dy\nnm\\
\leq&-\intra \frac12|\p^{\alpha}n|^2\p_y \bar{u}_1dy-\intra \frac12|\p^{\alpha}n|^2\p_y\psi_1dy\nnm\\
&+\int_{\R}|\p^\alpha\psi_1||\dy n||\p^\alpha n|+|\p^\alpha\bar{u}_1||\dy n||\p^\alpha n|
        +|n||\dy\p^\alpha\psi_1||\p^\alpha n|+|\dy\p^\alpha\bar{u}_1||n||\p^\alpha n|dy\nnm\\
\leq&-\intra \frac12|\p^{\alpha}n|^2\p_y \bar{u}_1dy
 +C\epsilon^{a-1}\mathcal{E}^{\frac12}(\tau)\mathcal{D}(\tau)
 +C\frac{\epsilon^a}{(\frac{\delta}{\sigma}+\epsilon^a\tau)}\mathcal{E}(\tau).
\ema
For $J_n^2$, by  Lemma \ref{rarefaction4}, 
we have
\bma
\begin{aligned}
&\quad \intra \dy R^1_n\p^{\alpha}ndy=-\intra R^1_n\p^{\alpha}\dy ndy\nnm\\
&\leq \lambda\epsilon^{1-a}\|\dy\p^{\alpha}n\|^2_{L^2_y}
 +C\lambda^{-1}\epsilon^{1-a}\int_{\R}|\p^\alpha(\rho,\theta)|^2|\p^\alpha n|^2
   +|n|^2|\p^\alpha\dy\rho|^2
   +|\p^\alpha(\rho,\theta)|^2|n|^2|\dy\rho|^2dy\\	
&\leq \lambda\epsilon^{1-a}\|\dy\p^{\alpha}n\|^2_{L^2_y}
 +C\lambda^{-1}\(\mathcal{E}(\tau)+\sigma^2\frac{\epsilon^{2a}}{\delta^2}\)\mathcal{D}(\tau).
\end{aligned}
\ema
And by using $(u\times B)_1=(u_2B_3-u_3B_2)=(\psi_2B_3-\psi_3B_2)$, we have
\bma
\int_{\R}R^2_n\p^{\alpha}ndy
 \leq&\, C\epsilon\int_{\R}|\p^{\alpha}n|\Big[|E_1|(|\p^\alpha\dy\theta|+|\p^\alpha\theta|^2)
    +|\p^\alpha\theta|(|\p^{\alpha}E_1| +|(u\times B)_1|)
    +|\p^\alpha(u\times B)_1|\Big]dy\nnm\\
\leq&\lambda\epsilon^{1+a}\|\p^{\alpha}n\|^2_{L^2_y}
   +C\lambda^{-1}\epsilon^{1-a}\int_{\R}|E_1|^2\Big(|\p^\alpha\dy\theta|^2+|\p^\alpha\theta|^4\Big)dy\nnm\\
&+C\lambda^{-1}\epsilon^{1-a}\int_{\R}\Big[|\p^\alpha\theta|^2(|\p^{\alpha}E_1|^2 +|(u\times B)_1|^2)
  +|\p^\alpha(u\times B)_1|^2\Big]dy\nnm\\
\leq&\,\lambda\epsilon^{1+a}\|\p^{\alpha}n\|^2_{L^2_y}
 +C\lambda^{-1}\frac{\epsilon^{1+a}}{(\frac{\delta}{\sigma}+\epsilon^a\tau)^2}\mathcal{E}(\tau)
 +C\lambda^{-1}\epsilon^{-a}\mathcal{E}(\tau)\mathcal{D}(\tau).
\ema
For $J_n^3$, by \eqref{lambda2}, we have
\bma
J_n^3
=&\,\int_{\R}\dy\p^{\alpha}n\int_{\R^3}v_1\p^{\alpha}\L_M^{-1}\Lambda_2 dvdy \nnm\\
\leq&\,\lambda\epsilon^{1-a}\|\dy\p^{\alpha}n\|^2_{L^2_y}
 +C\lambda^{-1}\frac{\epsilon^{1+a}}{(\frac{\delta}{\sigma}+\epsilon^a\tau)^2}\|\p^\alpha(E,B)\|^2_\sigma\nnm\\
&+C\lambda^{-1}\(\sigma^2\frac{\epsilon^{2a}}{\delta^2}+\mathcal{E}(\tau)\)\mathcal{D}(\tau)
 +C\lambda^{-1}\epsilon^{1-a}\sum_{|\alpha_1|=2}\|\p^{\alpha_1}\g\|^2_\sigma.\nnm
\ema
Then, combining all estimates for $J_n^i$, $i=1,2,3$ with $\lambda$ being small enough, we can complete the proof of  the lemma.
\end{proof}
\subsection{Estimates on derivatives of $\phi$}
In this subsection,  we continue to estimate  $\epsilon^{1-a}|\p^\alpha\dy\phi|^2$ and $\epsilon^{1-a}|\dtau^2(\phi,\psi,\zeta)|^2$.
For this, we apply the operator $\p^\alpha(|\alpha|=1)$ to the system \eqref{pert-sys-n} to have
\bma\label{pert-sys-n1}
\left\{\begin{aligned}
&\dtau \p^\alpha\phi+u_1\dy\p^\alpha\phi+\rho\dy\p^\alpha\psi_1=N_1+W_1,\\
&\dtau \p^\alpha\psi_1+ u_1\dy \p^\alpha\psi_1+R\dy\p^\alpha\zeta+R\frac{\theta}{\rho}\dy\p^\alpha\phi
=N_2+W_2\\
&\qquad +\p^\alpha\[\epsilon^a\frac{n}{\rho}(E_1+(u\times B)_1)-\frac{1}{\rho}\intr v_1^2\dy G_1dv
   +\frac{\epsilon^a}{\rho}\intr (v\times B)_1 G_2dv\],\\
&\dtau \p^\alpha\psi_i+ u_1\dy \p^\alpha\psi_i=W_{i+1}\\
&\qquad +\p^\alpha\[\epsilon^a\frac{n}{\rho}(E_i+(u\times B)_i)-\frac{1}{\rho}\intr v_1v_i\dy G_1dv
   +\frac{\epsilon^a}{\rho}\intr (v\times B)_i G_2dv\],~ i=2,3,\\
&\dtau \p^\alpha\zeta+ u_1\dy \p^\alpha\zeta+R\theta\dy\p^\alpha\psi_1=N_5+W_5\\
&\qquad +\p^\alpha\[\frac{1}{\rho}\intr v_1\((v\cdot u)-\frac12 |v|^2\)\dy G_1dv
   +\frac{\epsilon^a}{\rho}\intr \(E+(u\times B)\)\cdot vG_2dv\],
\end{aligned}\right.
\ema
where $N_k$ and $W_k$, $k=1,2,3,4,5$ are defined in \eqref{N1}-\eqref{W1}.
\begin{flushleft}
\textbf{Step 1.}  Estimates on $\epsilon^{1-a}\|\p^\alpha\dy\phi\|^2_{L^2_y}$.
\end{flushleft}
\begin{lem}\label{pphi}
Under the same assumptions as in Proposition \ref{priori3}, we have
\begin{align*}
&\epsilon^{1-a}\frac{d}{d\tau}\int_{\R}\p^\alpha\psi_1\p^\alpha\dy\phi dy
  +\epsilon^{1-a}\|\p^\alpha\dy\phi\|^2_{L^2_y}\nnm\\
\leq&\,C\epsilon^{1-a}\|\p^\alpha\dy(\psi_1,\zeta)\|^2_{L^2_y}
 +C\(\sigma^2\frac{\epsilon^{2a}}{\delta^2}+\mathcal{E}(\tau)\)\mathcal{D}(\tau)
 +C\epsilon^{1-a}\|\p^\alpha\dy\gt\|^2_\sigma\nnm\\
&+C\frac{\epsilon^{1+3a}}{\delta^2(\frac{\delta}{\sigma}+\epsilon^a\tau)^2}\|(\phi,\psi_1,\zeta)\|^2_{L^2_y}
 +C\frac{\epsilon^{3+2a}}{\delta^3(\frac{\delta}{\sigma}+\epsilon^a\tau)^2}.
\end{align*}
\end{lem}
\begin{proof}
Multiplying \eqref{pert-sys-n1}$_2$ by $\epsilon^{1-a}\p^\alpha\dy \phi$ and integrating over $\R_y$, we have
\bma\label{phi2}
&\epsilon^{1-a}\int_{\R}\p_\tau \p^\alpha\psi_1\p^\alpha\dy\phi dy
  +\epsilon^{1-a}\int_{\R}\frac{R\theta}{\rho}|\p^\alpha\p_y\phi|^2dy\nnm\\
=&-\epsilon^{1-a}\int_{\R}\(u_1\p^\alpha\dy\psi_1+R \p^\alpha\dy\zeta-N_2-W_2\)\p^\alpha\p_y\phi dy\nnm\\
&+\epsilon\int_{\R}\p^\alpha\[\frac{n}{\rho}(E_1+(u\times B)_1)\]\p^\alpha\p_y\phi dy
 -\epsilon^{1-a}\int_{\R}\p^\alpha\dy\phi\p^\alpha\[\frac{1}{\rho}\int_{\R^3}v_1^2 \dy G_1 dv\]dy\nnm\\
&+\epsilon\int_{\R}\p^\alpha\dy\phi\p^\alpha\[\frac{1}{\rho}\intr (v\times B)_1 G_2dv\]dy.
\ema
By \eqref{pert-sys-n1}$_1$ and integration by parts, we have
\bma\label{phipsi}
\int_{\R}\p_\tau \p^\alpha\psi_1\p^\alpha\p_y\phi dy
&=\frac{d}{d\tau}\int_{\R}\p^\alpha\psi_1\p^\alpha\p_y\phi dy
  -\int_{\R}\p^\alpha\psi_1\dy\p_\tau\p^\alpha\phi dy\nnm\\
&= \frac{d}{d\tau}\int_{\R}\p^\alpha\psi_1\p^\alpha\dy\phi dy
  +\int_{\R}\p^\alpha\dy\psi_1\p_\tau\p^\alpha\phi dy\nnm\\
&= \frac{d}{d\tau}\int_{\R}\p^\alpha\psi_1\p^\alpha\dy\phi dy
  -\int_{\R}\p^\alpha\dy\psi_1\(u_1\p^\alpha\dy\phi+\rho\p^\alpha\dy\psi_1-N_1-W_1\)dy.
\ema
Substituting \eqref{phipsi} into \eqref{phi2}, we have
\bma
&\epsilon^{1-a}\frac{d}{d\tau}\int_{\R}\p^\alpha\psi_1\p^\alpha\dy\phi dy
  +\epsilon^{1-a}\int_{\R}\frac{R\theta}{\rho}|\p^\alpha\p_y\phi|^2dy\nnm\\
=&\int_{\R}\p^\alpha\dy\psi_1\(u_1\p^\alpha\dy\phi+\rho\p^\alpha\dy\psi_1-N_1-W_1\)dy\nnm\\
&-\epsilon^{1-a}\int_{\R}\(u_1\p^\alpha\dy\psi_1+R \p^\alpha\dy\zeta-N_2-W_2\)\p^\alpha\p_y\phi dy\nnm\\
&+\epsilon\int_{\R}\p^\alpha\[\frac{n}{\rho}(E_1+(u\times B)_1)\]\p^\alpha\p_y\phi dy
 -\epsilon^{1-a}\int_{\R}\p^\alpha\dy\phi\p^\alpha\[\frac{1}{\rho}\int_{\R^3}v_1^2 \dy G_1 dv\]dy\nnm\\
&+\epsilon\int_{\R}\p^\alpha\dy\phi\p^\alpha\[\frac{1}{\rho}\intr (v\times B)_1 G_2dv\]dy.
:=\sum_{i=1}^{5}J_{\phi}^{i}.
\ema
For $J_{\phi}^1,J_{\phi}^2$, by Lemma \ref{rarefaction4} and using  $(\p^\alpha\dy\psi_1u_1\p^\alpha\dy\phi-u_1\p^\alpha\dy\psi_1\p^\alpha\dy\phi)=0$,
we have
\bma
J_{\phi}^1+J_{\phi}^2
&\leq C\epsilon^{1-a}\int_{\R}|\p^\alpha\dy\psi_1||\p^\alpha\dy\zeta|+|\p^\alpha\dy\psi_1|^2
 +|\p^\alpha\p_y(\phi,\psi_1)||\p^\alpha(\phi,\psi_1,\zeta)||\p_y(\bar{\rho},\bar{u}_1)|\nnm\\
&\quad+|\p^\alpha\dy(\phi,\psi_1)||(\phi,\psi_1,\zeta)||\p^\alpha\dy(\bar{\rho},\bar{u}_1)|
 +|\p^\alpha\dy(\phi,\psi_1)||\p^\alpha(\rho,u_1,\theta)||\dy(\phi,\psi_1)|dy\nnm\\
&\leq \lambda\epsilon^{1-a}\|\p^\alpha\dy\phi\|^2_{L^2_y}
 +C\lambda^{-1}\(\sigma^2\frac{\epsilon^{2a}}{\delta^2}+\mathcal{E}(\tau)\)\mathcal{D}(\tau)\nnm\\
&\quad+C\lambda^{-1}\epsilon^{1-a}\|\p^\alpha\dy(\psi_1,\zeta)\|^2_{L^2_y}
 +C\lambda^{-1}\frac{\epsilon^{1+3a}}{\delta^2(\frac{\delta}{\sigma}+\epsilon^a\tau)^2}\|(\phi,\psi_1,\zeta)\|^2_{L^2_y}.
\ema
For $J_{\phi}^3$, we have
\bma
J_{\phi}^3
& \leq \lambda\epsilon^{1-a}\|\p^\alpha\dy\phi\|^2_{L^2_y}
 +C\lambda^{-1}\epsilon^{1+a}\int_{\R}|\p^\alpha n|^2|(E,B)|^2+|n|^2|\p^\alpha(E,B)|^2dy\nnm\\
&\quad+C\lambda^{-1}\epsilon^{1+a}\int_{\R}|\p^\alpha(\rho,u)|^2|n|^2|(E,B)|^2dy\nnm\\
& \leq \lambda\epsilon^{1-a}\|\p^\alpha\dy\phi\|^2_{L^2_y}
 +C\lambda^{-1}\mathcal{E}(\tau)\mathcal{D}(\tau).
\ema
For $J_{\phi}^4$, by using  $G_1=\bar{G}_1+\sqrt{\mu}\gt$ and \eqref{Gb2}-\eqref{Gb3}, we have
\bma
J_{\phi}^4
& \leq\lambda\epsilon^{1-a}\|\p^\alpha\dy\phi\|^2_{L^2_y}
 +C\lambda^{-1}\epsilon^{1-a}\int_{\R}\left|\int_{\R^3}v^2_1\sqrt{\mu}\p^\alpha\dy\(\frac{\bar{G}_1}{\sqrt{\mu}}+\gt\)dv\right|^2dy\nnm\\
&\quad+C\lambda^{-1}\epsilon^{1-a}\int_{\R}|\p^\alpha\rho|^2\left|\int_{\R^3}v^2_1\sqrt{\mu}\dy\(\frac{\bar{G}_1}{\sqrt{\mu}}+\gt\)dv\right|^2dy\nnm\\
& \leq \lambda\epsilon^{1-a}\|\p^\alpha\dy\phi\|^2_{L^2_y}
 +C\lambda^{-1}\epsilon^{1-a}\|\p^\alpha\dy\gt\|^2_\sigma
 +C\lambda^{-1}\(\sigma^2\frac{\epsilon^2}{\delta^2}+\mathcal{E}(\tau)\)\mathcal{D}(\tau)
 +C\lambda^{-1}\frac{\epsilon^{3+2a}}{\delta^3(\frac{\delta}{\sigma}+\epsilon^a\tau)^2}.
\ema
Similarly, for $J_{\phi}^5$, by using  $G_2=\sqrt{\mu}\g$, we have
\bma
J_{\phi}^5
&\leq\lambda\epsilon^{1-a}\|\p^\alpha\dy\phi\|^2_{L^2_y}
 +C\lambda^{-1}\epsilon^{1+a}\int_{\R}|B|^2\left|\int_{\R^3}|v|\sqrt{\mu}\p^\alpha \g dv\right|^2dy\nnm\\
&\quad+C\lambda^{-1}\epsilon^{1+a}\int_{\R}(|\p^\alpha\rho|^2|B|^2+|\p^\alpha B|^2)\left|\int_{\R^3}|v|\sqrt{\mu} \g dv\right|^2dy\nnm\\
&\leq \lambda\epsilon^{1-a}\|\p^\alpha\dy\phi\|^2_{L^2_y}
 +C\lambda^{-1}\mathcal{E}(\tau)\mathcal{D}(\tau).
\ema
Finally, by combining the estimates of $J^i_\phi,~i=1,2,...,5$, and choosing $\lambda>0$ small enough,
we complete the proof of the lemma.
\end{proof}

\begin{flushleft}
\textbf{Step 2. } Estimate on $\epsilon^{1-a}\|\dtau^2(\phi,\psi,\zeta,n)\|^2_{L^2_y}$.
\end{flushleft}	
\begin{lem}\label{psia}
Under the same assumptions as in Proposition \ref{priori3}, 
we have
\begin{align*}
&\epsilon^{1-a}\|\dtau^2(\psi,\phi,\zeta,n)\|_{L^2_y}\nnm\\
\leq&\,C\epsilon^{1-a}\|\dtau\dy(\psi,\phi,\zeta,n)\|^2_{L^2_y}
  +C\epsilon^{1-a}\|\dtau\dy(\gt,\g)\|^2_\sigma
  +C\frac{\epsilon^{3+2a}}{\delta^3(\frac{\delta}{\sigma}+\epsilon^a\tau)^2}\nnm\\
&+C\(\sigma^2\frac{\epsilon^{2a}}{\delta^2}+\mathcal{E}(\tau)\)\mathcal{D}(\tau)
  +C\frac{\epsilon^{1+3a}}{\delta^2(\frac{\delta}{\sigma}+\epsilon^a\tau)^2}\|(\phi,\psi_1,\zeta)\|^2_{L^2_y}.
\end{align*}
\end{lem}

\begin{proof}
As $\p^\alpha:=\dtau$ in \eqref{pert-sys-n1},
multiplying the equation $\eqref{pert-sys-n1}_1$ by $\epsilon^{1-a}\dtau^2\phi$,
$\eqref{pert-sys-n1}_{i+1}$ by $\epsilon^{1-a}\dtau^2\psi_i~(i=1,2,3)$,
$\eqref{pert-sys-n1}_5$ by $\epsilon^{1-a}\dtau^2\zeta$,
and integrating  over $\R_y$,
we have
\bma
\epsilon^{1-a}\|\dtau^2\phi\|^2_{L^2_y}
 =&-\epsilon^{1-a}\int_{\R}\(u_1\dtau\dy\phi
   -\rho\dtau\dy\psi_1+N_1+W_1\)\dtau^2\phi,\nnm\\
\epsilon^{1-a}\|\dtau^2\psi\|^2_{L^2_y}
 =&-\epsilon^{1-a}\int_{\R}u_1\dtau\dy\psi\cdot\dtau^2\psi dy
   -\epsilon^{1-a}\int_{\R}\[\frac{R\theta}{\rho} \dtau\dy \phi+R\dtau\dy \zeta-N_2-W_2\]\dtau^2\psi_1 \nnm\\
 &+\epsilon^{1-a}\sum^3_{i=2}\int_{\R}W_{i+1}\dtau^2\psi_i dy
  +\epsilon\int_{\R}\dtau\[\frac{n}{\rho}(E+u\times B)\]\cdot\dtau^2\psi dy\nnm\\
 &-\epsilon^{1-a}\int_{\R}\dtau^2\psi\cdot\dtau\[\frac1\rho\intr v_1v\dy G_1dv\]dy
  +\epsilon\int_{\R}\dtau^2\psi\cdot\dtau\[\frac1\rho\intr (v\times B)G_2dv\]dy,\nnm\\
\epsilon^{1-a}\|\dtau^2\zeta\|^2_{L^2_y}
=&-\epsilon^{1-a}\int_{\R}\(u_1\dtau\dy\zeta+R\theta\dtau\dy\psi_1-N_5-W_5\)\dtau^2\zeta\nnm\\
 &+\epsilon^{1-a}\int_{\R}\dtau^2\zeta\[\frac1\rho\intr v_1\((v\cdot u)-\frac12 |v|^2\)\dy G_1dv\]dy\nnm\\
 &+\epsilon\int_{\R}\dtau^2\zeta\dtau\[\frac1\rho\intr \(E+(u\times B)\)\cdot vG_2dv\]dy.
 \ema
Moreover, taking inner product between $\dtau\eqref{macro-eq-nond}_6$ and $\epsilon^{1-a}\dtau^2 n$ to yield
$$\epsilon^{1-a}\|\dtau^2 n\|^2_{L^2_y}
= -\epsilon^{1-a}\int_{\R}\dtau\dy(nu_1)\dtau^2 ndy
 -\epsilon^{1-a}\int_{\R}\dtau^2 n\int_{\R^3}v_1 \dtau\dy G_2dvdy.
$$
By Cauchy-Schwarz inequality and Sobolev embedding inequality, one has
\begin{align*}
\epsilon^{1-a}\|\dtau^2(\psi,\phi,\zeta,n)\|^2_{L^2_y}
\leq&\, \lambda\epsilon^{1-a}\|\dtau^2(\psi,\phi,\zeta,n)\|^2_{L^2_y}
  +C\lambda^{-1}\epsilon^{1-a}\|\dtau\dy(\psi,\phi,\zeta,n)\|^2_{L^2_y}\nnm\\
&+C\lambda^{-1}\frac{\epsilon^{1+3a}}{\delta^2(\frac{\delta}{\sigma}+\epsilon^a\tau)^2}\|(\phi,\psi_1,\zeta)\|^2_{L^2_y}
 +C\lambda^{-1}\(\sigma^2\frac{\epsilon^{2a}}{\delta^2}+\mathcal{E}(\tau)\)\mathcal{D}(\tau)\nnm\\
&+C\lambda^{-1}\epsilon^{1-a}\|\dtau\dy(\gt,\g)\|^2_\sigma
 +C\lambda^{-1}\frac{\epsilon^{3+2a}}{\delta^3(\frac{\delta}{\sigma}+\epsilon^a\tau)^2}.
\end{align*}
Hence, by choosing $\lambda>0$ small enough, we complete the proof of the  lemma.
\end{proof}

\subsection{Estimates on  derivatives of $(\gt,\g)$}
Applying $\p^{\alpha}~(|\alpha|=1)$ to \eqref{g1}-\eqref{g2}, we have
\bma\label{g1a}
&\p_\tau\p^{\alpha}\gt+v_1\p_y\p^{\alpha}\gt
  +\epsilon^a\frac{1}{\sqrt{\mu}}\[(E+v\times B)\cdot\nabla_v(\sqrt{\mu}\p^{\alpha}\g)\]
  -\epsilon^{a-1} L \p^{\alpha}\gt \nnm\\
=&\frac{1}{\sqrt{\mu}}\p^{\alpha}P_0\big[v_1 \sqrt{\mu}\p_y\gt+\epsilon^a(E+v\times B)\cdot\nabla_v(\sqrt{\mu}\g)\big]\nnm\\
&-\frac1{\sqrt{\mu}}\p^{\alpha}P_1\[v_1\(\frac{|v-u|^2\p_y\zeta}{2R\theta^2}+\frac{(v-u)\p_y\psi}{R\theta}\)M\]
 +\p^{\alpha}R_{\Gamma1}+\p^{\alpha}R_{g1}+R_{d1},
\ema
and
\bma\label{g2a}
&\dtau \p^{\alpha}\g+v_1\dy \p^{\alpha}\g
  +\epsilon^a\frac{1}{\sqrt{\mu}}\[(E+v\times B)\cdot\Tdv (\sqrt{\mu}\p^{\alpha}\gt)\]
  -\epsilon^{a-1} \L \p^{\alpha}\g\nnm\\
=&\frac{1}{\sqrt{\mu}}\p^{\alpha}\[P_d (v_1 \sqrt{\mu}\p_y\g)  -\epsilon^a  (E+v\times B)\cdot\Tdv \bar{G}_1\]
  +\p^{\alpha}R_{\Gamma2}+\p^{\alpha}R_{g2}+R_{d2},
\ema
where $R_{\Gamma1},R_{\Gamma2},R_{g1},R_{g2}$ are defined in \eqref{remainder} and
\bma\label{R1d}
R_{d1}:=&-\frac1{\sqrt{\mu}}\epsilon^a\p^{\alpha}(E+v\times B)\cdot\nabla_v(\sqrt{\mu}\g), \nnm\\
R_{d2}:=&-\frac1{\sqrt{\mu}}\epsilon^a\p^{\alpha}(E+v\times B)\cdot\nabla_v(\sqrt{\mu}\gt).
\ema

\begin{lem}\label{GGal}
Under the same assumptions as in Proposition \ref{priori3} and by choosing $\sigma\frac{\epsilon^a}{\delta}$ small enough,
it holds that for $|\alpha|=1$,
\bma
&\frac{d}{d\tau}\|\p^\alpha (\gt,\g)\|^2 +\sigma_1{\epsilon}^{a-1}\|\p^\alpha (\gt,\g)\|_{\sigma}^2\nnm\\
\leq&\, C\(\eta^2_1+\mathcal{E}(\tau)+\sigma^2\frac{\epsilon^{2a}}{\delta^2}\)\mathcal{D}(\tau)
 +C\frac{\epsilon^3}{\delta(\frac{\delta}{\sigma}+\epsilon^a\tau)^2}
 +C\frac{\epsilon^{1+a}}{(\frac{\delta}{\sigma}+\epsilon^a\tau)^2}\mathcal{E}(\tau)
 +C\epsilon^{1+a}\|\p^\alpha(E,B)\|^2_{L^2_y}\nnm\\
&+C\epsilon^{1-a}\|\dy\p^\alpha(\gt,\g)\|_\sigma^2
 +C\epsilon^{1-a}\|\dy\p^\alpha(\psi,\zeta,n)\|^2_{L^2_y}
 + C\mathcal{E}^{\frac12}(\tau)\mathcal{F}_\omega(\tau),
\ema
where $\eta_1>0$ is defined in \eqref{local}.
\end{lem}

\begin{proof}
Multiplying \eqref{g1a} by $\p^{\alpha}\gt$, \eqref{g2a} by $\p^{\alpha}\g$,  and then integrating over $\R^3_v\times\R_y$, by Lemma \ref{L-weight}, we have
\bma
&\frac12\frac{d}{d\tau}\norm{\p^{\alpha}(\gt,\g)}^2+\sigma_1\epsilon^{a-1}\norms{\p^{\alpha}(\gt,\g)}^2\nnm\\
\leq&\,-\epsilon^a\int_{\R}\int_{\R^3}\frac{1}{\mu}(E+v\times B)\cdot\nabla_v\(\mu\p^{\alpha}\gt\p^{\alpha}\g \)dvdy\nnm\\
&+\int_{\R}\int_{\R^3}\frac{1}{\sqrt{\mu}}\p^{\alpha}P_0\big[v_1 \sqrt{\mu}\p_y\gt
   +\epsilon^a(E+v\times B)\cdot\nabla_v(\sqrt{\mu}\g)\big]\p^{\alpha}\gt dvdy\nnm\\
&+\int_{\R}\int_{\R^3}\frac{1}{\sqrt{\mu}}\p^{\alpha}\[P_d(v_1 \sqrt{\mu}\p_y\g)- \epsilon^a (E+v\times B)\cdot\Tdv \bar{G}_1\] \p^{\alpha}\g dvdy\nnm\\
&-\int_{\R}\int_{\R^3}\frac1{\sqrt{\mu}}\p^{\alpha}P_1\[v_1\(\frac{|v-u|^2\p_y\zeta}{2R\theta^2}
    +\frac{(v-u)\p_y\psi}{R\theta}\)M\]\p^{\alpha}\gt dvdy\nnm\\
&+\int_{\R}\int_{\R^3}(\p^{\alpha}R_{\Gamma1}+\p^{\alpha}R_{g1}+R_{d1})\p^{\alpha}\gt
    +(\p^{\alpha}R_{\Gamma2}+\p^{\alpha}R_{g2}+R_{d2})\p^{\alpha}\g dvdy\nnm\\
:=&\,I_{g}^1+I_{g}^2+I_{g}^3+I_{g}^4+I_{g}^5.
\ema
For $I_{g}^1$, by using $\nabla_v\cdot(E+v\times B)=0$ and  $v\cdot(v\times B)=0$, we have
\bma\label{Ig1}
I_g^1
&=\epsilon^a\int_{\R}\int_{\R^3}\nabla_v\(\frac{1}{\mu}\)\cdot(E+v\times B)\big(\mu\p^{\alpha}\gt\p^{\alpha}\g \big)dvdy\nnm\\
&\leq  C\epsilon^{a}\|E\|_{L^\infty_y}\|\v^{\frac12}\p^{\alpha}(\gt,\g)\|^2
\leq C\mathcal{E}^{\frac12}(\tau)\mathcal{F}_\omega(\tau).
\ema
For $I_g^2$, $I_g^3$ and $I_g^4$, by using   \eqref{pov}, we have
\bma\label{pov1}		
I_g^2+I_g^3+I_g^4
\leq&\,\lambda\epsilon^{a-1}\norm{\p^\alpha(\gt,\g)}_\sigma^2
     +C\lambda^{-1}\epsilon^{1+a}\|(E,B)\|^2_{L^\infty_y}\|\p^\alpha\g\|^2_\sigma\nnm\\
&+C\lambda^{-1}\epsilon^{1-a}\norm{\p^\alpha\dy(\gt,\g)}_\sigma^2
     +C\lambda^{-1}\epsilon^{1-a}\|\p^\alpha\dy(\psi,\zeta)\|^2_{L^2_y}\nnm\\
&+C\lambda^{-1}\epsilon^{1+a}\int_{\R}|\p^\alpha(\rho,u,\theta)|^2|(E,B)|^2|(\gt,\g)|^2_2+|\p^\alpha(E,B)|^2|(\gt,\g)|^2_\sigma dy\nnm\\
&+C\lambda^{-1}\epsilon^{1-a}\int_{\R}|\p^\alpha(\rho,u,\theta)|^2\(|\dy(\gt,\g)|^2_\sigma+|\dy(\psi,\zeta)|^2\)dy\nnm\\
&+C\lambda^{-1}\epsilon^{1+3a}\Big(\|(E,B)\|^2_{L^\infty_y}\|\mu^{-\frac12}\v^{\frac32}\p^{\alpha}\nabla_v\bar{G}_1\|^2
 +\|\p^\alpha(E,B)\|^2_{L^2_y}\||\mu^{-\frac12}\v^{\frac32}\nabla_v\bar{G}_1|_2\|^2_{L^2_y}\Big)\nnm\\
\leq&\,\lambda\epsilon^{a-1}\norm{\p^\alpha(\gt,\g)}_\sigma^2
     +C\lambda^{-1}\(\mathcal{E}(\tau)+\sigma^2\frac{\epsilon^{2a}}{\delta^2}\)\mathcal{D}(\tau)
     +C\frac{\epsilon^{3+a}}{(\frac{\delta}{\sigma}+\epsilon^a\tau)^2}\mathcal{E}(\tau)\nnm\\
&+C\lambda^{-1}\epsilon^{1-a}\norm{\p^\alpha\dy(\gt,\g)}_\sigma^2
 +C\lambda^{-1}\epsilon^{1-a}\|\p^\alpha\dy(\psi,\zeta)\|^2_{L^2_y}.
\ema
Next, we focus on $I_g^5$. For terms involving $R_{g1},R_{g2}$ defined in \eqref{remainder}, letting $\sigma\frac{\epsilon^a}{\delta}\leq1$, we use  \eqref{Gb2}-\eqref{Gb3} and \eqref{barG7-2} to have
\bma
&\int_{\R}\int_{\R^3}\p^\alpha R_{g1}\p^\alpha\gt+\p^\alpha R_{g2}\p^\alpha\g dvdy\nnm\\
\leq&\,\lambda\epsilon^{a-1}\norm{\p^\alpha(\gt,\g)}_\sigma^2
 +C\lambda^{-1}\epsilon^{1-a}\int_{\R}\sum_{|\alpha_1|=2}\left|\v^{\frac12}\frac{\p^{\alpha_1}\bar{G}_1}{\sqrt{\mu}}\right|^2
 + |\p^\alpha(\rho,u,\theta)|^2\left|\frac{\dy\bar{G}_1}{\sqrt{\mu}}\right|^2_2dy\nnm\\
&+C\lambda^{-1}\epsilon^{1-a}\|\p^\alpha\dy n\|^2_{L^2_y}
 +C\lambda^{-1}\epsilon^{1-a}\int_{\R}|n|^2\(|\p^\alpha\dy(\rho,u,\theta)|^2+|\p^\alpha(\rho,u,\theta)|^4\)
                                                  +|\p^\alpha n|^2|\p^\alpha(\rho,u,\theta)|^2dy\nnm\\
&+C\lambda^{-1}\epsilon^{1+a}\|\p^\alpha(E,B)\|^2_{L^2_y}
 +C\lambda^{-1}\epsilon^{1+a}\int_{\R}|(E,B)|^2|\p^\alpha(\rho,u,\theta)|^2dy\nnm\\
\leq&\,\lambda\epsilon^{a-1}\norm{\p^\alpha(\gt,\g)}_\sigma^2
 +C\lambda^{-1}\epsilon^{1-a}\|\p^\alpha\dy n\|^2_{L^2_y}
 +C\lambda^{-1}\epsilon^{1+a}\|\p^\alpha(E,B)\|^2_{L^2_y}
 +C\lambda^{-1}\frac{\epsilon^{1+3a}}{(\frac{\delta}{\sigma}+\epsilon^a\tau)^2}\|(E,B)\|^2_{L^2_y}\nnm\\
&+C\lambda^{-1}\frac{\epsilon^{3+2a}}{\delta^3(\frac{\delta}{\sigma}+\epsilon^a\tau)^2}
  +C\lambda^{-1}\(\sigma^2\frac{\epsilon^{2a}}{\delta^2}+\mathcal{E}(\tau)\)\mathcal{D}(\tau)
  +C\lambda^{-1}\frac{\epsilon^{1+3a}}{\delta^2(\frac{\delta}{\sigma}+\epsilon^a\tau)^2}\|n\|^2_{L^2_y}.
\ema
For terms involving $R_{\Gamma_1}$ and $R_{\Gamma_2}$ defined in \eqref{remainder}, recalling that $F_2=\frac n\rho M+\sqrt{\mu}\g,G_1=\bar{G}_1+\sqrt{\mu}\gt$ and using Lemma \ref{GMG}, Lemma \ref{QGG}, we have
\bma
&\int_{\R}\int_{\R^3}\p^{\alpha}R_{\Gamma1}\p^{\alpha}\gt+\p^{\alpha}R_{\Gamma2}\p^{\alpha}\g dvdy\nnm\\
\leq&\,\lambda\epsilon^{a-1}\|\p^{\alpha}(\gt,\g)\|^2_{\sigma}
 +C\lambda^{-1}\(\eta^2_1+\mathcal{E}(\tau)+\sigma^2\frac{\epsilon^{2a}}{\delta^2}\)\mathcal{D}(\tau)\nnm\\
&+C\lambda^{-1}\frac{\epsilon^3}{\delta(\frac{\delta}{\sigma}+\epsilon^a\tau)^2}
 +C\lambda^{-1}\frac{\epsilon^{1+a}}{(\frac{\delta}{\sigma}+\epsilon^a\tau)^2}\|n\|^2_{L^2_y}.
\ema			
Then we estimate terms involving $R_{d1}$ and $R_{d2}$ defined in \eqref{R1d} as follow:
\bma
&\quad\int_{\R}\int_{\R^3}R_{d1}\p^\alpha\gt+R_{d2}\p^\alpha\g dvdy\nnm\\
&=\epsilon^a\int_{\R}\int_{\R^3}\frac1{\sqrt{\mu}}\p^{\alpha}(E+v\times B)
   \cdot\nabla_v\sqrt{\mu}(\g\p^{\alpha}\gt+\gt\p^{\alpha}\g) dvdy\nnm\\
&\quad+\epsilon^a\int_{\R}\int_{\R^3}\p^{\alpha}(E+v\times B)
   \cdot(\nabla_v\g\p^{\alpha}\gt+\nabla_v\gt\p^{\alpha}\g) dvdy\nnm\\
&\leq C\epsilon^a\left\| |\v^2\v^{-\frac12} (|(\gt,\g)|+|\nabla_v(\gt,\g)| )|_2\right\|_{L^\infty_y}
 \|\p^{\alpha}(E,B)\|_{L^2_y}\|\v^{-\frac12}\p^\alpha(\gt,\g)\|\nnm\\
&\leq C\epsilon^a\|\p^{\alpha}(E,B)\|_{L^2_y} \|\p^\alpha(\gt,\g)\|_\sigma
   \sum_{|\alpha|\leq 1,|\beta|\leq 1}\|\p^{\alpha}_{\beta}(\gt,\g)(\tau)\|_{\sigma,\omega}\nnm\\
&\leq C\epsilon\mathcal{E}^{\frac12}(\tau)\mathcal{D}(\tau).
\ema
Then combining all estimates for $I_g^i$ $(i=1,2,...,5)$ and letting $\lambda$ small enough, we complete the proof of the lemma.
\end{proof}

\subsection{The highest order derivative  on   $(F_1,F_2)$}
In this subsection, we shall estimate the highest order derivatives of $(F_1,F_2)$. By \eqref{L12}, the system \eqref{VML3} can be rewritten as
\bma\label{F11}
&\quad\dtau\(\frac{F_1}{\sqrt{\mu}}\)
 +v_1\dy\(\frac{F_1}{\sqrt{\mu}}\)
 +\epsilon^a(E+v\times B)\cdot \frac{\Tdv F_2}{\sqrt{\mu}} \nnm\\
&=\epsilon^{a-1}\[ L \gt+\Gamma\(\gt,\frac{M-\mu}{\sqrt{\mu}}\)
 +\Gamma\(\frac{M-\mu}{\sqrt{\mu}},\gt\)+\Gamma\(\frac{G_1}{\sqrt{\mu}},\frac{G_1}{\sqrt{\mu}}\)
 +\frac{L_M\bar{G}_1}{\sqrt{\mu}}\],
\ema
and
\bma \label{F21}
&\quad\dtau\(\frac{F_2}{\sqrt{\mu}}\)
 +v_1\dy \(\frac{F_2}{\sqrt{\mu}}\)
 +\epsilon^a(E+v\times B)\cdot \frac{\Tdv F_1}{\sqrt{\mu}}\nnm\\
& =\epsilon^{a-1}\[ \L \g+\Gamma\(\g,\frac{M-\mu}{\sqrt{\mu}}\)
  +\Gamma\(\frac{F_2}{\sqrt{\mu}},\frac{G_1}{\sqrt{\mu}}\)\].
\ema
Applying $\p^{\alpha}$, $|\alpha|=2$ to \eqref{F11} and \eqref{F21}, we have
\bma\label{F1-1}
&\dtau\(\frac{\dya F_1}{\sqrt{\mu}}\)+v_1\dy \(\frac{\dya F_1}{\sqrt{\mu}}\)
 +\epsilon^a(E+v\times B)\cdot \frac{\Tdv\dya F_2}{\sqrt{\mu}}\nnm\\
=&\epsilon^{a-1}\[ L \dya\gt+\dya\Gamma\(\gt,\frac{M-\mu}{\sqrt{\mu}}\)
 +\dya\Gamma\(\frac{M-\mu}{\sqrt{\mu}},\gt\)+\dya\Gamma\(\frac{G_1}{\sqrt{\mu}},\frac{G_1}{\sqrt{\mu}}\)
 +\dya\frac{L_M\bar{G}_1}{\sqrt{\mu}}\]\nnm\\
&-\epsilon^a\sum_{1\leq|\alpha_1|\leq |\alpha|}C_{\alpha}^{\alpha_1}\p^{\alpha_1}(E+v\times B)\cdot
 \frac{\Tdv\p^{\alpha-\alpha_1} F_2}{\sqrt{\mu}} ,
  \ema
and
\bma\label{F2-1}
&\dtau\(\frac{\dya F_2}{\sqrt{\mu}}\)
 +v_1\dy \(\frac{\dya F_2}{\sqrt{\mu}}\)
 +\epsilon^a(E+v\times B)\cdot \frac{\Tdv\dya F_1}{\sqrt{\mu}} \nnm\\
=&\epsilon^{a-1}\[ \L \dya\g+\dya\Gamma\(\g,\frac{M-\mu}{\sqrt{\mu}}\)
 +\dya\Gamma\(\frac{F_2}{\sqrt{\mu}},\frac{G_1}{\sqrt{\mu}}\)\]\nnm\\
&-\epsilon^a\sum_{1\leq|\alpha_1|\leq |\alpha|}C_{\alpha}^{\alpha_1}\p^{\alpha_1}(E+v\times B)\cdot
 \frac{\Tdv\p^{\alpha-\alpha_1} F_1}{\sqrt{\mu}}.
\ema
 We have the following lemma.
\begin{lem}\label{Fa}
Under the same assumptions as in Proposition \ref{priori3} and by letting $\frac{\epsilon^a}{\delta}$ small,
it holds that for $|\alpha|=2$,
\bma
&\epsilon^{2-2a}\frac{d}{d\tau}\|\mu^{-\frac12}\p^\alpha(F_1,F_2)\|^2
 +\sigma_1\epsilon^{1-a}\|\p^\alpha (\gt,\g)\|_{\sigma}^2\nnm\\
& +\epsilon^{2-2a}\frac{d}{d\tau}\intra \frac1{R\theta}(|\dya E|^2+2\dya E\cdot u\times \dya B+|\dya B|^2)dy\nnm\\
\leq&\, \lambda\epsilon^{1-a}\|\p^\alpha(\phi,\psi,\zeta,n)\|^2_{L^2_y}
  +C\lambda^{-1}\(\eta^2_1+\sigma^2\frac{\epsilon^{2a}}{\delta^2}+\mathcal{E}(\tau)\)\mathcal{D}(\tau)
 + C\mathcal{E}^{\frac12}(\tau)\mathcal{F}_\omega(\tau)\nnm\\
&+C\lambda^{-1}\frac{\epsilon^{1+2a}}{\delta(\frac{\delta}{\sigma}+\epsilon^a\tau)^2}
 +C\lambda^{-1}(1+T)\epsilon^{a-1}\mathcal{E}(\tau)\mathcal{D}(\tau)
 +C\frac{\epsilon^a}{\frac{\delta}{\sigma}+\epsilon^a\tau}\mathcal{E}(\tau),\nnm
\ema
where $\lambda>0$ is a small constant, and $\eta_1>0$ is defined in \eqref{local}.
\end{lem}

\begin{proof}
Taking the inner product of \eqref{F1-1} with $\frac{\dya F_1}{\sqrt{\mu}}$,
and taking the inner product of \eqref{F2-1} with $\frac{\dya F_2}{\sqrt{\mu}}$, by
using $\nabla_v\cdot(E+v\times B)=0$, $v\cdot(v\times B)=0$,
and $F_1=M+\bar{G}_1+\sqrt{\mu}\gt,~F_2=\frac n\rho M+\sqrt{\mu}\g$, we obtain
\bmas
&\frac12\frac{d}{d\tau}\|\frac{1}{\sqrt{\mu}}\dya (F_1,F_2)\|^2
 +\sigma_1\epsilon^{a-1}\|\dya(\gt,\g)\|^2_\sigma\\
\leq&\,\epsilon^{a}\int_{\R}\int_{\R^3} E\cdot v\frac{\dya F_2\dya F_1}{\mu}dydv
 +\epsilon^{a-1}\int_{\R}\int_{\R^3}\frac{\dya F_1}{\mu}\dya{L_M\bar{G}_1}dvdy\\
&-\epsilon^{a}\sum_{1\leq|\alpha_1|\leq|\alpha|}C_{\alpha}^{\alpha_1}
  \int_{\R}\int_{\R^3} \frac{1}{\mu} \p^{\alpha_1}(E+v\times B)\cdot
     (\Tdv\pt^{\alpha-\alpha_1} F_2 \dya F_1+\Tdv\pt^{\alpha-\alpha_1} F_1 \dya F_2)dvdy\\
&+\epsilon^{a-1}\int_{\R}\int_{\R^3}\frac{(\dya M+\dya \bar{G}_1)}{\sqrt{\mu}} L  \dya\gt
 +\frac{\dya\(\frac n\rho M\)}{\sqrt{\mu}} \L  \dya\g dvdy+I_{f,\Gamma}\\
:=&\,I_f^1+I_f^2+I_f^3+I_f^4+I_{f,\Gamma},
\emas
where
\bma
I_{f,\Gamma}
=&\int_{\R}\int_{\R^3}\dya\[\Gamma\(\gt,\frac{M-\mu}{\sqrt{\mu}}\)
  +\Gamma\(\frac{M-\mu}{\sqrt{\mu}},\gt\)
  +\Gamma\(\frac{G_1}{\sqrt{\mu}},\frac{G_1}{\sqrt{\mu}}\)\]\frac{\dya F_1}{\sqrt{\mu}}dvdy\nnm\\
&+\int_{\R}\int_{\R^3}\dya\[\Gamma\(\g,\frac{M-\mu}{\sqrt{\mu}}\)
  +\Gamma\(\frac{F_2}{\sqrt{\mu}},\frac{G_1}{\sqrt{\mu}}\)\]\frac{\dya F_2}{\sqrt{\mu}}dvdy,
\ema
which can be  estimated in Lemma \ref{QMGF} and Lemma \ref{GGF}.

For $I_{f}^1$, we have
\be\label{If1}
I_{f}^1
\leq C\epsilon^a\|E\|_{L^\infty_y}\|\v^{\frac12}\mu^{-\frac12}\dya(F_1,F_2)\|^2
\leq  C\epsilon^{2a-2}\mathcal{E}^{\frac12}(\tau)\mathcal{F}_\omega(\tau).
 \ee
For $I_f^2$, by Lemma \ref{rarefaction4}, \eqref{barG1-1} and $P_1=I-P_0$, we have
\bma
I_f^2
&= \epsilon^{a-1}\int_{\R}\int_{\R^3}\frac{\dya F_1}{\mu}\dya{L_M\bar{G}_1}dvdy\nnm\\
&= \int_{\R}\int_{\R^3}\p^\alpha P_1\[v_1M\(\frac{|v-u|^2\dy\bar{\theta}}{2R\theta^2}+\frac{(v-u)\dy\bar{u}}{R\theta}\)\]
   \frac{\p^\alpha F_1}{\mu}dvdy\nnm\\
&\leq C\bigg\|\v^{\frac12}\frac{1}{\sqrt{\mu}}\p^\alpha P_1\[v_1M\(\frac{|v-u|^2\dy\bar{\theta}}{2R\theta^2}
                               +\frac{(v-u)\dy\bar{u}}{R\theta}\)\]\bigg\|
\bigg\|\v^{-\frac{1}{2}}\frac{\p^\alpha F_1}{\sqrt{\mu}}\bigg\|\nnm\\
&\leq \lambda\epsilon^{a-1}\|\mu^{-\frac12}\p^\alpha F_1\|^2_\sigma
 +C\lambda^{-1}\frac{\epsilon^{1+4a}}{\delta^3(\frac{\delta}{\sigma}+\epsilon^a\tau)^2}
 +C\lambda^{-1}\sigma^2\frac{\epsilon^{2a}}{\delta^2}\mathcal{D}(\tau).
\ema
Next, for $I_f^3$, as $ |\alpha_1|=1$, by using  $F_1=M+\bar{G}_1+\sqrt{\mu}\gt$ and $F_2=\frac{n}{\rho}M+\sqrt{\mu}\g$, we have
\bma
I_{f}^{3}
&= -2\epsilon^{a}
  \int_{\R}\int_{\R^3} \frac{1}{\mu} \p^{\alpha_1}(E+v\times B)\cdot
     (\Tdv\pt^{\alpha-\alpha_1} F_2 \dya F_1+\Tdv\pt^{\alpha-\alpha_1} F_1 \dya F_2)dvdy\nnm\\
&\leq C\epsilon^a\|\p^{\alpha_1}(E,B)\|_{L^\infty_y}					\|\v^2\v^{-\frac12}\mu^{-\frac12}\nabla_v\p^{\alpha-\alpha_1}(F_1,F_2)\| \|\v^{-\frac12}\mu^{-\frac12}\p^{\alpha}(F_1,F_2)\|\nnm\\
&\leq \lambda\epsilon^{a-1}\|\mu^{-\frac12}\p^{\alpha}(F_1,F_2)\|^2_\sigma
 +C\lambda^{-1}\epsilon^{1+a}\|\p^{\alpha_1}(E,B)\|^2_{L^\infty_y}\|\v^2\mu^{-\frac12}\nabla_v\p^{\alpha-\alpha_1}(F_1,F_2)\|^2_\sigma\nnm\\
&\leq \lambda\epsilon^{a-1}\|\mu^{-\frac12}\p^{\alpha}(F_1,F_2)\|^2_\sigma
 +C\lambda^{-1}\epsilon^{3a-1}\mathcal{E}(\tau)\mathcal{D}(\tau)
 +C\lambda^{-1}\frac{\epsilon^{1+2a}}{\frac{\delta}{\sigma}+\epsilon^a\tau}\|\p^{\alpha_1}(E,B)\|^2_{L^\infty_y}.
\ema
For $ |\alpha_1|=2$, by noting $\nabla_v M=-\frac{v-u}{R\theta}M$, we have
\bma\label{If-3}
I_f^{3}
=&\,\epsilon^a\int_{\R}\int_{\R^3} \frac{\dya (E+v\times B)\cdot(v-u){M} \dya F_2}{R\theta{M}}dvdy
-\epsilon^a\int_{\R}\int_{\R^3}  \dya (E+v\times B)\cdot\Tdv M\dya F_2\big(\frac{1}{\mu}-\frac{1}{{M}}\big)dvdy\nnm\\
&-\epsilon^a\int_{\R}\int_{\R^3} \frac{\dya (E+v\times B)\cdot\Tdv G_1\dya F_2}{\mu}dvdy
 -\epsilon^a\int_{\R}\int_{\R^3} \frac{\p^{\alpha}(E+v\times B)\cdot\Tdv F_2 \dya F_1}{\mu}dvdy \nnm\\
:=&\,I_f^{31}+I_f^{32}+I_f^{33} +I_f^{34}.
\ema
Since
$
(v\times B)\cdot v=0,~(v\times B)\cdot u=-(u\times B)\cdot v$ and $ \dy E_1=\epsilon^a n,
$
then by \eqref{VML3} we have
\bma
I_f^{31}
=&\epsilon^a\int_{\R}\int_{\R^3} \frac1{R\theta}(u\times\dya B)\cdot v\dya F_2dvdy
  +\epsilon^a\int_{\R}\int_{\R^3} \frac1{R\theta}{\dya E\cdot(v-u)\dya F_2}dydv\nnm\\
=&\intra \frac1{R\theta} (\dya E+u\times\dya B)\cdot\[\epsilon^a\intr v\dya F_2dv\]dy
  -\intra \frac1{R\theta}\dya E\cdot u\[\epsilon^a\dya\intr F_2dv\] dy\nnm\\
=& -\intra \frac1{R{\theta}} (\dya E+u\times\dya B)\cdot\[\dya(\dtau E-\O\dy B)\]dy
  -\int_{\R}\frac{u\cdot\dya E}{R\theta}\dya\dy E_1 dy.
\ema
By using $\pt_{\tau}B=-\O\dy E$ in \eqref{VML3} and $E\cdot(\O B)=-(\O E)\cdot B$, we have
\bma
-\intra \frac1{R\theta}\dya E\dtau \dya Edy
=&-\frac12\frac{d}{d\tau}\intra \frac1{R\theta}|\dya E|^2dy+\frac12\intra\dtau\(\frac1{R\theta}\)|\dya E|^2dy,\nnm\\
-\intra \frac1{R\theta} (u\times \dya B)\cdot\dtau \dya Edy
=&-\frac{d}{d\tau}\intra \frac1{R\theta} (u\times \dya B)\cdot \dya Edy
  +\intra \frac1{R\theta} (u\times \dya\dtau B)\cdot \dya Edy\nnm\\
&+\intra \[\dtau\(\frac{u}{R\theta}\)\times \dya B\]\cdot \dya Edy,\nnm\\
\intra \frac1{R\theta}\dya E\cdot(\O\dya\dy B)dy
=&\intra \dy\(\frac1{R\theta}\)\O\dya E\cdot\dya Bdy
 -\intra \frac1{R\theta}\dya\dtau B\cdot\dya Bdy\nnm\\
=&-\frac12\frac{d}{d\tau}\intra \frac1{R\theta}|\dya B|^2dy
 +\intra \dtau\(\frac1{2R\theta}\)|\dya B|^2dy\nnm\\
&+\intra \dy\(\frac1{R\theta}\)\O\dya E\cdot\dya Bdy,\nnm\\
\intra \frac1{R\theta} (u\times \dya B)\cdot\O\dya\dy Bdy
=&\int_{\R}\frac{u_1}{R\theta}\sum^3_{i=2}\dya B_i\dya\dy B_idy
=-\frac12\int_{\R}\dy\(\frac{u_1}{R\theta}\)\sum^3_{i=2}|\dya B_i|^2dy,\nnm
\ema
and
\bma
&-\int_{\R}\frac{u\cdot\dya E}{R\theta}\dya\dy E_1 dy
 +\intra \frac1{R\theta} (u\times \dya\dtau B)\cdot \dya Edy\nnm\\
=&-\int_{\R}\frac{u\cdot\dya E}{R\theta}\dya\dy E_1 dy
 -\intra \frac1{R\theta} (u\times \O\dya\dy E)\cdot \dya Edy\nnm\\
=&\int_{\R}\dy\(\frac{u_1}{R\theta}\)|\dya E_1|^2dy
 -\int_{\R}\dy\(\frac{u_1}{R\theta}\)(|\dya E_2|^2+|\dya E_3|^2)dy\nnm\\
&+\int_{\R}\dy\(\frac{u_2}{R\theta}\)\dya E_1\dya E_2+\dy\(\frac{u_3}{R\theta}\)\dya E_1\dya E_3dy.\nnm
\ema
It follows that
\bma
I_f^{31}
\leq&-\frac12\frac{d}{d\tau}\intra \frac1{R\theta}(|\dya E|^2+2\dya E\cdot u\times \dya B+|\dya B|^2)dy
 +C\sum_{|\alpha_1|=1}\int_{\R}|\p^{\alpha_1}(u,\theta)||\dya(E,B)|^2dy\nnm\\
\leq&-\frac12\frac{d}{d\tau}\intra \frac1{R\theta}(|\dya E|^2+2\dya E\cdot u\times \dya B+|\dya B|^2)dy\nnm\\
&+C\frac{\epsilon^a}{\frac{\delta}{\sigma}+\epsilon^a\tau}\|\dya(E,B)\|^2_{L^2_y}
 + C(1+T)\epsilon^{3a-3}\mathcal{E}(\tau)\mathcal{D}(\tau).
\ema
For $I_f^{32}$, we have for $\epsilon(1+T)\leq 1$ and $\frac\delta\sigma\leq 1$ that
\bma
I_f^{32}
&\leq C\eta_0\epsilon^a\intra |\dya(E,B)||\mu^{-\frac12}\v^{-\frac12}\dya F_2|_2dy\nnm\\
&\leq C\eta_0(\epsilon^{1+a}\|\dya(E,B)\|^2_{L^2_y} +\epsilon^{a-1}\|\mu^{-\frac12}\dya F_2\|^2_\sigma)\nnm\\
&\leq C\eta_0\epsilon(1+T)\frac{\epsilon^a}{\frac\delta\sigma+\epsilon^a\tau}\|\dya(E,B)\|^2_{L^2_y}
  +C\eta_0\epsilon^{a-1}\|\mu^{-\frac12}\dya F_2\|^2_\sigma\nnm\\
&\leq C \eta_0\frac{\epsilon^a}{\frac\delta\sigma+\epsilon^a\tau}\mathcal{E}(\tau)
 +C\eta_0\epsilon^{a-1}\|\mu^{-\frac12}\dya F_2\|^2_\sigma.
\ema
Next, for $I_f^{33}$, by using  $F_2=\frac{n}{\rho}M+\sqrt{\mu}\g$, $G_1=\bar{G}_1+\sqrt{\mu}\gt$ and \eqref{barG7-3}, we have
\bmas
I_f^{33}
&\le \epsilon^a\intra |\dya(E,B)||\mu^{-\frac12}\v^2\v^{-\frac12}\Tdv G_1|_2|\mu^{-\frac12}\v^{-\frac12}\dya F_2|_2dy\\
&\le C\epsilon^a\|\dya(E,B)\|_{L^2_y}\| \||\mu^{-\frac12}\v^2\Tdv G_1|_\sigma\|_{L^\infty_y}\|\mu^{-\frac12}\dya F_2\|_\sigma\\				
&\le C\lambda^{-1}\frac{\epsilon^{1+3a}}{(\frac{\delta}{\sigma}+\epsilon^a\tau)^2}\|\dya(E,B)\|_{L^2_y}^2
  +C\epsilon^{2a}\mathcal{E}(\tau)\mathcal{D}(\tau)
  +\lambda\epsilon^{a-1}\|\mu^{-\frac12}\dya F_2\|^2_\sigma.
\emas
Similarly, it holds that
\bmas
I_f^{34}
&\le C\epsilon^a\|\dya(E,B)\|_{L^2_y}\| \||\mu^{-\frac12}\Tdv F_2|_{\sigma,\omega}\|_{L^\infty_y}\|\mu^{-\frac12}\dya F_1\|_\sigma\\
&\le C\lambda^{-1}\epsilon^{1+a}\|\dya(E,B)\|^2_{L^2_y}\| \||\mu^{-\frac12}\Tdv F_2|_{\sigma,\omega}\|^2_{L^\infty_y}
  +\lambda\epsilon^{a-1}\|\mu^{-\frac12}\dya F_1\|^2_\sigma\\				
&\le C\lambda^{-1}\epsilon^{2a-2}\mathcal{E}(\tau)\mathcal{D}(\tau)
 +\lambda\epsilon^{a-1}\|\mu^{-\frac12}\dya F_1\|^2_\sigma.
\emas	
For $I^4_f$ and $|\alpha|=2$, we have
\bma\label{M1-2}
\partial^{\alpha}M
=&\mu\left\{\frac{\partial^{\alpha}\rho}{\rho}
  +\frac{(v-u)\cdot\partial^{\alpha}u}{R\theta}
  +\(\frac{|v-u|^2}{2R\theta}-\frac32\)\frac{\partial^{\alpha}\theta}{\theta}\right\}\nnm\\
&+(M-\mu)\left\{\frac{\partial^{\alpha}\rho}{\rho}
  +\frac{(v-u)\cdot\partial^{\alpha}u}{R\theta}
  +\(\frac{|v-u|^2}{2R\theta}-\frac32\)\frac{\partial^{\alpha}\theta}{\theta}\right\}\nnm\\
 &+\sum_{|\alpha_1|=1}C_{\alpha}^{\alpha_1}\bigg\{\partial^{\alpha_1}\(\frac M\rho\)\partial^{\alpha-\alpha_1}\rho
  +\partial^{\alpha_1}\(M\frac{v-u}{R\theta}\)\cdot\partial^{\alpha-\alpha_1}u\nnm\\
 &+\partial^{\alpha_1}\(M\frac{|v-u|^2}{2R\theta^2}-M\frac{3}{2\theta}\)\partial^{\alpha-\alpha_1}\theta\bigg\}
:=J^\alpha_1+J^\alpha_2+J^\alpha_3.
\ema
Since $\mu^{-\frac12}J^\alpha_1\in {\rm ker} L $, it follows that
$( L \gt, \mu^{-\frac12}J^\alpha_1 )=0$.
From \eqref{L12}, \eqref{norm-mu}, Lemmas \ref{Gamma}-\ref{Gb} and Lemma \ref{rarefaction4}, we have
\bma
I_{f}^4
\leq&\,\lambda\epsilon^{a-1}\|\p^\alpha(\gt,\g)\|^2_{\sigma}
  +C\lambda^{-1}\eta^2_1\epsilon^{a-1}\|\p^\alpha(\rho,u,\theta,n)\|^2_{L^2_y}\nnm\\
&+C\lambda^{-1}\epsilon^{a-1}\sum_{|\alpha_1|=1}\|\p^{\alpha_1}(\rho,u,\theta,n)\|^4_{L^4_y}
  +C\lambda^{-1}\epsilon^{a-1}\|\mu^{-\frac12}\p^\alpha\bar{G}_1\|^2\nnm\\
\leq&\, \lambda\epsilon^{a-1}\|\p^\alpha(\gt,\g)\|^2_{\sigma}
 +C\lambda^{-1}\eta^2_1\epsilon^{a-1}\|\p^\alpha(\phi,\psi,\zeta,n)\|^2_{L^2_y}\nnm\\
&+C\lambda^{-1}\epsilon^{2a-2}\(\sigma^2\frac{\epsilon^{2a}}{\delta^2}+\mathcal{E}(\tau)\)\mathcal{D}(\tau)
 +C\lambda^{-1}\frac{\epsilon^{4a-1}}{\delta(\frac{\delta}{\sigma}+\epsilon^a\tau)^2}.\nnm
\ema
In addition, for the terms $\|\mu^{-\frac12}\p^\alpha F_i\|^2_{\sigma}~(i=1,2)$, by using $F_1=M+\bar{G}_1+\sqrt{\mu}\gt$, $F_2=\frac{n}{\rho}M+\sqrt{\mu}\g$ and \eqref{Gb3},  \eqref{norm-m},
we obtain
\bma
\epsilon^{1-a}\|\mu^{-\frac12}\p^\alpha F_1\|^2_{\sigma}
\leq&\, \epsilon^{1-a}\|\p^\alpha \gt\|^2_{\sigma}
 +\epsilon^{1-a}\|\mu^{-\frac12}\p^\alpha\bar{G}\|^2_{\sigma}+\epsilon^{1-a}\|\mu^{-\frac12}\p^\alpha M\|^2_{\sigma}\nnm\\
\leq&\,\epsilon^{1-a}(\|\p^\alpha \gt\|^2_{\sigma} +C \|\p^\alpha(\phi,\psi,\zeta)\|^2_{L^2_y})
 +C\(\mathcal{E}(\tau)+\sigma^2\frac{\epsilon^{2a}}{\delta^2}\)\mathcal{D}(\tau)
 +C\frac{\epsilon^{1+2a}}{\delta(\frac{\delta}{\sigma}+\epsilon^a\tau)^2},\label{F1-2order}
 \\
\epsilon^{1-a}\|\mu^{-\frac12}\p^\alpha F_2\|^2_{\sigma}
\leq&\, \epsilon^{1-a}\|\p^\alpha \g\|^2_{\sigma} +\epsilon^{1-a}\|\mu^{-\frac12}\p^\alpha (\frac{n}{\rho}M)\|^2_{\sigma}\nnm\\
\leq&\,\epsilon^{1-a}(\|\p^\alpha \g\|^2_{\sigma}+C \|\p^\alpha n\|^2_{L^2_y})
 + C\(\mathcal{E}(\tau)+\sigma^2\frac{\epsilon^{2a}}{\delta^2}\)\mathcal{D}(\tau)
 + C\frac{\epsilon^{1+2a}}{\delta(\frac{\delta}{\sigma}+\epsilon^a\tau)^2}\|n\|^2_{L^\infty_y}.\label{F2-2order}
\ema
Then combining all the above estimates of $I_{f}^i~(i=1,2,3,4)$, multiplying it by $\epsilon^{2-2a}$ and using \eqref{F1-2order}-\eqref{F2-2order},
the lemma follows by choosing $\lambda$ small enough.
\end{proof}
Combining Lemma \ref{mac1}-Lemma \ref{Fa} and by letting $\sigma\frac{\epsilon^a}{\delta},\lambda$ small enough,
we obtain the conclusion in Lemma \ref{lolem1}.

\section{Weighted energy estimates}\setcounter{equation}{0}
\subsection{Weighted estimates on space-time derivatives of  $(\gt,\g)$ }
\begin{lem}\label{GGa2}
Under the same assumptions as in Proposition \ref{priori3}, with 
$\omega=\omega(\alpha,0)$ defined in \eqref{weight}, it holds that for $|\alpha|\leq1$,
\bma
&\frac{d}{d\tau}{ \sum_{|\alpha|\le 1}}\(\|\p^\alpha (\gt,\g)\|^2_\omega
 +\frac{\epsilon^a q_1q_2}{(1+\epsilon^a\tau)^{1+q_1}}\|\v\p^{\alpha}(\gt,\g)\|_\omega^2
 +\epsilon^{a-1}\|\p^\alpha (\gt,\g)\|_{\sigma,\omega}^2\)\nnm\\
\leq&\, C\(\eta^2_1+\mathcal{E}(\tau)+\sigma^2\frac{\epsilon^{2a}}{\delta^2}\)\mathcal{D}(\tau)
 +C\frac{\epsilon^3}{\delta(\frac{\delta}{\sigma}+\epsilon^a\tau)^2}
 + C\mathcal{E}^{\frac12}(\tau)\mathcal{F}_\omega(\tau)
  +C\frac{\epsilon^{1+a}}{(\frac{\delta}{\sigma}+\epsilon^a\tau)^2}\mathcal{E}(\tau)\nnm\\
&+ C\epsilon^{1+a}\sum_{|\alpha'|=1}\|\p^{\alpha'}(E,B)\|^2_{L^2_y}
 +C\epsilon^{1+a}\|E+u\times B\|^2_{L^2_y}\nnm\\
&+C\epsilon^{1-a}{ \sum_{|\alpha|\le 1}}\(\|\dy\p^\alpha(\gt,\g)\|_{\sigma,\omega}^2
  + \|\dy\p^\alpha(\psi,\zeta,n)\|^2_{L^2_y}\)
  +C\epsilon^{a-1}{ \sum_{|\alpha|\le 1}}\|\p^{\alpha}(\gt,\g)\|_{\sigma}^2,\nnm
\ema
where $\eta_1>0$ is defined in \eqref{local}.
\end{lem}
\begin{proof}
Multiplying \eqref{g1a} by $\omega^2(\alpha,0)\p^{\alpha}\gt$, \eqref{g2a} by $\omega^2(\alpha,0)\p^{\alpha}\g$ with $|\alpha|\leq1$, then integrating  over $\R_y\times\R^3_v$,  by  Lemma \ref{L-weight}, we have
\bma
&\frac12\frac{d}{d\tau}\|\p^{\alpha}(\gt,\g)\|^2_\omega
  +\frac{\epsilon^a q_1q_2}{2(1+\epsilon^a\tau)^{1+q_1}}\norm{\v\p^{\alpha}(\gt,\g)}_{\omega}^2\nnm\\
  &+\epsilon^{a-1}\norm{\p^{\alpha}(\gt,\g)}_{\sigma,\omega}^2
  -C\epsilon^{a-1}\|\p^{\alpha}(\gt,\g)\|_{\sigma}^2\nnm\\
\leq&\,-\epsilon^a\int_{\R}\int_{\R^3}\frac{\omega^2(\alpha,0)}{\mu}(E+v\times B)
     \cdot\nabla_v\big(\mu\p^{\alpha}\gt\p^{\alpha}\g \big)dvdy\nnm\\
&+\int_{\R}\int_{\R^3}\frac{\omega^2(\alpha,0)}{\sqrt{\mu}}\p^{\alpha}P_0\Big[v_1\p_y(\sqrt{\mu}\gt)
         +\epsilon^a(E+v\times B)\cdot\nabla_v(\sqrt{\mu}\g)\Big]\p^{\alpha}\gt dvdy\nnm\\
&+\int_{\R}\int_{\R^3}\frac{\omega^2(\alpha,0)}{\sqrt{\mu}}\p^{\alpha}\Big[P_d\big(v_1\p_y(\sqrt{\mu}\g))
          -\epsilon^a(E+v\times B)\cdot\nabla_v\bar{G}_1\Big]\p^{\alpha}\g dvdy\nnm\\					
&-\int_{\R}\int_{\R^3}\frac{\omega^2(\alpha,0)}{\sqrt{\mu}}\p^{\alpha}P_1\[v_1\(\frac{|v-u|^2\p_y\zeta}{2R\theta^2}
          +\frac{(v-u)\p_y\psi}{R\theta}\)M\]\p^{\alpha}\gt dvdy\nnm\\
&+\int_{\R}\int_{\R^3}\omega^2(\alpha,0)[ (\p^{\alpha}R_{\Gamma1}+\p^{\alpha}R_{g1}+R_{d1} )\p^{\alpha}\gt
   + (\p^{\alpha}R_{\Gamma2}+\p^{\alpha}R_{g2}+R_{d2} )\p^{\alpha}\g] dvdy\nnm\\
:=&\,J_{g,\omega}^1+J_{g,\omega}^2+J_{g,\omega}^3+J_{g,\omega}^4+J_{g,\omega}^5,
\ema
where $R_{\Gamma1},R_{\Gamma2},R_{g1},R_{g2}$ are defined in \eqref{remainder}
and $R_{d1},R_{d2}$ are defined in \eqref{R1d}.

For $J_{g,w}^1$, by  $\nabla_v\cdot(E+v\times B)=0$ and $v\cdot(v\times B)=0$, we have
\bma\label{Jgw1}
J_{g,\omega}^1
&= \epsilon^a\int_{\R}\int_{\R^3}\(\frac1{\mu}\nabla_v\omega^2(\alpha,0)+\omega^2(\alpha,0)\nabla_v\(\frac{1}{\mu}\)\)\cdot(E+v\times B)\big(\mu\p^{\alpha}\gt\p^{\alpha}\g \big) dvdy\nnm\\
&\leq C\epsilon^a\|E\|_{L^\infty_y}\|\v^{\frac12}\p^{\alpha}\big(\gt,\g \big)\|_\omega^2
\leq  C\mathcal{E}^{\frac12}(\tau)\mathcal{F}_{\omega}(\tau),
\ema
where we have used
\be
\p_{v_i}\omega(\alpha,0)
=2(l-|\alpha|)\frac{v_i}{\v^2}\omega(\alpha,0)
  +q(\tau)v_i\omega(\alpha,0).\nnm
\ee
For $J_{g,\omega}^2$, $J_{g,\omega}^3$ and $J_{g,\omega}^4$, by the same arguments as for  \eqref{pov}-\eqref{Jg234}, we have
\bma	
J_{g,\omega}^2+J_{g,\omega}^3+J_{g,\omega}^4
&\leq\lambda\epsilon^{a-1}\|\p^\alpha(\gt,\g)\|_{\sigma,\omega}^2
     +C\lambda^{-1}\(\mathcal{E}(\tau)+\sigma^2\frac{\epsilon^{2a}}{\delta^2}\)\mathcal{D}(\tau)
     +C\lambda^{-1}\frac{\epsilon^{3+a}}{(\frac{\delta}{\sigma}+\epsilon^a\tau)^2}\mathcal{E}(\tau)\nnm\\
&\quad+C\lambda^{-1}\epsilon^{1-a}\|\p^\alpha\dy(\gt,\g)\|_{\sigma,\omega}^2
 +C\lambda^{-1}\epsilon^{1-a}\|\p^\alpha\dy(\psi,\zeta)\|^2_{L^2_y}.
\ema
We now turn to estimate $J_{g,\omega}^5$. For terms involving $R_{g1},R_{g2}$  in $J_{g,\omega}^5$,
we use \eqref{remainder}, Lemma \ref{Gb} and \eqref{barG7-2} to have
\bma\label{Jgw5-1}
&\int_{\R}\int_{\R^3}\omega^2(\alpha,0) (\p^\alpha R_{g1}\p^\alpha\gt+\p^\alpha R_{g2}\p^\alpha\g ) dvdy\nnm\\
\leq& \lambda\epsilon^{a-1}\norm{\p^\alpha(\gt,\g)}_{\sigma,\omega}^2
 +C\lambda^{-1}\epsilon^{1-a}\|\p^\alpha\dy n\|^2_{L^2_y}
 + C\lambda^{-1}\epsilon^{1+a}\sum_{|\alpha'|=1}\|\p^{\alpha'}(E,B)\|^2_{L^2_y}\nnm\\
&+{C\lambda^{-1}\epsilon^{1+a}\|E+u\times B\|^2_{L^2_y}}
 +C\lambda^{-1}\frac{\epsilon^{1+3a}}{(\frac{\delta}{\sigma}+\epsilon^a\tau)^2}\|(E,B)\|^2_{L^2_y}
 +C\lambda^{-1}\frac{\epsilon^{1+a}}{(\frac{\delta}{\sigma}+\epsilon^a\tau)^2}\|n\|^2_{L^2_y}\nnm\\
&+C\lambda^{-1}\frac{\epsilon^{3}}{\delta(\frac{\delta}{\sigma}+\epsilon^a\tau)^2}
  +C\lambda^{-1}\sigma^2\frac{\epsilon^{2a}}{\delta^2}\mathcal{D}(\tau)
  +C\lambda^{-1}\mathcal{E}(\tau)\mathcal{D}(\tau).
\ema
For terms involving $R_{\Gamma_1}$ and $R_{\Gamma_2}$ in $J_{g,\omega}^5$, by using Lemma \ref{GMG} and Lemma \ref{QGG}, we have
\bma
&\int_{\R}\int_{\R^3}\omega^2(\alpha,0) (\p^{\alpha}R_{\Gamma1}\p^{\alpha}\gt+\p^{\alpha}R_{\Gamma2}\p^{\alpha}\g ) dvdy\nnm\\
\leq&\lambda\epsilon^{a-1}\|\p^{\alpha}(\gt,\g)\|^2_{\sigma,\omega}
 +C\lambda^{-1}\(\eta^2_1+\mathcal{E}(\tau)+\sigma^2\frac{\epsilon^{2a}}{\delta^2}\)\mathcal{D}(\tau)\nnm\\
&+C\lambda^{-1}\frac{\epsilon^3}{\delta(\frac{\delta}{\sigma}+\epsilon^a\tau)^2}
 +C\lambda^{-1}\frac{\epsilon^{1+a}}{(\frac{\delta}{\sigma}+\epsilon^a\tau)^2}\|n\|^2_{L^2_y}.
\ema		
Then we estimate terms involving $R_{d1}$ and $R_{d2}$ in $J_{g,\omega}^5$. When  $|\alpha|=0$, this term vanishes.
When  $|\alpha| =1$, we have
\bma\label{Jgw5-2}
&\int_{\R}\int_{\R^3}\omega^2(\alpha,0)\big(R_{d1}\p^\alpha\gt+R_{d2}\p^\alpha\g\big) dvdy\nnm\\
=&\,\epsilon^a\int_{\R}\int_{\R^3}\omega^2(\alpha,0)\frac1{\sqrt{\mu}}\p^{\alpha}(E+v\times B)
   \cdot\nabla_v\sqrt{\mu}(\g\p^{\alpha}\gt+\gt\p^{\alpha}\g) dvdy\nnm\\
&+\epsilon^a\int_{\R}\int_{\R^3}\omega^2(\alpha,0)\p^{\alpha}(E+v\times B)
   \cdot(\nabla_v\g\p^{\alpha}\gt+\nabla_v\gt\p^{\alpha}\g) dvdy\nnm\\
\leq&\,C\epsilon^a\int_{\R}|\p^{\alpha}(E,B)|\int_{\R^3}\v\omega(\alpha,0) (|(\gt,\g)|+|\nabla_v(\gt,\g)| )|\omega(\alpha,0)\p^\alpha(\gt,\g)|dvdy\nnm\\
\leq&\,C\epsilon^a\left\| \omega(\alpha,0) (|(\gt,\g)|_2+|\nabla_v(\gt,\g)|_2 )\right\|_{L^\infty_y}
 \|\p^{\alpha}(E,B)\|_{L^2_y}\|\v\omega(\alpha,0)\p^\alpha(\gt,\g)\|\nnm\\
\leq&\, C\epsilon^a\|\p^{\alpha}(E,B)\|_{L^2_y} \|\v\p^\alpha(\gt,\g)\|_{\omega}
   \sum_{|\alpha|\leq 1,|\beta|\leq1}\|\v\p^{\alpha}_{\beta}(\gt,\g)(\tau)\|_\omega\nnm\\
\leq&\, C\mathcal{E}^{\frac12}(\tau)\mathcal{F}_{\omega}(\tau).
\ema
Then combining all estimates for $J_{g,\omega}^i$, $i=1,2,...,5$, we completes the proof of the lemma.
\end{proof}

\subsection{Weighted estimates on mixed-derivatives of $(\gt,\g)$}
Applying $\p^{\alpha}_\beta$ to \eqref{g1}-\eqref{g2} with $|\alpha|+|\beta|\leq 2$, $|\beta|\geq 1$, for $e_1=(1,0,0)$, we have
\bma\label{g1ab}
&\p_\tau\p^{\alpha}_{\beta}\gt+v_1\p_y\p^{\alpha}_{\beta}\gt+ C^{\beta-e_1}_{\beta}\p_y\p^{\alpha}_{\beta-e_1}\gt
  +\epsilon^a\frac{1}{\sqrt{\mu}}\big[(E+v\times B)\cdot\nabla_v(\sqrt{\mu}\p^{\alpha}_{\beta}\g)\big]
  -\epsilon^{a-1}  \p^{\alpha}_{\beta}L\gt \nnm\\
=&\,\p^{\alpha}_{\beta}\[\frac{1}{\sqrt{\mu}}P_0\Big(v_1\sqrt{\mu}\dy\gt
          +\epsilon^a(E+v\times B) \cdot\nabla_v(\sqrt{\mu}\g)\Big)\]\nnm\\
&-\p^{\alpha}_{\beta}\[\frac1{\sqrt{\mu}}P_1\(v_1\(\frac{|v-u|^2\dy\zeta}{2R\theta^2}+\frac{(v-u)\cdot\dy\psi}{R\theta}\)M\)\]
  +\p^{\alpha}_{\beta}R_{\Gamma1}+\p^{\alpha}_{\beta}R_{g1}+R_{2d1},
\ema
and
\bma\label{g2ab}
&\dtau \p^{\alpha}_{\beta}\g+v_1\dy \p^{\alpha}_{\beta}\g+ C^{\beta-e_1}_{\beta}\dy \p^{\alpha}_{\beta-e_1}\g
 +\epsilon^a \frac{1}{\sqrt{\mu}}\big[(E+v\times B)\cdot\nabla_v(\sqrt{\mu}\p^{\alpha}_{\beta}\gt)\big]
 -\epsilon^{a-1} \p^{\alpha}_{\beta}\L \g\nnm\\
=&\,\p^{\alpha}_{\beta}\[\frac{1}{\sqrt{\mu}}P_d\big(v_1\sqrt{\mu}\dy\g\big)-\epsilon^a\frac{1}{\sqrt{\mu}}(E+v\times B)\nabla_v \bar{G}_1\]
 +\p^{\alpha}_{\beta}R_{\Gamma2}+\p^{\alpha}_{\beta}R_{g2}+R_{2d2},
\ema
where $R_{\Gamma1},R_{\Gamma2},R_{g1},R_{g2}$ are defined in \eqref{remainder} and
\bma
R_{2d1}:=&\epsilon^a\sum_{1\leq|\alpha_1|+|\beta_1|\leq 2}\p^{\alpha_1}_{\beta_{1}}
    \(\frac{E+v\times B}{\sqrt{\mu}}\)\cdot\nabla_v\p^{\alpha-\alpha_1}_{\beta-\beta_1}(\sqrt{\mu}\g)\nnm\\
&+\frac1{\sqrt{\mu}}\epsilon^a(E+v\times B)\cdot
     \sum_{1\leq|\beta_1|\leq|\beta|}\nabla_v(\p_{\beta_1}\sqrt{\mu}\p^\alpha_{\beta-\beta_1}\g),\nnm\\
R_{2d2}:=&\epsilon^a\sum_{1\leq|\alpha_1|+|\beta_1|\leq 2}\p^{\alpha_1}_{\beta_{1}}
    \(\frac{E+v\times B}{\sqrt{\mu}}\)\cdot\nabla_v\p^{\alpha-\alpha_1}_{\beta-\beta_1}(\sqrt{\mu}\gt)\nnm\\
&+\frac1{\sqrt{\mu}}\epsilon^a(E+v\times B)\cdot\sum_{1\leq|\beta_1|\leq|\beta|}\nabla_v(\p_{\beta_1}\sqrt{\mu}\p^\alpha_{\beta-\beta_1}\gt).\nnm
\ema

\begin{lem}\label{GGabw}
Under the same assumptions as in Proposition \ref{priori3}, 
for $|\alpha|+|\beta|\leq 2$ and $|\beta|\geq 1$, we have 
\bma
\blue&\sum_{|\alpha|+|\beta|\leq 2,|\beta|\geq 1}\(\frac{d}{d\tau}\|\p^\alpha_\beta (\gt,\g)\|^2_\omega
 +\frac{\epsilon^a q_1q_2}{(1+\epsilon^a\tau)^{1+q_1}}\|\v\p^{\alpha}_\beta(\gt,\g)\|_\omega^2
 +\epsilon^{a-1}\|\p^\alpha_\beta (\gt,\g)\|_{\sigma,\omega}^2\)\nnm\\
&\qquad\leq C\sum_{|\alpha|\leq1}\epsilon^{1-a}\|\p^\alpha\dy(\psi,\zeta,n)\|^2_{L^2_y}
 + C\epsilon^{1+a}\sum_{|\alpha'|=1}\|\p^{\alpha'}(E,B)\|^2_{L^2_y}
 +C\epsilon^{1+a}\|E+u\times B\|^2_{L^2_y}\nnm\\
&\quad\qquad+C\(\eta^2_1+\mathcal{E}(\tau)+\sigma^2\frac{\epsilon^{2a}}{\delta^2}\)\mathcal{D}(\tau)
 +C\frac{\epsilon^3}{\delta(\frac{\delta}{\sigma}+\epsilon^a\tau)^2}
 +  C\mathcal{E}^{\frac12}(\tau)\mathcal{F}_\omega(\tau)
 +C\frac{\epsilon^{1+a}}{(\frac{\delta}{\sigma}+\epsilon^a\tau)^2}\mathcal{E}(\tau)\nnm\\
&\quad\qquad+C\lambda^{-1}\epsilon^{1-a}\sum_{1\leq|\alpha_1|\leq2}\|\p^{\alpha_1}(\gt,\g)\|^2_{\sigma,\omega}
 +\lambda\epsilon^{a-1}\|\p^\alpha(\gt,\g)\|^2_{\sigma,\omega},\nnm
\ema
where $\lambda>0$ is a small constant, and $\eta_1>0$ is defined in \eqref{local}.
\end{lem}
\begin{proof}
Multiplying \eqref{g1ab} by $\omega^2(\alpha,\beta)\p_{\beta}^{\alpha}\gt$, \eqref{g2ab} by $\omega^2(\alpha,\beta)\p_{\beta}^{\alpha}\g$ and integrating over $\R_y\times\R^3_v$,
by  Lemma \ref{L-weight}, we have
\bma\label{GGabw-1}
&\frac12\frac{d}{d\tau}\|\p_{\beta}^{\alpha}(\gt,\g)\|^2_\omega
 +\frac{\epsilon^a q_1q_2}{2(1+\epsilon^a\tau)^{1+q_1}}\|\v\p_{\beta}^{\alpha}(\gt,\g)\|_{\omega}^2\nnm\\
& +\epsilon^{a-1}\|\p_{\beta}^{\alpha}(\gt,\g)\|_{\sigma,\omega}^2
 -C\epsilon^{a-1}\sum_{|\beta_1|<|\beta|}\|\p^{\alpha}_{\beta_1}(\gt,\g)\|_{\sigma,\omega}^2\nnm\\
\leq&\,-\epsilon^a\int_{\R}\int_{\R^3}\frac{\omega^2(\alpha,\beta)}{\mu}(E+v\times B)
  \cdot\nabla_v\big(\mu\p_{\beta}^{\alpha}\gt\p^{\alpha}_{\beta}\g \big)dvdy\nnm\\
&+\int_{\R}\int_{\R^3}\p_{\beta}^{\alpha}\[\frac{1}{\sqrt{\mu}}P_0\Big(v_1\sqrt{\mu}\dy\gt
  +\epsilon^a(E+v\times B)\cdot \nabla_v(\sqrt{\mu}\g)\Big)\]\omega^2(\alpha,\beta)\p_{\beta}^{\alpha}\gt dvdy\nnm\\
&+\int_{\R}\int_{\R^3}\p_{\beta}^{\alpha}\[\frac1{\sqrt{\mu}}P_d\big(v_1\sqrt{\mu}\dy\g\big)-\epsilon^a\frac{1}{\sqrt{\mu}}(E+v\times B)\nabla_v\bar{G}_1\]\omega^2(\alpha,\beta)\p^{\alpha}_{\beta}\g dvdy\nnm\\				&-\int_{\R}\int_{\R^3}\p_{\beta}^{\alpha}\[\frac{1}{\sqrt{\mu}}P_1\(v_1\(\frac{|v-u|^2\p_y\zeta}{2R\theta^2}
   +\frac{(v-u)\cdot\p_y\psi}{R\theta}\)M\)\]\omega^2(\alpha,\beta)\p_{\beta}^{\alpha}\gt dvdy\nnm\\					
&+\int_{\R}\int_{\R^3} (\p_{\beta}^{\alpha}R_{\Gamma1}+\p_{\beta}^{\alpha}R_{g1}+R_{2d1} )\omega^2(\alpha,\beta)\p^{\alpha}_{\beta}\gt
   + (\p_{\beta}^{\alpha}R_{\Gamma2}+\p_{\beta}^{\alpha}R_{g2}+R_{2d2} )\omega^2(\alpha,\beta)\p_{\beta}^{\alpha}\g dvdy\nnm\\
&-C^{\beta-e_1}_{\beta}\int_{\R}\int_{\R^3}\omega^2(\alpha,\beta)(\p_y\p^{\alpha}_{\beta-e_1}\gt\p^{\alpha}_{\beta}\gt
                              +\p_y\p^{\alpha}_{\beta-e_1}\g\p^{\alpha}_{\beta}\g)dvdy\nnm\\
:=&\,J_{g,\omega}^{\beta,1}+J_{g,\omega}^{\beta,2}+J_{g,\omega}^{\beta,3}+J_{g,\omega}^{\beta,4}+J_{g,\omega}^{\beta,5}
 +J_{g,\omega}^{\beta,6} .
\ema
For $J_{g,\omega}^{\beta,1}$, by noting that $\nabla_v\cdot(E+v\times B)=0,~v\cdot(v\times B)=0$, we have
\bma\label{Jgwb1}
J_{g,\omega}^{\beta,1}
&=\epsilon^a\int_{\R}\int_{\R^3}\bigg|\nabla_v\(\frac{\omega^2(\alpha,\beta)}{\mu}\)(E+v\times B)
    \big(\p_{\beta}^{\alpha}\gt\p^{\alpha}_{\beta}\g \big)\bigg|dvdy\nnm\\
&\le C\epsilon^a\|E\|_{L^\infty_y}\|\v^{\frac12}\p_{\beta}^{\alpha}(\gt,\g)\|_{\omega}^2
\leq  C\mathcal{E}^{\frac12}(\tau)\mathcal{F}_\omega(\tau),
\ema
where we have used
\be
\p_{v_i}\omega(\alpha,\beta)=2(l-|\alpha|-|\beta|)\frac{v_i}{\v^2}\omega(\alpha,\beta)+q(\tau)\frac{v_i}{\v}\omega(\alpha,\beta).\nnm
\ee
For $J_{g,\omega}^{\beta,2}$, $J_{g,\omega}^{\beta,3}$ and $J_{g,\omega}^{\beta,4}$, by using arguments similar to those for \eqref{pov}-\eqref{pov1}, we have
\bma
J_{g,\omega}^{\beta,2}+J_{g,\omega}^{\beta,3}+J_{g,\omega}^{\beta,4}
&\leq\lambda\epsilon^{a-1}\|\p^\alpha_\beta(\gt,\g)\|_{\sigma,\omega}^2
     +C\lambda^{-1}\(\mathcal{E}(\tau)+\sigma^2\frac{\epsilon^{2a}}{\delta^2}\)\mathcal{D}(\tau)
     +C\lambda^{-1}\frac{\epsilon^{3+a}}{(\frac{\delta}{\sigma}+\epsilon^a\tau)^2}\mathcal{E}(\tau)\nnm\\
&\quad+C\lambda^{-1}\epsilon^{1-a}\|\p^\alpha\dy(\gt,\g)\|_{\sigma,\omega}^2
 +C\lambda^{-1}\epsilon^{1-a}\|\p^\alpha\dy(\psi,\zeta)\|^2_{L^2_y}.
\ema
Similar to \eqref{Jgw5-1}-\eqref{Jgw5-2}, we have
\bma
J_{g,\omega}^{\beta,5}
\leq&\lambda\epsilon^{a-1}\|\p^{\alpha}_\beta(\gt,\g)\|^2_{\sigma,\omega}
 +C\lambda^{-1}\(\eta^2_1+\mathcal{E}(\tau)+\sigma^2\frac{\epsilon^{2a}}{\delta^2}\)\mathcal{D}(\tau)
 +C\lambda^{-1}\frac{\epsilon^3}{\delta(\frac{\delta}{\sigma}+\epsilon^a\tau)^2}\nnm\\
&+C\lambda^{-1}\epsilon^{1-a}\|\p^\alpha\dy n\|^2_{L^2_y}
 + C\lambda^{-1}\epsilon^{1+a}\sum_{|\alpha'|=1}\|\p^{\alpha'}(E,B)\|^2_{L^2_y}
 + C\lambda^{-1}\epsilon^{1+a}\|E+u\times B\|^2_{L^2_y}\nnm\\
&+{ C\lambda^{-1}\frac{\epsilon^{1+3a}}{(\frac{\delta}{\sigma}+\epsilon^a\tau)^2}\|(E,B)\|^2_{L^2_y}}
 +C\lambda^{-1}\frac{\epsilon^{1+a}}{(\frac{\delta}{\sigma}+\epsilon^a\tau)^2}\|n\|^2_{L^2_y}
 +C\mathcal{E}^{\frac12}(\tau)\mathcal{F}_{\omega}(\tau).
\ema
For $J_{g,\omega}^{\beta,6}$,  we have
\bma
J_{g,\omega}^{\beta,6}
=&-C^{\beta-e_1}_{\beta}\int_{\R}\int_{\R^3}\omega^2(\alpha,\beta)
 (\p_y\p^{\alpha}_{\beta-e_1}\gt\p^{\alpha}_{\beta}\g+\p_y\p^{\alpha}_{\beta-e_1}\g\p^{\alpha}_{\beta}\gt)dvdy\nnm\\
\leq&C\int_{\R}\int_{\R^3}\(\omega(\alpha+1,\beta-e_1)|\dy\p^{\alpha}_{\beta-e_1}(\gt,\g)|\)
   \( \omega(\alpha,\beta)|\p^\alpha_{\beta}(\gt,\g)|\)dvdy.
\ema
Then by  $\v^2\omega(\alpha,\beta)=\omega(\alpha,\beta-e_1)$ and
$\p^\alpha_\beta=\p_{e_1}\p^\alpha_{\beta-e_1}$, with  $|\beta|=1$, $|\alpha|\leq1$,
we have from \eqref{sigma1} that
\bma
J_{g,\omega}^{\beta,6}
&\leq C\lambda^{-1}\epsilon^{1-a}\|\v^{-\frac12}\omega(\alpha+1,\beta-e_1)\dy\p^{\alpha}_{\beta-e_1}(\gt,\g)\|^2
 +\lambda\epsilon^{a-1}\|\v^2\v^{-\frac32}\omega(\alpha,\beta)\p_{e_1}\p^\alpha_{\beta-e_1}(\gt,\g)\|^2\nnm\\
&\leq C\lambda^{-1}\epsilon^{1-a}\|\omega(\alpha+1,\beta-e_1)\dy\p^{\alpha}_{\beta-e_1}(\gt,\g)\|^2_\sigma
 +\lambda\epsilon^{a-1}\|\v^2\omega(\alpha,\beta)\p^\alpha_{\beta-e_1}(\gt,\g)\|^2_\sigma\nnm\\
&\leq C\lambda^{-1}\epsilon^{1-a}\sum_{1\leq|\alpha_1|\leq2}\|\p^{\alpha_1}(\gt,\g)\|^2_{\sigma,\omega}
 +\lambda\epsilon^{a-1}\|\p^\alpha(\gt,\g)\|^2_{\sigma,\omega}.
\ema	
When  $|\beta|=2$, $|\alpha|=0$, we have from \eqref{sigma1} that
\bma
J_{g,\omega}^{\beta,6}
&\leq \lambda\epsilon^{a-1}\|\v^{-\frac12}\omega(1,\beta-e_1)\dy\p_{\beta-e_1}(\gt,\g)\|^2
 +C\lambda^{-1}\epsilon^{1-a}\|\v^2\v^{-\frac32}\omega(0,\beta)\p_{e_1}\p_{\beta-e_1}(\gt,\g)\|^2\nnm\\
&\leq \lambda\epsilon^{a-1}\|\omega(1,\beta-e_1)\dy\p_{\beta-e_1}(\gt,\g)\|^2_\sigma
 +C\lambda^{-1}\epsilon^{1-a}\|\v^2\omega(0,\beta)\p_{\beta-e_1}(\gt,\g)\|^2_\sigma\nnm\\
&\leq \lambda\epsilon^{a-1}\|\dy\p_{\beta-e_1}(\gt,\g)\|^2_{\sigma,\omega}
 +C\lambda^{-1}\epsilon^{1-a}\|\p_{\beta-e_1}(\gt,\g)\|^2_{\sigma,\omega}.
\ema
Then combining all the above estimates with $\eps,\lambda>0$ being small enough
	completes the proof of the lemma.
\end{proof}			
\subsection{Weighted estimate on the highest order derivatives of  $(F_1,F_2)$}
\begin{lem}\label{Fal}
Under the same assumptions as in Proposition \ref{priori3},   with  $\omega=\omega(\alpha,0)$ defined in \eqref{weight},
we have for $|\alpha|=2$, 
\bmas
\epsilon^{2-2a}&\frac{d}{d\tau}\|\mu^{-\frac12}\dya (F_1,F_2)\|^2_\omega
 +\frac{q_1q_2\epsilon^{2-a}}{2(1+\epsilon^a\tau)^{1+q_1}}\|\v\mu^{-\frac12}\dya (F_1,F_2)\|^2_\omega
 +\epsilon^{1-a}\|\dya (\gt,\g)\|^2_{\sigma,\omega}\\
\leq&\,
 C\(\sigma^2\frac{\epsilon^{2a}}{\delta^2}+\epsilon^{-1}\mathcal{E}(\tau)\)\mathcal{D}(\tau)
 +C\frac{\epsilon^{1+2a}}{\delta(\frac{\delta}{\sigma}+\epsilon^a\tau)^2}
 +C\epsilon^{1-a}\|\dya(\gt,\g)\|^2_\sigma\\
&+ C\lambda^{-1}\frac{\epsilon^a}{\frac{\delta}{\sigma}+\epsilon^a\tau}\mathcal{E}(\tau)
 +C\epsilon^{1-a}\|\p^\alpha(\phi,\psi,\zeta,n)\|^2_{L^2_y}
 +C\(\epsilon^a+\mathcal{E}^{\frac12}(\tau)\)\mathcal{F}_\omega(\tau).
\emas
\end{lem}
\begin{proof}
Taking the inner product of \eqref{F1-1} with $\omega^2(\alpha,0)\frac{\dya F_1}{\sqrt{\mu}}$,
and taking the inner product of \eqref{F2-1} with $\omega^2(\alpha,0)\frac{\dya F_2}{\sqrt{\mu}}$,
by using Lemma \ref{L-weight}, we obtain
\bmas
&\frac12 \frac{d}{d\tau}\|\mu^{-\frac12}\dya (F_1,F_2)\|^2_\omega
 +\frac{q_1q_2\epsilon^{a}}{2(1+\epsilon^a\tau)^{1+q_1}}\|\v\mu^{-\frac12}\dya (F_1,F_2)\|^2_\omega\\
&+\epsilon^{a-1}\|\dya (\gt,\g)\|^2_{\sigma,\omega}
 -C\epsilon^{a-1}\|\dya (\gt,\g)\|^2_\sigma\\
\leq&-\epsilon^{a}\int_{\R}\int_{\R^3} \frac{\omega^2(\alpha,0)}{{\mu}}(E+v\times B)\cdot\nabla_v(\dya F_2\dya F_1)dvdy
 +\epsilon^{a-1}\int_{\R}\int_{\R^3}\omega^2(\alpha,0)\frac{\dya F_1}{\mu}\dya{L_M\bar{G}_1}dvdy\\
&-\epsilon^{a}\sum_{1\leq|\alpha_1|\leq|\alpha|}C_{\alpha}^{\alpha_1}
  \int_{\R}\int_{\R^3} \frac{\omega^2(\alpha,0)}{\mu} \p^{\alpha_1}(E+v\times B)\cdot
     (\Tdv\pt^{\alpha-\alpha_1} F_2 \dya F_1+\Tdv\pt^{\alpha-\alpha_1} F_1 \dya F_2)dvdy\\
&+\epsilon^{a-1}\int_{\R}\int_{\R^3} \omega^2(\alpha,0)\frac{(\dya M+\dya \bar{G}_1)}{\sqrt{\mu}} L  \dya\gt
 +\omega^2(\alpha,0)\frac{\dya\(\frac n\rho M\)}{\sqrt{\mu}} \L  \dya\g dvdy+I^\omega_{f,\Gamma}\\
:=&\,I_{f,\omega}^1+I_{f,\omega}^2+I_{f,\omega}^3+I_{f,\omega}^4+I^\omega_{f,\Gamma},
\emas
where
\bmas
I^\omega_{f,\Gamma}
=&\int_{\R}\int_{\R^3}\omega^2(\alpha,0)\[\dya\Gamma\(\gt,\frac{M-\mu}{\sqrt{\mu}}\)
  +\dya\Gamma\(\frac{M-\mu}{\sqrt{\mu}},\gt\)
  +\dya\Gamma\(\frac{G_1}{\sqrt{\mu}},\frac{G_1}{\sqrt{\mu}}\)\]\frac{\dya F_1}{\sqrt{\mu}}dvdy\nnm\\
&+\int_{\R}\int_{\R^3}\omega^2(\alpha,0)\[\dya\Gamma\(\g,\frac{M-\mu}{\sqrt{\mu}}\)
  +\dya\Gamma\(\frac{F_2}{\sqrt{\mu}},\frac{G_1}{\sqrt{\mu}}\)\]\frac{\dya F_2}{\sqrt{\mu}}dvdy,
\emas
which is estimated in Lemma \ref{QMGF} and  Lemma \ref{GGF}.

For $I_{f,\omega}^1$, by using  $\nabla_v\cdot(E+v\times B)=0$ and $v\cdot(v\times B)=0$, we have
\be\label{Ifw1}
I_{f,\omega}^1
= \epsilon^{a}\int_{\R}\int_{\R^3} \nabla_v\(\frac{\omega^2(\alpha,0)}{{\mu}}\)\cdot(E+v\times B)(\dya F_2\dya F_1)dvdy
 \leq  C\epsilon^{2a-2}\mathcal{E}^{\frac12}(\tau)\mathcal{F}_\omega(\tau).
 \ee
For $I_{f,\omega}^2$, by Lemma \ref{rarefaction4}, \eqref{barG1-1} and $P_1=I-P_0$, we have
\bma
I_{f,\omega}^2
&= \epsilon^{a-1}\int_{\R}\int_{\R^3}\omega^2(\alpha,0)\frac{\dya F_1}{\mu}\dya{L_M\bar{G}_1}dvdy\nnm\\
&= \int_{\R}\int_{\R^3}\omega^2(\alpha,0)
  \p^\alpha P_1\bigg[v_1M\bigg(\frac{|v-u|^2\dy\bar{\theta}}{2R\theta^2}+\frac{(v-u)\dy\bar{u}}{R\theta}\bigg)\bigg]
   \frac{\p^\alpha F_1}{\mu}dvdy\nnm\\
&\leq C\bigg\|\omega(\alpha,0)\v^{\frac12}\frac{1}{\sqrt{\mu}}
    \p^\alpha P_1\bigg[v_1M\bigg(\frac{|v-u|^2\dy\bar{\theta}}{2R\theta^2}+\frac{(v-u)\dy\bar{u}}{R\theta}\bigg)\bigg]\bigg\|
  \bigg\|\omega(\alpha,0)\v^{-\frac12}\frac{\p^\alpha F_1}{\sqrt{\mu}}\bigg\|\nnm\\
&\leq \lambda\epsilon^{a-1}\|\mu^{-\frac12}\p^\alpha F_1\|^2_{\sigma,\omega}
 +C\lambda^{-1}\frac{\epsilon^{1+4a}}{\delta^3(\frac{\delta}{\sigma}+\epsilon^a\tau)^2}
 +C\lambda^{-1}\sigma^2\frac{\epsilon^{2a}}{\delta^2}\mathcal{D}(\tau).
\ema
Next, for $I_{f,\omega}^3$, as $ |\alpha_1|=1$, by noting  $F_1=M+\bar{G}_1+\sqrt{\mu}\gt$ and $F_2=\frac{n}{\rho}M+\sqrt{\mu}\g$, we have
\bma
I_{f,\omega}^{3}
&= -2\epsilon^{a}
  \int_{\R}\int_{\R^3}\omega^2(\alpha,0) \frac{1}{\mu} \p^{\alpha_1}(E+v\times B)\cdot
     (\Tdv\pt^{\alpha-\alpha_1} F_2 \dya F_1+\Tdv\pt^{\alpha-\alpha_1} F_1 \dya F_2)dvdy\nnm\\
&\leq C\epsilon^a\|\p^{\alpha_1}(E,B)\|_{L^\infty_y}					
  \|\omega(\alpha,0)\v^{\frac32}\v^{-\frac32}\mu^{-\frac12}\p^{\alpha-\alpha_1}\nabla_v(F_1,F_2)\|
  \|\omega(\alpha,0)\v\mu^{-\frac12}\p^{\alpha}(F_1,F_2)\| \nnm\\
&\leq C\epsilon^{2a}\|\omega(\alpha,0)\v\mu^{-\frac12}\p^{\alpha}(F_1,F_2)\|^2
 +C\|\p^{\alpha_1}(E,B)\|^2_{L^\infty_y}\|{\omega(\alpha-\alpha_1,0)}\mu^{-\frac12} \p^{\alpha-\alpha_1}(F_1,F_2)\|^2_\sigma \nnm\\
&\leq C\epsilon^{3a-2}\mathcal{F}_\omega(\tau)
 +C\|\p^{\alpha_1}(E,B)\|^2_{L^\infty_y}
   \(\|\mu^{-\frac12}\p^{\alpha-\alpha_1}(M,\frac{n}{\rho}M,\bar{G}_1)\|^2_{\sigma,\omega}+\|\p^{\alpha-\alpha_1}(\gt,\g)\|^2_{\sigma,\omega}\)\nnm\\
&\leq C\epsilon^{3a-2}\mathcal{F}_\omega(\tau)
 +C\epsilon^{a-2}\mathcal{E}(\tau)\mathcal{D}(\tau)
 +C\frac{\epsilon^a}{\frac{\delta}{\sigma}+\epsilon^a\tau}\|\p^{\alpha_1}(E,B)\|^2_{L^\infty_y},
\ema
where we have used the fact $\v^{\frac32}\omega(\alpha,0)\le\omega(\alpha-1,0).$

As $ |\alpha_1|=|\alpha|=2$, we have
\bma
I_{f,\omega}^3
=&\,-2\epsilon^{a}
  \int_{\R}\int_{\R^3}\omega^2(\alpha,0) \frac{1}{\mu} \p^{\alpha}(E+v\times B)\cdot
     (\Tdv F_2 \dya F_1+\Tdv F_1 \dya F_2)dvdy\nnm\\\
\leq&\, C\epsilon^a\|\p^{\alpha}(E,B)\|_{L^2_y}					
\||\mu^{-\frac12}\v^{\frac32}\omega(\alpha,0)\nabla_vM|_2\|_{L_y^\infty}
 \|\mu^{-\frac12}\v^{-\frac12}\omega(\alpha,0)\p^{\alpha}(F_1,F_2)\|\nnm\\
&+C\epsilon^a\|\p^{\alpha}(E,B)\|_{L^2_y}					
\||\mu^{-\frac12}\omega(\alpha,0)\nabla_v(F_1-M,F_2)|_2\|_{L_y^\infty}
 \|\mu^{-\frac12}\v\omega(\alpha,0)\p^{\alpha}(F_1,F_2)\|\nnm\\
\leq&\, C\lambda^{-1}\epsilon^{1+a}\|\p^{\alpha}(E,B)\|^2_{L^2_y}
 \||\mu^{-\frac12}\v^{\frac32}\omega(\alpha,0)\nabla_v M|_2\|^2_{L_y^\infty}\nnm\\
&+\lambda\epsilon^{a-1}\|\mu^{-\frac12}\v^{-\frac12}\omega(\alpha,0)\p^{\alpha}(F_1,F_2)\|^2\nnm\\
&+C\|\p^{\alpha}(E,B)\|^2_{L^2_y}
 \||\mu^{-\frac12}\omega(\alpha,0)\nabla_v(\frac{n}{\rho}M,\bar{G}_1,\sqrt{\mu}\gt,\sqrt{\mu}\g)|_2\|^2_{L_y^\infty}\nnm\\
&+C\epsilon^{2a}\|\mu^{-\frac12}\v\omega(\alpha,0)\p^{\alpha}(F_1,F_2)\|^2\nnm\\
\leq&\,C\lambda^{-1}\frac{\epsilon^a}{\frac{\delta}{\sigma}+\epsilon^a\tau}\|\p^{\alpha}(E,B)\|^2_{L^2_y}
 +\lambda\epsilon^{a-1}\|\mu^{-\frac12}\p^{\alpha}(F_1,F_2)\|^2_{\sigma,\omega}\nnm\\
 &+C\epsilon^{2a-3}\mathcal{E}(\tau)\mathcal{D}(\tau)
 +C\epsilon^{3a-2}\mathcal{F}_\omega(\tau),
\ema
where we have used for $\frac\delta\sigma\leq1$ and $\epsilon (1+T)\leq 1$,
\bmas
&\epsilon^{1+a}
 \|\p^{\alpha}(E,B)\|^2_{L^2_y}\||\mu^{-\frac12}\v^{\frac32}\omega(\alpha,0)\nabla_vM|_2\|^2_{L_y^\infty}\nnm\\
\leq&C\epsilon(1+T)\frac{\epsilon^a}{(\frac\delta\sigma+\epsilon^a\tau)}\|\p^{\alpha}(E,B)\|^2_{L^2_y}
\leq C\frac{\epsilon^a}{\frac{\delta}{\sigma}+\epsilon^a\tau}\|\p^{\alpha}(E,B)\|^2_{L^2_y},
\emas
\bmas
&\|\p^{\alpha}(E,B)\|^2_{L^2_y}\||\mu^{-\frac12}\omega(\alpha,0)\nabla_v(\frac{n}{\rho}M,\bar{G}_1)|_2\|^2_{L_y^\infty}\nnm\\
\leq&C\|\p^{\alpha}(E,B)\|^2_{L^2_y}\|n\|_{L^2_y}\|\dy n\|_{L^2_y}
  +C\|\p^{\alpha}(E,B)\|^2_{L^2_y}\epsilon^{1-a}|\dy(\bar{u}_1,\bar{\theta})|^2_{L_y^\infty}\nnm\\
\leq&C\epsilon^{2a-3}\mathcal{E}(\tau)\mathcal{D}(\tau)
 +C\frac{\epsilon^{1+a}}{(\frac{\delta}{\sigma}+\epsilon^a\tau)^2}\|\p^{\alpha}(E,B)\|^2_{L^2_y},
\emas
and
\bmas
&\|\p^{\alpha}(E,B)\|^2_{L^2_y}\||\mu^{-\frac12}\omega(\alpha,0)\nabla_v(\sqrt{\mu}\gt,\sqrt{\mu}\g)|_2\|^2_{L_y^\infty}\nnm\\
\leq& C\|\p^{\alpha}(E,B)\|^2_{L^2_y}
       \|\omega(\alpha,0)\v^{\frac32}(\gt,\g)\|_\sigma \|\omega(\alpha,0)\v^{\frac32}\dy(\gt,\g)\|_\sigma \nnm\\
\leq& C\|\p^{\alpha}(E,B)\|^2_{L^2_y}\|\omega(0,0)(\gt,\g)\|_\sigma \|\omega(1,0)\dy(\gt,\g)\|_\sigma \nnm\\
\leq& C\epsilon^{a-1}\mathcal{E}(\tau)\mathcal{D}(\tau).
\emas

Lastly, for $I_{f,\omega}^4$, by \eqref{L12}, \eqref{M1-2}, Lemma \ref{Gamma} and Lemma \ref{Gb}, we have
\bma
I_{f,\omega}^4
\leq&\, \lambda\epsilon^{a-1}\|\p^\alpha(\gt,\g)\|^2_{\sigma,\omega}
 +C\lambda^{-1}\epsilon^{a-1}\|\p^\alpha(\phi,\psi,\zeta,n)\|^2_{L^2_y}\nnm\\
&+C\lambda^{-1}\frac{\epsilon^{4a-1}}{\delta(\frac{\delta}{\sigma}+\epsilon^a\tau)^2}
 +C\lambda^{-1}\epsilon^{2a-2}\(\sigma^2\frac{\epsilon^{2a}}{\delta^2}+\mathcal{E}(\tau)\)\mathcal{D}(\tau).
\ema
Moreover, as for \eqref{F1-2order}-\eqref{F2-2order}, we have from \eqref{norm-m} that
\bma
\epsilon^{1-a}\|\mu^{-\frac12}\p^\alpha F_1\|^2_{\sigma,\omega}
\leq&\,\epsilon^{1-a}\|\p^\alpha \gt\|^2_{\sigma,\omega}
 +C\epsilon^{1-a}\|\p^\alpha(\phi,\psi,\zeta)\|^2_{L^2_y}\nnm\\
 &+C\(\mathcal{E}(\tau)+\sigma^2\frac{\epsilon^2}{\delta^2}\)\mathcal{D}(\tau)
 +C\frac{\epsilon^{1+2a}}{\delta(\frac{\delta}{\sigma}+\epsilon^a\tau)^2},\label{Fa1-2order}\\
\epsilon^{1-a}\|\mu^{-\frac12}\p^\alpha F_2\|^2_{\sigma,\omega}
\leq&\,\epsilon^{1-a}\|\p^\alpha \g\|^2_{\sigma,\omega}+C\epsilon^{1-a}\|\p^\alpha n\|^2_{L^2_y}\nnm\\
&+C\(\mathcal{E}(\tau)+\sigma^2\frac{\epsilon^{2a}}{\delta^2}\)\mathcal{D}(\tau)
 +C\frac{\epsilon^{1+2a}}{\delta(\frac{\delta}{\sigma}+\epsilon^a\tau)^2}\|n\|^2_{L^\infty_y}.\label{Fa2-2order}
\ema
Then combining all the above estimates $I_{f,\omega}^i(i=1,2,3,4)$ and  multiplying the results by $\epsilon^{2-2a}$
and using \eqref{Fa1-2order}-\eqref{Fa2-2order},
choosing $\lambda>0$ small enough, we complete the proof of the lemma.
\end{proof}
\section{Convergence rate }\setcounter{equation}{0}
Based on the energy estimates derived in Sections 3, 4 and 5, in this section we will  complete the
proof of Theorem \ref{mt}: By a suitable linear combination of  Lemma \ref{lolem}, Lemma \ref{lolem1} and Lemma \ref{GGa2}-Lemma \ref{Fal},
and integrating over $[0,s],s\in(0,\tau_1],~\tau_1=T/\epsilon^a$ and using the fact (cf. \cite{HuangYT})
that
$$
E^2+2E\cdot\bar u \times B+B^2\sim E^2+B^2,\quad \eta\sim|(\phi,\psi,\zeta)|^2,
$$
we have for small enough $\lambda>0$ that
\bma\label{convergence1}
&\quad\mathcal{E}(s)
 +\frac{q_1q_2}{(1+T)^{1+q_1}}\int^s_0\mathcal{F}_\omega(\tau)d\tau
 +\int^s_0\mathcal{D}(\tau)d\tau
 \nnm\\
&\leq C\(\eta^2_1+\epsilon^a+\sigma^2\frac{\epsilon^{2a}}{\delta^2}+{\epsilon^{a-1}\sup_{0\leq\tau\leq s}\mathcal{E}^{\frac12}(\tau)
       +(1+T)\epsilon^{-1}\sup_{0\leq\tau\leq s}\mathcal{E}(\tau)}\)\int^s_0\mathcal{D}(\tau)d\tau\nnm\\
&\quad+ C\int^s_0\frac{\epsilon^a}{\frac{\delta}{\sigma}+\epsilon^a\tau}\mathcal{E}(\tau)d\tau
 + C\Big(\epsilon^a+\sup_{0\leq\tau\leq s}\mathcal{E}^{\frac12}(\tau)\Big)\int^s_0\mathcal{F}_\omega(\tau)d\tau\nnm\\
&\quad+C\sigma(1+T)\(\frac{\epsilon^{2-a}}{\delta^2}+\frac{\epsilon^{1+a}}{\delta^2}\)
 +C\mathcal{E}(0).
\ema
Here $\mathcal{E}(\tau)$, $\mathcal{D}(\tau)$ and $\mathcal{F}_\omega(\tau)$ are defined by \eqref{energy-a}-\eqref{energy-F}, respectively.
We need to require the terms on the right hand side of \eqref{convergence1} to be absorbed by the left hand side. From the assumption of Proposition \ref{priori3}, there exists a small positive constants $\epsilon_1$, independent of $\epsilon,~\delta$ and $T$, such that
\bma
&~~\epsilon^{a-1}\sup_{0\leq\tau\leq s}\mathcal{E}^{\frac12}(\tau)
 \leq(1+T)\bigg(\frac{\epsilon^{\frac a2}}{\delta^{\frac32}}+\frac{\epsilon^{\frac32a-\frac12}}{\delta^{\frac32}}\bigg)
 \ll\epsilon_1,\nnm\\
&~~(1+T)\epsilon^{-1}\sup_{0\leq\tau\leq s}\mathcal{E}(\tau)
 \leq(1+T)^3\bigg(\frac{\epsilon^{1-a}}{\delta^3}+\frac{\epsilon^a}{\delta^3}\bigg)
 \ll\epsilon_1,\label{q1q2}\\
&\sup_{0\leq\tau\leq s}\mathcal{E}^{\frac12}(\tau)
 \leq(1+T)\bigg(\frac{\epsilon^{1-\frac a2}}{\delta^{\frac32}}+\frac{\epsilon^{\frac12+\frac a2}}{\delta^{\frac32}}\bigg)
 \leq \frac{q_1q_2}{(1+T)^2}
 <\frac{q_1q_2}{(1+T)^{1+q_1}},\nnm
\ema
that is
$$
q_1\in\bigg[\frac{(1+T)^3}{q_2}
  \bigg(\frac{\epsilon^{1-\frac a2}}{\delta^{\frac32}}+\frac{\epsilon^{\frac12+\frac a2}}{\delta^{\frac32}}\bigg),1\bigg),\quad a\in\Big(\frac13,1\Big).
$$
Hence, by using the smallness of $\eps_1,\eta_1,\frac{\eps}{\delta}> 0$,  we have for $s\in(0,\tau_1]$
\be\label{convergence2}
\mathcal{E}(s)
\leq C\sigma(1+T)\(\frac{\epsilon^{2-a}}{\delta^2}+\frac{\epsilon^{1+a}}{\delta^2}\)
 +C\int^s_0\frac{\epsilon^a}{\frac{\delta}{\sigma}+\epsilon^a\tau}\mathcal{E}(\tau)d\tau,\quad a\in\Big(\frac13,1\Big).
\ee
Then by  the Gronwall's inequality in Lemma \ref{Gronwall}, \eqref{convergence2} implies that for $\tau_1=\frac{T}{\epsilon^a}$ and $a\in(\frac13,1)$,
\be
\sup\limits_{0\leq\tau\leq\tau_1}\mathcal{E}(\tau)
\leq C\sigma^2(\frac{\delta}{\sigma}+T)(1+T)\bigg(\frac{\epsilon^{2-a}}{\delta^3}+\frac{\epsilon^{1+a}}{\delta^3}\bigg)
\leq C\sigma^2(1+T)^2\bigg(\frac{\epsilon^{2-a}}{\delta^3}+\frac{\epsilon^{1+a}}{\delta^3}\bigg) .
\ee
Thus by choosing $\sigma>0$ small enough so that $C\sigma^2<\frac12$,
the a priori assumption \eqref{priori1} is closed and Proposition \ref{priori3} is proved. Therefore, by the uniform a priori estimates and the local existence of the solution, the standard continuity argument gives the existence and uniqueness of local solutions to the VML equation \eqref{VML2} with initial data \eqref{initial0}.

By \eqref{priori4} and using the Sobolev imbedding, we have 
\be\label{priori5}
\sup_{0\leq \tau\leq \tau_1}\|(\phi,\psi,\zeta)\|^2_{L^{\infty}_y}
\leq C\sup_{0\leq \tau\leq \tau_1}\|(\phi,\psi,\zeta)\|_{L^2_y}\|\p_y(\phi,\psi,\zeta)\|_{L^2_y}
\leq C(1+T)^2\bigg(\frac{\epsilon^{2-a}}{\delta^3}+\frac{\epsilon^{1+a}}{\delta^3}\bigg).
\ee
Then \eqref{priori4} and \eqref{priori5} imply that
there exists a positive constant $C$ independent of $\epsilon,\delta,T$ such that 
\be
\sup_{0\leq\tau\leq\tau_1}\left\{\|(\phi,\psi,\zeta,n)\|^2_{L^\infty_y}
 +\|(\gt,\g)\|^2_{L^\infty_yL^2_v}
 +\|\mu^{-\frac12}\bar{G}_1\|^2_{L^\infty_yL^2_v }\right\}
\leq C(1+T)^2\bigg(\frac{\epsilon^{2-a}}{\delta^3}+\frac{\epsilon^{1+a}}{\delta^3}\bigg).\nnm
\ee
Therefore,  the local solution
$[F_1(t,x,v),F_2(t,x,v), E(t,x), B(t,x)]$ to  the VML system \eqref{VML2} satisfies
\be \label{rate1}
\sup_{t\in[0,T]}\bigg\{\bigg\|\frac{F_1-M_{[\bar{\rho},\bar{u},\bar{\theta}]}}{\sqrt{\mu}},\frac{F_2}{\sqrt{\mu}}\bigg\|^2_{ L_x^\infty L_v^2}
  +\|(E,B)\|^2_{L_x^\infty}\bigg\}
\leq C(1+T)^2\bigg(\frac{\epsilon^{2-a}}{\delta^3}+\frac{\epsilon^{1+a}}{\delta^3}\bigg),
\ee
for $a\in (\frac13,1 )$.
This gives \eqref{mt1} and then it completes the proof of Theorem \ref{mt}.

Next, we give the proof of Theorem \ref{mt2} as follows. Since the terms
$(1+T)\epsilon^{-1}\sup\limits_{0\leq\tau\leq s}\mathcal{E}(\tau) $ and $ \epsilon^{a-1}\sup\limits_{0\leq\tau\leq s}\mathcal{E}^{\frac12}(\tau) $ in \eqref{q1q2} lose an $\epsilon$ decay, by  Lemma \ref{rarefaction4}, we have
\bma\label{rate}
&\quad \|(\rho,u,\theta)-(\rho^{r},u^{r},\theta^{r})\|_{L^{\infty}_y}\nnm\\
&\leq {(\eps^{-\frac12}+\eps^{a-1})}\|(\phi,\psi,\zeta)\|_{L^{\infty}_y}
 +\left\| (\bar{\rho},\bar{u},\bar{\theta})-(\rho^{r},u^{r},\theta^{r})\right\|_{L^{\infty}_y}\nnm\\
&\leq C(1+T) \bigg(\frac{\epsilon^{\frac12-\frac a2}}{\delta^\frac32}+{\frac{\epsilon^{\frac{3a}2-\frac12}}{\delta^\frac32}}\bigg)
 +Ct^{-1}\delta\(\ln(1+t)+|\ln\delta|\),
\ema
for any $t\in [0,T]$. We can take $\delta=(1+T)^{\frac25}\max\{\epsilon^{\frac15-\frac a5}, \epsilon^{\frac{3a}5-\frac15} \}$ in  \eqref{rate1} and \eqref{rate}, for any $h>0$, we have for $ a\in(\frac13,1)$,
$$
\sup_{t\in[h,T]}\bigg\{\bigg\|\frac{F_1-M_{[\rho^{r},u^{r},\theta^{r}]}}{\sqrt{\mu}},
   \frac{F_2}{\sqrt{\mu}}\bigg\|_{ L_x^\infty L_v^2}
 +\|(E,B)\|_{L_x^\infty}\bigg\}
\leq C\frac1h(1+T)^{\frac25}\[\ln(1+T)+|\ln\epsilon|\]\big(\epsilon^{\frac15-\frac a5}+\epsilon^{\frac{3a}5-\frac15}\big).
$$
This completes the proof of Theorem \ref{mt2}.
\appendix
\section{Appendix }
\setcounter{equation}{0}
\renewcommand{\theequation}{A.\arabic{equation}}
In this section, we will present some basic estimates, which have been used in the previous calculations. First of all, we list some properties of the smooth rarefaction wave defined in \eqref{smooth}.
Then, regarding the slow Knudsen decay rate of $\|\dy(\bar{u},\bar{\theta})\|^{2}$ in the term $P_1(v_1\dy M)$ in \eqref{G1}, we provide some properties of the Burnett functions and the rapid velocity decay of $\bar{G}_1$ defined in \eqref{barG1-1} to overcome this slow decay rate.
Finally, we establish some fundamental properties of the collision operators  used in the  energy analysis.

Now, we list the properties of the smooth approximate 3-rarefaction wave $(\bar{\rho},\bar{u},\bar{\theta})(t,x)$ constructed in \eqref{smooth}, cf. \cite{HuangLW,xin1993}.

\begin{lem}\label{rarefaction4}
The functions $(\bar{\rho},\bar{u},\bar{\theta})$ constructed by \eqref{smooth} have the following properties:

\text{(1)} $\p_{x}\bar{u}_1(t,x)>0$, $\bar{u}_2(t,x)=\bar{u}_3(t,x)=0$ for $x\in \R$, $t\geq0,\delta>0.$

\text{(2)} For any $1\leq p\leq\infty$, there is a constant $C>0$ 
 such that 
 for all $t>0,~\delta>0,$ $\p^{\alpha}=\p^{\alpha_{1}}_{t}\p^{\alpha_{2}}_{x}$,
\bma
\|\p^{\alpha}(\bar{\rho},\bar{u},\bar{\theta})\|_{L^p_x}
   &\leq C\sigma^{\frac1p}(\frac{\delta}{\sigma}+t)^{-1+\frac1p},
   \quad for~|\alpha|=1,\nnm\\
\|\p^{\alpha}(\bar{\rho},\bar{u},\bar{\theta})\|_{L^p_x}
   &\leq C\delta^{-|\alpha|+1+\frac1p}(\frac{\delta}{\sigma}+t)^{-1},
   \quad for~|\alpha|\geq2,\nnm
\ema
\quad~~where $\sigma=|\rho_{+}-\rho_{-}|+|u_{+}-u_{-}|+|\theta_+-\theta_-|$ is a small constant.

\text{(3)} There exist a constant $\delta_{0}\in(0,1)$, such that for any $\delta\in(0,\delta_{0}]$ and $t>0,$
$$
\|(\bar{\rho},\bar{u},\bar{\theta})(t,\cdot)-(\rho^{r},u^{r},\theta^{r})(\frac{\cdot}{t})\|_{L_{x}^{\infty}}
\leq Ct^{-1}\delta\(\ln(1+t)+|\ln\delta|\).
$$
\end{lem}
Note that the scaling transformation $(\tau,y)=(\epsilon^{-a} t,\epsilon^{-a}x)$ is considered through
the proof. The following lemma is directly follows from Lemma \ref{rarefaction4}$_{(2)}$, which has been used in energy estimates.
\begin{lem}\label{rarefaction5}
Let $(\bar{\rho},\bar{u},\bar{\theta})(t,x)$ be the smooth approximate 3-rarefaction wave defined in \eqref{smooth}, then it holds that
$$
\|\p^{k}_y(\bar{\rho},\bar{u},\bar{\theta})(\epsilon^a\tau,\epsilon^a y)\|_{L^p_y}
\leq C\epsilon^{a(k-\frac1p)}\|\p^{k}_x(\bar{\rho},\bar{u},\bar{\theta})(t,x)\|_{L^p_{x}},
\quad k\geq1,
$$
for any $\tau=\frac{t}{\epsilon^a}>0$, $y=\frac{x}{\epsilon^a}$ and $p\in[1,+\infty]$.
\end{lem}
\begin{lem}[Gronwall's inequality \cite{Evans}]\label{Gronwall}
Let $\xi(t),~k(t)$ be two non-negative integrable functions on $[0,T]$
satisfying 
$$
\xi(t)\leq \int_{0}^{t}k(s)\xi(s)ds+C_1,
$$
for a constant $C_1\geq 0$. Then
$$
\xi(t)\leq C_1\exp\(\int^t_0k(s)ds \), \quad 0 \leq t\leq T.
$$
\end{lem}
Now recall the  Burnett functions 
\be\label{Burnett1}
\hat{\A}_j(v)=\frac{|v|^2-5}{2}v_j,\quad \hat{\B}_{ij}(v)=v_{i}v_{j}-\frac13\delta_{ij}|v|^2,
\quad  i,j = 1,2,3.
\ee
Noting that $\hat{\A}_j\(\frac{v-u}{\sqrt{R\theta}}\)M$ and $\hat{\B}_{ij}\(\frac{v-u}{\sqrt{R\theta}}\)M$ are orthogonal to the null space $N_0$ of $L_M$, we can define functions $\A_{j}\(\frac{v-u}{\sqrt{R\theta}}\)$ and $\B_{ij}\(\frac{v-u}{\sqrt{R\theta}}\)$ as
\be\label{Burnett2}
\A_j\(\frac{v-u}{\sqrt{R\theta}}\)=L_M^{-1}\[\hat{\A}_j\(\frac{v-u}{\sqrt{R\theta}}\) M\],\quad
\B_{ij}\(\frac{v-u}{\sqrt{R\theta}}\)=L_M^{-1}\[\hat{\B}_{ij}\(\frac{v-u}{\sqrt{R\theta}}\) M\].
\ee

By using the Burnett functions,  the viscosity coefficient $\kappa_1(\theta)$ and heat conductivity coefficient $\kappa_2(\theta)$ can be represented by
\bma
\kappa_1(\theta)&=-R\theta\int_{\R^3}\hat{\B}_{ij}\(\frac{v-u}{\sqrt{R\theta}}\)\B_{ij}\(\frac{v-u}{\sqrt{R\theta}}\) dv>0,
  \quad i\neq j,\label{viscosity}\\
\kappa_2(\theta)&=-R^2\theta\int_{\R^3}\hat{\A}_{j}\(\frac{v-u}{\sqrt{R\theta}}\)\A_{j}\(\frac{v-u}{\sqrt{R\theta}}\)dv>0.\label{heat}
\ema
Notice that these coefficients are positive smooth functions depending only on $\theta$.

\begin{lem}[\cite{Duan6}]\label{Burnett4}
Suppose that $U(v)$ is any polynomial of $\frac{v-\hat{u}}{\sqrt{R}\hat{\theta}}$ such that $U(v)\hat{M} \in(\ker L_{\hat{M}})^{\bot}$ for any Maxwellian $\hat{M}=M_{[\hat{\rho},\hat{u},\hat{\theta}]}(v)$ where $L_{\hat{M}}$ is defined in \eqref{collision1}. For any $\varepsilon_0 \in(0,1)$ and any multi-index $\beta$, there exists a constant $C_\beta> 0$ such that
\be
|\p_{\beta}L_{\hat{M}}^{-1}(U(v)\hat{M})|\leq C_{\beta}(\hat{\rho},\hat{u},\hat{\theta})\hat{M}^{1-\varepsilon_0}.
\ee
In particular, if the assumptions of \eqref{local} hold, there exists constant $C_{\beta}>0$ such that
\be
\left|\p_{\beta}\A_j\(\frac{v-u}{\sqrt{R\theta}}\) \right| + \left|\p_{\beta}\B_{ij}\(\frac{v-u}{\sqrt{R\theta}}\)\right|
\leq C_{\beta}M^{1-\varepsilon_0},\quad i,j = 1,2,3.
\ee
\end{lem}

\begin{lem}\label{Burnett5}
Suppose that $U(v)$ is any polynomial of $v$ such that $U(v)\hat{M} \in(\ker \L_{\hat{M}})^{\bot}$ for any Maxwellian $\hat{M}=M_{[\hat{\rho},\hat{u},\hat{\theta}]}(v)$ where $\L_{\hat{M}}$ is defined in  \eqref{collision2}. For any $\varepsilon_0 \in(0,1)$ and any multi-index $\beta$, there exists a constant $C_\beta> 0$ such that
$$
|\p_{\beta}\L_{\hat{M}}^{-1}(U(v)\hat{M})|\leq C_{\beta}\hat{M}^{1-\varepsilon_0}.
$$
\end{lem}

We refer to Lemma 6.1 in \cite{Duan6} for a proof.



The following lemma is about the velocity decay of the Burnett functions, which is  crucial in our energy estimates, cf. \cite{duan2021,Duan6}.
\begin{lem}\label{Burnett3}
Under the assumptions \eqref{local}, for any $|\beta|\geq 0$ and $m>0$,
\be\label{fast decay}
\int_{\R^3}\frac{\left|\v^m \sqrt{\mu}\p_{\beta}\A_{j}\(\frac{v-u}{\sqrt{R\theta}}\)\right|^2}{M^2}dv
 +\int_{\R^3}\frac{\left|\v^m \sqrt{\mu}\p_{\beta}\B_{ij}\(\frac{v-u}{\sqrt{R\theta}}\)\right|^{2}}{M^2}dv\leq C,
\ee
and
\be
\int_{\R^3}\frac{\left|\v^m \sqrt{\mu}\p_{\beta}\mathcal{L}^{-1}_M(v_iM)\right|^2}{M^2}dv \leq C.
\ee

\end{lem}

For the linear operators  $L$ and $\L$ in \eqref{L12}, we have the following facts \cite{Wang}.
First, $L$ and $\L$ are self-adjoint and non-positive, and  their null spaces are
$$
 \ker L ={\rm span}\{\sqrt{\mu},v\sqrt{\mu},|v|^{2}\sqrt{\mu}\},\quad \ker \L ={\rm span}\{\sqrt{\mu}\}.
$$
Moreover, there are some velocity weighted estimates for $L,~\L$ and $\Gamma$ defined in \eqref{L12}.	
\begin{lem}[\cite{Guo1,Wang}]\label{linear}
Let $\tilde{L}$ be $ L $ or $ \L $ defined in \eqref{L12}. We have
$$
\langle \tilde{L} g,h\rangle=\langle g,\tilde{L} h\rangle,\quad
\langle \tilde{L} g,g\rangle \geq 0,\quad \text{and}\quad \tilde{L} g= 0,
$$
if and only if $g=\tilde{P}g$, where $\widetilde{P}$ is the $L_v^2$ orthogonal projection onto the null space of $\tilde{L}$. Moreover, there exists a $\sigma_1>0$ such that
\be\label{H1}
-\langle\tilde{L}g,g\rangle\ge \sigma_1 |\{I-\tilde{P}\}g |_\sigma^2.
\ee
\end{lem}

\begin{lem}[\cite{Wang,SG}]\label{L-weight}
Let $\tilde{L}$ be $L$ or  $\L$  defined in \eqref{L12} and
$\omega=\omega(\alpha,\beta)$ in \eqref{weight} with sufficiently small $q_1>0$ and $q_2\in(0,\frac12)$.
For any small constant $\lambda>0$, there exists  $C_\lambda>0$ such that
\begin{equation}
-\langle\p^\alpha_\beta \tilde{L}g,\omega^2(\alpha,\beta)\p^\alpha_\beta g\rangle
 \geq |\omega(\alpha,\beta)\p^\alpha_\beta g|_\sigma^2-\lambda\sum_{|\beta_1|=|\beta|}|\omega(\alpha,\beta_1)\p^\alpha_{\beta_1} g|_\sigma^2
	-C_\lambda\sum_{|\beta_1|<|\beta|}|\omega(\alpha,\beta_1)\p^\alpha_{\beta_1} g|_\sigma^2.
\end{equation}
If $|\beta|= 0$, then we have
\begin{equation}
-\langle\p^\alpha\tilde{L}g,\omega^2(\alpha,0)\p^\alpha g\rangle
 \geq C_{\lambda,q}|\omega(\alpha,0)\p^\alpha g|_\sigma^2-C_\lambda|\chi_\lambda(v)\p^\alpha g|_\sigma^2,
\end{equation}
where $\chi_\lambda(v)$ is a general cutoff function, that is,
$\chi_\lambda(v)=1,~\text{for}~|v|\geq\lambda;~\chi_\lambda(v)=0,~ \text{for}~ |v|\leq\lambda.$
\end{lem}

\begin{lem}[\cite{SG,SZ,Wang}]\label{Gamma}
Let $\Gamma(g_1,g_2)$ be defined in \eqref{L12} and $\omega=\omega(\alpha,\beta)$ in \eqref{weight}
with  sufficiently small $q_1>0$ and $q_2\in(0,\frac12)$.
For any small constant $\varepsilon_0>0$, it holds that
\begin{equation}
\langle\p^\alpha \Gamma(g_1,g_2), g_3\rangle
\leq C\sum_{\alpha_1\leq\alpha}|\mu^{\varepsilon_0}\p^{\alpha_1}g_1|_2| \p^{\alpha-\alpha_1}g_2|_\sigma | g_3|_\sigma,
\end{equation}
and
\begin{equation}
\langle\p^\alpha_\beta \Gamma(g_1,g_2), \omega^2(\alpha,\beta)g_3\rangle
\leq C\sum_{\alpha_1\leq\alpha}\sum_{\bar{\beta}\leq\beta_1\leq\beta}|\mu^{\varepsilon_0}\p^{\alpha_1}_{\bar{\beta}}g_1|_2
  |\omega(\alpha,\beta) \p^{\alpha-\alpha_1}_{\beta-\beta_1}g_2|_\sigma |\omega(\alpha,\beta)g_3|_{\sigma}.
\end{equation}
\end{lem}

\subsection{Estimates on  $\bar{G}_1$ and $M$}
\begin{lem}\label{Gb}
Recalling $\bar{G}_1$ defined in \eqref{barG1-1} and $\langle v\rangle=\sqrt{1+|v|^2}$.
Under the same assumptions as in  Proposition \ref{priori3}, for any $b \geq 0$ and $|\beta| \geq 0$, we have
\be\label{Gb1}
\left\|\langle v\rangle^b \omega(0,\beta)\p_\beta\(\frac{\bar{G}_1}{\sqrt{\mu}}\)\right\|^2
 +\left\|\langle v\rangle^b \omega(0,\beta)\p_\beta\(\frac{\bar{G}_1}{\sqrt{\mu}}\)\right\|_{\sigma}^2
\leq C\sigma\frac{\epsilon^{2-a}}{(\frac{\delta}{\sigma}+\epsilon^a\tau)} .
\ee
For $|\alpha|=1$, we have
\be\label{Gb2}
\left\|\langle v\rangle^b\omega(\alpha,\beta) \p_\beta^\alpha\(\frac{\bar{G}_1}{\sqrt{\mu}}\)\right\|^2
 +\left\|\langle v\rangle^b\omega(\alpha,\beta) \p_\beta^\alpha\(\frac{\bar{G}_1}{\sqrt{\mu}}\)\right\|_{\sigma}^2
 \leq C\frac{\epsilon^{2+a}}{\delta(\frac{\delta}{\sigma}+\epsilon^a\tau)^2}+C\sigma^2\frac{\epsilon^{1+a}}{\delta^2}\mathcal{D}(\tau).
\ee
And for $|\alpha|=2$,  by choosing $\frac{\epsilon^a}{\delta}$ small enough, it holds that
\be\label{Gb3}
 \left\|\langle v\rangle^b\omega(\alpha,\beta)\p_\beta^\alpha\(\frac{\bar{G}_1}{\sqrt{\mu}}\)\right\|^2
 +\left\|\langle v\rangle^b\omega(\alpha,\beta)\p_\beta^\alpha\(\frac{\bar{G}_1}{\sqrt{\mu}}\)\right\|_{\sigma}^2
\leq C\frac{\epsilon^{2+3a}}{\delta^3(\frac{\delta}{\sigma}+\epsilon^a\tau)^2}
  +C\sigma^2\frac{\epsilon^{1+a}}{\delta^2}\mathcal{D}(\tau).
\ee
\end{lem}
\begin{proof}
We first use the Burnett functions to rewrite $\bar{G}_1$ defined in \eqref{barG1-1} as follows
$$
\bar{G}_1=\epsilon^{1-a} \frac{\sqrt{R}}{\sqrt{\theta}} \dy \bar{\theta} \A_1\(\frac{v-u}{\sqrt{R \theta}}\)
 +\epsilon^{1-a} \dy\bar{u}_1 \B_{11}\(\frac{v-u}{\sqrt{R \theta}}\) .
$$
Next, for the velocity derivatives,  direct calculation yields, for $k=1,2,3$,
\be\label{Gbv}
\frac{\p \bar{G}_1}{\p v_k}
=\epsilon^{1-a}\frac{\sqrt{R}}{\sqrt{\theta}} \dy \bar{\theta}\p_{v_k} \A_1\(\frac{v-u}{\sqrt{R \theta}}\)\frac{1}{\sqrt{R \theta}}
 +\epsilon^{1-a} \dy \bar{u}_1 \p_{v_k} \B_{11}\(\frac{v-u}{\sqrt{R \theta}}\) \frac{1}{\sqrt{R \theta}}.
\ee
Similarly, we have
\bma\label{Gby}
\frac{\p \bar{G}_1}{\p y}
=&\epsilon^{1-a}\Bigg\{\frac{\sqrt{R}}{\sqrt{\theta}}\p_y^2 \bar{\theta} \A_1\(\frac{v-u}{\sqrt{R \theta}}\)
 -\frac{\sqrt{R}}{2 \sqrt{\theta^3}}\p_y\bar{\theta}\p_y \theta \A_1\(\frac{v-u}{\sqrt{R \theta}}\)\nnm\\
&-\frac{\sqrt{R}}{\sqrt{\theta}}\p_y \bar{\theta}\nabla_v \A_1\(\frac{v-u}{\sqrt{R \theta}}\)\cdot\(\frac{1}{\sqrt{R \theta}}\p_y u+\frac{v-u}{2\sqrt{R \theta^3}}\p_y\theta\)
\nnm\\
&+\p_y^2\bar{u}_1 \B_{11}\(\frac{v-u}{\sqrt{R \theta}}\) -\p_y \bar{u}_1\nabla_v \B_{11}\(\frac{v-u}{\sqrt{R \theta}}\)  \cdot\(\frac{1}{\sqrt{R\theta}}\p_y u
+\frac{v-u}{2 \sqrt{R \theta^3}}\p_y\theta\)\Bigg\},
\ema
and $\p_\tau \bar{G}_1$ has the similar expression as \eqref{Gby}.

Then, combining the above expansion and using Lemma \ref{Burnett4},
for any $|\beta|\geq 0$, we can obtain
\bma\label{barG7-1}
\left|\v^b \omega(0,\beta)\p_{\beta}\(\frac{\bar{G}_1}{\sqrt{\mu}}\)\right|_{2}
&\leq C\epsilon^{1-a}|\dy(\bar{u}_1,\bar{\theta})|
  \left|\v^b \omega(0,\beta)\p_\beta \(\mu^{-\frac12} \A_1\(\frac{v-u}{\sqrt{R \theta}}\)
                                             +\mu^{-\frac12}\B_{11}\(\frac{v-u}{\sqrt{R \theta}}\)\)\right|_2\nnm\\
&\leq C\epsilon^{1-a}|\dy(\bar{u}_1,\bar{\theta})||\v^b \omega(0,\beta)\mu^{-\frac12} M^{1-\varepsilon_0}|_2\nnm\\
&\leq C\epsilon^{1-a}|\dy(\bar{u}_1,\bar{\theta})|.
\ema
Here we have used the fact that $|\v^b \omega(0,\beta)\mu^{-\frac12} M^{1-\varepsilon_0}|_2\leq C$
for any $b\geq 0$ and $\varepsilon_0>0$ small enough, that is from the definition of $\omega$ in \eqref{weight}
\be
\v^{b+2( l-|\beta|)}e^{\frac{q_2}{(1+\epsilon^a\tau)^{q_1}}\frac{1+|v|^2}2}e^{\frac{|v|^2}4}e^{-(1-\varepsilon_0)\frac{|v-u|^2}{2R\theta}}
\leq C,\nnm
\ee
then, we need
\be
\frac{1-\varepsilon_0}{2R\theta}-\frac14-\frac{q_2}2>0\Longrightarrow  q_2<\frac1{R\theta}-\frac12-\frac{\varepsilon_0}{R\theta}<\frac12.\nnm
\ee

Similarly, we have
\be\label{barG7-3}
\left|\langle v\rangle^b \omega(0,\beta)\p_{\beta}\(\frac{\bar{G}_1}{\sqrt{\mu}}\)\right|_{\sigma}
\leq C\epsilon^{1-a}|\dy(\bar{u}_1,\bar{\theta})|,
\ee
and for $|\alpha|\ge 1$,
\bma\label{barG7-2}
&\left|\langle v\rangle^b \omega(\alpha,\beta)\p_{\beta}^{\alpha}\(\frac{\bar{G}_1}{\sqrt{\mu}}\)\right|_{2}
 +\left|\langle v\rangle^b \omega(\alpha,\beta)\p_{\beta}^{\alpha}\(\frac{\bar{G}_1}{\sqrt{\mu}}\)\right|_{\sigma}\nnm\\
\leq&\, C\epsilon^{1-a}\(|\p^{\alpha}\dy(\bar{u}_1,\bar{\theta})|+\cdots
 +|\dy(\bar{u}_1,\bar{\theta})||\p^{\alpha}(\rho,u,\theta)|\).
\ema
Then by Lemma \ref{rarefaction4}, we have  for $|\alpha|=0$,
\be
 \left\|\langle v\rangle^b\omega(0,\beta) \p_\beta\(\frac{\bar{G}_1}{\sqrt{\mu}}\)\right\|^2
 +\left\|\langle v\rangle^b\omega(0,\beta) \p_\beta\(\frac{\bar{G}_1}{\sqrt{\mu}}\)\right\|^2_{\sigma}
\leq C\epsilon^{2-2a}\|\dy(\bar{u}_1,\bar{\theta})\|^2_{L^2_y}
\leq C\sigma\frac{\epsilon^{2-a}}{\frac{\delta}{\sigma}+\epsilon^a\tau},\nnm
\ee
 for $|\alpha| =1$,
\bma
&\left\|\langle v\rangle^b\omega(\alpha,\beta) \p_\beta^\alpha\(\frac{\bar{G}_1}{\sqrt{\mu}}\)\right\|^2
 +\left\|\langle v\rangle^b\omega(\alpha,\beta) \p_\beta^\alpha\(\frac{\bar{G}_1}{\sqrt{\mu}}\)\right\|^2_{\sigma}\nnm\\
\leq&C\epsilon^{2-2a}\(\|\p^{\alpha}\dy(\bar{u}_1,\bar{\theta})\|^2_{L^2_y}
   +\|\p^{\alpha}(\bar{\rho},\bar{u}_1,\bar{\theta})\|^2_{L^4_y}
    +\|\dy(\bar{u}_1,\bar{\theta})\|^2_{L^\infty_y}\|\p^{\alpha}(\phi,\psi,\zeta)\|^2_{L^2_y}\)\nnm\\
\leq &C\frac{\epsilon^{2+a}}{\delta(\frac{\delta}{\sigma}+\epsilon^a\tau)^2}
 +C\sigma^2\frac{\epsilon^{1+a}}{\delta^2}\mathcal{D}(\tau),
\ema
and for $|\alpha|=2$, 
\bma\label{Gbsd}
&\left\|\langle v\rangle^b\omega(\alpha,\beta) \p_\beta^\alpha\(\frac{\bar{G}_1}{\sqrt{\mu}}\)\right\|^2
 +\left\|\langle v\rangle^b\omega(\alpha,\beta) \p_\beta^\alpha\(\frac{\bar{G}_1}{\sqrt{\mu}}\)\right\|_{\sigma}^2\nnm\\
\leq&C\epsilon^{2-2a}\int_{\R}|\p^\alpha\dy(\bar{u}_1,\bar{\theta})|^2
  +\sum_{|\alpha_1|=1}|\p^{\alpha_1}\dy(\bar{u}_1,\bar{\theta})|^2|\p^{\alpha-\alpha_1}(\rho,u, \theta)|^2
  +|\dy(\bar{u}_1,\bar{\theta})|^2|\p^\alpha(\rho,u,\theta)|^2d y\nnm\\
\leq&C\frac{\epsilon^{2+3a}}{\delta^3(\frac{\delta}{\sigma}+\epsilon^a\tau)^2}
 +C\sigma^2\frac{\epsilon^{1+a}}{\delta^2}\mathcal{D}(\tau).
\ema
This completes the proof of lemma.
\end{proof}
\begin{lem}\label{M}
Under the same assumptions as in Proposition \ref{priori3}, for any $|\beta|$ and $b\geq 0$ we have
\begin{equation}\label{norm-mu}
\bigg| \v^{b}\omega(0,\beta)\p _{\beta}\Big(\frac{M-\mu}{\sqrt{\mu}}\Big)\bigg|^2_{\sigma}
 +\bigg| \v^{b}\omega(0,\beta)\p _{\beta}\Big(\frac{M-\mu}{\sqrt{\mu}}\Big)\bigg|^2_{2}
\leq C\eta_1^2,
\end{equation}
and for $|\alpha|=2$,
\bma\label{norm-m}
&\|\v^b \omega(\alpha,0)\mu^{-\frac12}\p^{\alpha}M\|_{\sigma}^2
  +\|\v^b \omega(\alpha,0)\mu^{-\frac12}\p^{\alpha}M\|^2\nnm\\
\leq&\, C\|\p^\alpha(\phi,\psi,\zeta)\|^2_{L^2_y}
 +C\(\epsilon^{a-1}\mathcal{E}(\tau)+\sigma^2\frac{\epsilon^{3a-1}}{\delta^2}\)\mathcal{D}(\tau)
 +C\frac{\epsilon^{3a}}{\delta(\frac{\delta}{\sigma}+\epsilon^a\tau)^2}.
\ema
\end{lem}
\begin{proof}
Firstly, the proof of \eqref{norm-mu} can be found in Lemma 6.6 of \cite{Duan6}. 						
Next, for \eqref{norm-m}, we only estimate $\|\v^b \omega(\alpha,0)\mu^{-\frac12}\p^{\alpha}M\|^2$ because the other one is similar.
By \eqref{M1-2} and Lemma \ref{rarefaction4}, for $|\alpha|=2$, we have
\bma\label{norm-m-1}
&\|\v^b \omega(\alpha,0)\mu^{-\frac12}\p^{\alpha}M\|^2\nnm\\
\leq&\, C|\v^b \omega(\alpha,0)\mu^{-\frac12} M^{1-\varepsilon_0}|_2
   \int_{\R}|\p^\alpha(\rho,u,\theta)|^2
       +\sum_{|\alpha_1|=1}|\p^{\alpha-\alpha_1}(\rho,u,\theta)|^2|\p^{\alpha_1}(\rho,u,\theta)|^2 dy\nnm\\
\leq&\, C\|\p^\alpha(\phi,\psi,\zeta)\|^2_{L^2_y}
 +C\epsilon^{a-1}\Big(\mathcal{E}(\tau)+\sigma^2\frac{\epsilon^{2a}}{\delta^2}\Big)\mathcal{D}(\tau)
 +C\frac{\epsilon^{3a}}{\delta(\frac{\delta}{\sigma}+\epsilon^a\tau)^2},
\ema
which yields \eqref{norm-m}.
\end{proof}

\subsection{The interaction of macro-micro components}
In the following, we  consider the interaction of macroscopic and microscopic components $\Gamma(\frac{M-\mu}{\sqrt{\mu}},\mathbf{g}_i)$, $i=1,2$. 
\begin{lem}\label{GMG}
Under the same assumptions as in Proposition \ref{priori3}, by letting $g$ be $\gt$ or $\g$ and $l\geq 2$ in $\omega(\alpha,\beta)$ defined in \eqref{weight}, then for either $|\alpha|+|\beta|\leq 2$ and $|\beta|\geq 1$, or  $|\beta|=0$ and $|\alpha|\leq 1$,
 we have for any well-defined function $h$,
\bma\label{GMG1}
&\epsilon^{a-1}\left|\(\p_\beta^\alpha \Gamma\(g,\frac{M-\mu}{\sqrt{\mu}}\),\omega^2(\alpha,\beta)h\)\right|
 +\epsilon^{a-1}\left|\(\p_\beta^\alpha \Gamma\(\frac{M-\mu}{\sqrt{\mu}},g\), \omega^2(\alpha,\beta)h\)\right|\nnm\\
&\quad\leq\, C\lambda^{-1}\(\eta^2_1+\mathcal{E}(\tau)+\sigma^2\frac{\epsilon^{2a}}{\delta^2}\)\mathcal{D}(\tau)
 +\lambda\epsilon^{a-1}\|\omega(\alpha,\beta)h\|^2_{\sigma},
\ema
where $\lambda>0$ is a small constant and $\eta_1>0$ is defined in \eqref{local}.
\end{lem}
\begin{proof}
For the first term on the left hand of \eqref{GMG1}, by Lemma \ref{Gamma}, for $|\alpha|+|\beta| \leq 2$, $|\beta| \geq 1$ and any small constant ${\varepsilon_0}>0$, we have
\bma\label{GMG2}
&\epsilon^{a-1}\left|\(\p_\beta^\alpha \Gamma\(g,\Mt\), \omega^2(\alpha,\beta)h\)\right|\nnm\\
\leq& C \sum_{\alpha_1 \leq \alpha} \sum_{\bar{\beta}\leq\beta_1\leq\beta}
 \epsilon^{a-1} \int_{\R}\Big|\mu^{\varepsilon_0}\p_{\bar{\beta}}^{\alpha_1}g\Big|_2
   \left|\omega(\alpha,\beta)\p_{\beta-\beta_1}^{\alpha-\alpha_1}\(\Mt\)\right|_\sigma
     \big|\omega(\alpha,\beta)h\big|_\sigma dy\nnm\\
\leq& C\lambda^{-1}\sum_{\alpha_1 \leq \alpha} \sum_{\bar{\beta}\leq\beta_1\leq\beta}
 \epsilon^{a-1} \int_{\R}\Big|\mu^{\varepsilon_0}\p_{\bar{\beta}}^{\alpha_1}g\Big|^2_2
   \left|\omega(\alpha,\beta)\p_{\beta-\beta_1}^{\alpha-\alpha_1}\(\Mt\)\right|^2_\sigma dy
  +\lambda\epsilon^{a-1}\|\omega(\alpha,\beta)h\|^2_\sigma\nnm\\
:=& C\lambda^{-1}\sum_{\alpha_1 \leq \alpha} \sum_{\bar{\beta}\leq\beta_1\leq\beta} J_1
 +\lambda\epsilon^{a-1}\|\omega(\alpha,\beta)h\|^2_\sigma.
\ema
Note that $|\beta|\geq 1$, we have $|\alpha|\leq 1$.

As $|\alpha|=0$, or $|\alpha|=|\alpha_1|=1$, by \eqref{norm-mu}, we have
\be
J_1
\leq C\epsilon^{a-1}\Big\|\Big| \omega(\alpha,\beta) \p_{\beta-\beta_1}\Big(\frac{M-\mu}{\sqrt{\mu}}\Big)\Big|_{\sigma}\Big\|^2_{L_y^\infty}
					\|\mu^{\varepsilon_0}\p^{\alpha}_{\bar{\beta}}g\|^2
\leq C\eta^2_1\mathcal{D}(\tau).\nnm
\ee
As $|\alpha|=1,~|\alpha_1|=0$, by  the Sobolev inequality, we have
\bma
J_1&
\leq C\epsilon^{a-1}\Big\||\mu^{\varepsilon_0}\p_{\bar{\beta}}g|_2\Big\|^2_{L_y^\infty}
  \Big\|\Big| \omega(\alpha,\beta) \p^{\alpha}_{\beta-\beta_1}\Big(\frac{M-\mu}{\sqrt{\mu}}\Big)\Big|_{\sigma}\Big\|^2\nnm\\
&\leq C\epsilon^{a-1}\|\p^{\alpha}(\rho,u,\theta)\|^2_{L^2_y}
   \|\p_{\bar{\beta}} g\|_\sigma\|\p_{\bar{\beta}}\p_y g\|_\sigma\nnm\\
&\leq C\(\mathcal{E}(\tau)+\sigma^2\frac{\epsilon^{2a}}{\delta^2}\)\mathcal{D}(\tau).\nnm
\ema
Consequently, putting the above two estimates into \eqref{GMG2} gives
\begin{equation}\label{GMG1-1}
\epsilon^{a-1}\left|\(\p^\alpha_\beta \Gamma\(g,\frac{M-\mu}{\sqrt{\mu}}\),\omega^2(\alpha,\beta)h\)\right|
\leq C\lambda^{-1}\(\eta^2_1+\mathcal{E}(\tau)+\sigma^2\frac{\epsilon^{2a}}{\delta^2}\)\mathcal{D}(\tau)
 +\lambda\epsilon^{a-1}\|\omega(\alpha,\beta)h\|^2_{\sigma}.
\end{equation}
Similarly,  the second term on the left hand of \eqref{GMG1}  satisfies \eqref{GMG1-1}.
Thus, the proof of \eqref{GMG1} is completed.
The case for $|\beta|=0$ and $|\alpha|\leq 1$ is similar to the above, so that we omits the details.
\end{proof}

\begin{lem}\label{QMGF}
Under the same assumptions as in Proposition \ref{priori3}, and let $g$ be $\gt$ or $\g$ and $\omega(\alpha,0)=\omega$, then for any $\alpha$ with $|\alpha|=2$, it holds that for $i=1,2,$
\bma\label{QMGF1}
&\epsilon^{a-1} \left|\(\p^\alpha \Gamma\(\Mt,g\)+\p^\alpha \Gamma\(g,\Mt\),\omega^2\frac{\p^\alpha F_i}{\sqrt{\mu}}\)\right|\nnm\\
\leq&\, C\lambda^{-1}\epsilon^{2a-2}\(\mathcal{E}(\tau)+\sigma^2\frac{\epsilon^2}{\delta^2}+\eta^2_1\)\mathcal{D}(\tau)
 +C\frac{\epsilon^{4a}}{\delta^2(\frac{\delta}{\sigma}+\epsilon^a\tau)^2}\|n\|^2_{L^2_y}\nnm\\
 &+C\frac{\epsilon^{4a-1}}{\delta(\frac{\delta}{\sigma}+\epsilon^a\tau)^2}+\lambda\epsilon^{2a-2}\mathcal{D}(\tau),
\ema
where $\lambda>0$ is a small constant, and $\eta_1>0$ is defined in \eqref{local}.
\end{lem}
\begin{proof}
We only consider the first term on the left hand of \eqref{QMGF1} because the other terms can be
estimated similarly.
 As $|\alpha|=2$, for a small constant ${\varepsilon_0}>0$, by Lemma \ref{Gamma}, we have
\bma\label{QMGF1-1}
&\epsilon^{a-1} \(\p^\alpha \Gamma\(\Mt,g\), \omega^2\frac{\p^\alpha F_i}{\sqrt{\mu}}\)\nnm\\
\leq&\, C \sum_{\alpha_1 \leq \alpha}
 \epsilon^{a-1} \int_{\R}|\mu^{\varepsilon_0}\p^{\alpha_1}g|_2
   \left|\omega\p^{\alpha-\alpha_1}\(\Mt\)\right|_\sigma
     \left|\omega \frac{\p^\alpha F_i}{\sqrt{\mu}}\right|_\sigma dy\nnm\\
\leq&\,C\lambda^{-1}\sum_{\alpha_1 \leq \alpha}
  \epsilon^{a-1} \int_{\R}|\mu^{\varepsilon_0}\p^{\alpha_1}g|^2_2 \left|\omega\p^{\alpha-\alpha_1}\(\Mt\)\right|^2_\sigma dy
 +\lambda\epsilon^{a-1}\|\mu^{-\frac12}\p^\alpha F_i\|^2_{\sigma,\omega}\nnm\\
:=&\,C\lambda^{-1}\sum_{\alpha_1 \leq \alpha}K_1
 +\lambda\epsilon^{a-1}\|\mu^{-\frac12}\p^\alpha F_i\|^2_{\sigma,\omega}.
\ema
As for $K_1$ in \eqref{QMGF1-1}, if $|\alpha_1|=|\alpha|=2$, by \eqref{norm-mu}, we have
\bma\label{QMGF1-2}
K_1
=\epsilon^{a-1} \int_{\R}|\mu^{\varepsilon_0}\p^{\alpha}g|^2_2 \left|\omega\(\Mt\)\right|^2_\sigma dy
\leq C\eta_1^2\epsilon^{a-1}\|\mu^{\varepsilon_0}\p^{\alpha}g\|^2_{\sigma,\omega}.
\ema
For the case of lower-order derivatives on $\gt$, if $|\alpha_1|=0$ by Lemma \ref{M}, Lemma \ref{rarefaction4} and the Sobolev embedding, we have
\bma
K_1
&=\epsilon^{a-1} \int_{\R}|\mu^{\varepsilon_0} g|^2_2 \left|\omega\p^\alpha\(\Mt\)\right|^2_\sigma dy\nnm\\
&\leq C\epsilon^{a-1}\||\mu^{\varepsilon_0} g|_2\|^2_{L^\infty_y}
  \bigg(\|\p^\alpha(\rho,u,\theta)\|^2_{L^2_y}+\sum_{|\alpha'|=1}\|\p^{\alpha'}(\rho,u,\theta)\|^4_{L^4_y}\bigg)\nnm\\
&\leq C\epsilon^{2a-2}\mathcal{E}(\tau)\mathcal{D}(\tau)
  +C\sigma^2\frac{\epsilon^{4a}}{\delta^4}\mathcal{D}(\tau).
\ema
If $|\alpha_1|=1$, we can estimate $K_1$ as follows.
\bma\label{QMGF1-3}
K_1
&=\epsilon^{a-1}\sum_{|\alpha_1|=1}\int_{\R}|\mu^{\varepsilon_0}\p^{\alpha_1}g|^2_2 \left|\omega\p^{\alpha-\alpha_1}\(\Mt\)\right|^2_\sigma dy\nnm\\
&\leq  C\epsilon^{a-1}\sum_{|\alpha_1|=1}\|\p^{\alpha-\alpha_1}(\rho,u,\theta)\|^2_{L^\infty_y}\|\mu^{\varepsilon_0} \p^{\alpha_1}g\|^2\nnm\\
&\leq  C\epsilon^{2a-2}\mathcal{E}(\tau)\mathcal{D}(\tau)
  +C\sigma^2\frac{\epsilon^{2a}}{\delta^2}\mathcal{D}(\tau).
\ema
Substituting the estimates \eqref{QMGF1-2}-\eqref{QMGF1-3} into \eqref{QMGF1-1}, and by using \eqref{F1-2order}-\eqref{F2-2order} with  $\frac{\epsilon^a}{\delta}$ being small,
we obtain \eqref{QMGF1}.
\end{proof}

\subsection{Interaction between microscopic components}
\begin{lem}\label{QGG}
Under the same assumptions as in Proposition \ref{priori3} and by letting $\frac{\epsilon^a}{\delta}$ small enough, for either $|\alpha|+|\beta|\leq 2$ and $|\beta|\geq1$, or  $|\beta|=0$ and $|\alpha|\leq 1$, we have for any well-defined function $h$,
\bma\label{QGG3}
&\epsilon^{a-1}\left|\(\p^\alpha_\beta\Gamma\(\frac{G_1}{\sqrt{\mu}},
  \frac{G_1}{\sqrt{\mu}}\),\omega^2(\alpha,\beta)h\)\right|\nnm\\
\leq&\, \lambda\epsilon^{a-1}\|\omega(\alpha,\beta)h\|^2_{\sigma}
	+C\lambda^{-1}\frac{\epsilon^3}{\delta(\frac{\delta}{\sigma}+\epsilon^a\tau)^2}
    +C\lambda^{-1}\(\sigma\frac{\epsilon^2}{\delta^2}+\mathcal{E}(\tau)\)\mathcal{D}(\tau),
\\
\label{QF2G}
&\epsilon^{a-1}\left|\(\p^\alpha_\beta\Gamma\(\frac{F_2}{\sqrt{\mu}},\frac{G_1}{\sqrt{\mu}}\),\omega^2(\alpha,\beta)h\)\right|\nnm\\
\leq&\,\lambda\epsilon^{a-1}\|\omega(\alpha,\beta)h\|^2_{\sigma}
 +C\lambda^{-1}\(\mathcal{E}(\tau)+\sigma^2\frac{\epsilon^{2a}}{\delta^2}\)\mathcal{D}(\tau)
 +C\lambda^{-1}\frac{\epsilon^{1+a}}{(\frac{\delta}{\sigma}+\epsilon^a\tau)^2}\|n\|^2_{L^2_y},
\ema
where $\lambda>0$ is a small constant.
\end{lem}
\begin{proof}
In the following, we only prove \eqref{QGG3} because the other estimates  can be treated similarly.
Let $|\alpha|+|\beta|\leq 2$ and $|\beta|\geq1$, by $G_1=\bar{G}_1+\sqrt{\mu}\gt$, we have
\bma\label{QGG1}
&\epsilon^{a-1}\(\p^\alpha_\beta\Gamma\(\frac{G_1}{\sqrt{\mu}},\frac{G_1}{\sqrt{\mu}}\), \omega^2(\alpha,\beta)h\)\nnm\\
=&\,\epsilon^{a-1}\(\p^\alpha_\beta\Gamma\(\frac{\bar{G}_1}{\sqrt{\mu}},\frac{\bar{G}_1}{\sqrt{\mu}}\), \omega^2(\alpha,\beta)h\)
 +\epsilon^{a-1}\(\p^\alpha_\beta\Gamma\(\frac{\bar{G}_1}{\sqrt{\mu}},\gt\), \omega^2(\alpha,\beta)h\)\nnm\\
&+\epsilon^{a-1}\(\p^\alpha_\beta\Gamma\(\gt,\frac{\bar{G}_1}{\sqrt{\mu}}\), \omega^2(\alpha,\beta)h\)
 +\epsilon^{a-1}\(\p^\alpha_\beta\Gamma\(\gt,\gt\), \omega^2(\alpha,\beta)h\).
\ema
For the first term on the right hand side of \eqref{QGG1}, and $\alpha=0$, we deduce from Lemma \ref{Gamma} and \eqref{barG7-1} that
for any small constant ${\varepsilon_0}>0$,
\bma
&\epsilon^{a-1}\left|\(\p^\alpha_\beta \Gamma\(\frac{\bar{G}_1}{\sqrt{\mu}},\frac{\bar{G}_1}{\sqrt{\mu}}\),
   \omega^2(\alpha,\beta)h\)\right|\nnm\\
\leq &\, C\epsilon^{a-1}\sum_{\bar{\beta}\leq\beta_1\leq\beta}
 \int_{\R}\bigg|\mu^{\varepsilon_0}\p_{\bar{\beta}}\(\frac{\bar{G}_1}{\sqrt{\mu}}\)\bigg|_2
   \bigg|\omega(0,\beta) \p_{\beta-\beta_1}\(\frac{\bar{G}_1}{\sqrt{\mu}}\)\bigg|_{\sigma}
    \big| \omega(\alpha,\beta)h\big|_{\sigma}dy\nnm\\
\leq&\,\lambda\epsilon^{a-1}\|\omega(\alpha,\beta)h\|^2_{\sigma}
		+C\lambda^{-1}\frac{\epsilon^3}{\delta(\frac{\delta}{\sigma}+\epsilon^a\tau)^2}.\nnm
\ema
As $|\alpha|=1$, by Lemma \ref{Gamma} and \eqref{barG7-2}, Lemma \ref{rarefaction4}, we have
\bma
&\epsilon^{a-1}\left|\(\p^\alpha_\beta \Gamma\(\frac{\bar{G}_1}{\sqrt{\mu}},\frac{\bar{G}_1}{\sqrt{\mu}}\),
   \omega^2(\alpha,\beta)h\)\right|\nnm\\
\leq&\,C\epsilon^{a-1}\sum_{\alpha_1\leq\alpha}\sum_{\bar{\beta}\leq\beta_1\leq\beta}
  \int_{\R}\bigg|\mu^{\varepsilon_0}\p^{\alpha_1}_{\bar{\beta}}\(\frac{\bar{G}_1}{\sqrt{\mu}}\)\bigg|_2
     \bigg|\omega(\alpha,\beta) \p^{\alpha-\alpha_1}_{\beta-\beta_1}\(\frac{\bar{G}_1}{\sqrt{\mu}}\)\bigg|_{\sigma}
     \big| \omega(\alpha,\beta)h\big|_{\sigma}dy\nnm\\
\leq&\,\lambda\epsilon^{a-1}\|\omega(\alpha,\beta)h\|^2_{\sigma}
	+C\lambda^{-1}\epsilon^{a-1}\int_{\R}\epsilon^{4-4a}|\dy(\bar{u}_{1},\bar{\theta})|^2
        \big(|\p^\alpha\dy(\bar{u}_{1},\bar{\theta})|+|\dy(\bar{u}_{1},\bar{\theta})| |\p^{\alpha}(\rho,u,\theta)|\big)^2dy\nnm\\
\leq&\,\lambda\epsilon^{a-1}\|\omega(\alpha,\beta)h\|^2_{\sigma}
	+C\lambda^{-1}\frac{\epsilon^{3+2a}}{\delta(\frac{\delta}{\sigma}+\epsilon^a\tau)^4}
    +C\lambda^{-1}\sigma^4\frac{\epsilon^{2+2a}}{\delta^4}\mathcal{D}(\tau).\nnm
\ema
Therefore, for the first term on the right hand side of \eqref{QGG1}, by letting $\frac{\epsilon^a}{\delta}\leq1$, we have
\be\label{QGG1-1}
\epsilon^{a-1}\left|\(\p^\alpha_\beta \Gamma\(\frac{\bar{G}_1}{\sqrt{\mu}},\frac{\bar{G}_1}{\sqrt{\mu}}\),
   \omega^2(\alpha,\beta)h\)\right|
\leq\lambda\epsilon^{a-1}\|\omega(\alpha,\beta)h\|^2_{\sigma}
	+C\lambda^{-1}\frac{\epsilon^3}{\delta(\frac{\delta}{\sigma}+\epsilon^a\tau)^2}
    +C\lambda^{-1}\sigma^2\frac{\epsilon^{2+2a}}{\delta^4}\mathcal{D}(\tau).
\ee
For the second term on the right-hand side of \eqref{QGG1}, by using Lemma \ref{Gamma} again, we have for any small constant ${\varepsilon_0}>0$,
\bma
&\epsilon^{a-1}\(\p^\alpha_\beta\Gamma\(\frac{\bar{G}_1}{\sqrt{\mu}},\gt\), \omega^2(\alpha,\beta)h\)\nnm\\
\leq& C\sum_{\alpha_1\leq\alpha}\sum_{\bar{\beta}\leq\beta_1\leq\beta}
				\epsilon^{a-1}\int_{\R}\left|\mu^{\varepsilon_0}\p^{\alpha_1}_{\bar{\beta}}\(\frac{\bar{G}_1}{\sqrt{\mu}}\)\right|_2
   |\omega(\alpha,\beta) \p^{\alpha-\alpha_1}_{\beta-\beta_1}\gt|_{\sigma}|\omega(\alpha,\beta)h|_{\sigma}dy\nnm\\
\leq&C\lambda^{-1}\sum_{\alpha_1\leq\alpha}\sum_{\bar{\beta}\leq\beta_1\leq\beta}
				\epsilon^{a-1}\int_{\R}\left|\mu^{\varepsilon_0}\p^{\alpha_1}_{\bar{\beta}}\(\frac{\bar{G}_1}{\sqrt{\mu}}\)\right|^2_2
   |\omega(\alpha,\beta) \p^{\alpha-\alpha_1}_{\beta-\beta_1}\gt|^2_{\sigma}dy
 +\lambda\epsilon^{a-1}\|h\|^{2}_{\sigma,\omega}\nnm\\
:=&C\lambda^{-1}J^1_g+\lambda\epsilon^{a-1}\|\omega(\alpha,\beta)h\|^2_{\sigma}.\nnm
\ema
As $|\alpha_1|=0$, by \eqref{barG7-1} and  Lemma \ref{rarefaction4}, we have
\bma
J^1_g=\epsilon^{a-1}\int_{\R}\left|\mu^{\varepsilon_0}\p_{\bar{\beta}}\(\frac{\bar{G}_1}{\sqrt{\mu}}\)\right|^2_2
   |\omega(\alpha,\beta) \p^\alpha_{\beta-\beta_1}\gt|^2_{\sigma}dy
\leq C\sigma^2\frac{\epsilon^2}{\delta^2} \mathcal{D}(\tau).\nnm
\ema
As $|\alpha_1|=|\alpha|=1$, by the Sobolev imbedding inequality, we have
\bma
&\epsilon^{a-1}\int_{\R}\left|\mu^{\varepsilon_0}\p^{\alpha}_{\bar{\beta}}\(\frac{\bar{G}_1}{\sqrt{\mu}}\)\right|^2_2
   |\omega(\alpha,\beta) \p_{\beta-\beta_1}\gt|^2_{\sigma}dy\nnm\\
\leq&C\epsilon^{1-a}\big\|\big|\omega(\alpha,\beta) \p_{\beta-\beta_1}\gt\big|_{\sigma}\big\|^2_{L^\infty_y}
 \int_{\R}\(|\p^{\alpha}\dy(\bar{u}_{1},\bar{\theta})|^2
  +|\dy(\bar{u}_{1},\bar{\theta})|^2|\p^{\alpha}(\rho,u,\theta)|^2\)dy\nnm\\
\leq&C\sigma^2\frac{\epsilon^{2+a}}{\delta^3}\mathcal{D}(\tau)
  +C\sigma^2\frac{\epsilon^2}{\delta^2}\mathcal{E}(\tau)\mathcal{D}(\tau).\nnm
\ema
Thus, for the second term on the right-hand side of \eqref{QGG1}, we have
\bma\label{QGG1-2}
\epsilon^{a-1}\(\p^\alpha_\beta\Gamma\(\frac{\bar{G}_1}{\sqrt{\mu}},\gt\), \omega^2(\alpha,\beta)h\)
\leq \lambda\epsilon^{a-1}\|\omega(\alpha,\beta)h\|^2_{\sigma}
 +C\lambda^{-1}\sigma^2\frac{\epsilon^{2+a}}{\delta^3} \mathcal{D}(\tau)
 +C\lambda^{-1}\sigma^2\frac{\epsilon^2}{\delta^2}\mathcal{E}(\tau)\mathcal{D}(\tau).
\ema
Note that the third term on the right-hand side of \eqref{QGG1} has the same bound as the above.
			
It now remains to handle the last term of \eqref{QGG1}. From Lemma \ref{Gamma}, we have for any small constant ${\varepsilon_0}>0$,
\bma\label{QGG1-3}
&\epsilon^{a-1}|(\p^\alpha_\beta \Gamma(\gt,\gt),\omega^2(\alpha,\beta)h)|\nnm\\
\leq& C\sum_{\alpha_1\leq\alpha}\sum_{\bar{\beta}\leq\beta_1\leq\beta}
		\epsilon^{a-1}\int_{\R}|\mu^{\varepsilon_0}\p^{\alpha_1}_{\bar{\beta}}\gt|_2
 |\omega(\alpha,\beta) \p^{\alpha-\alpha_1}_{\beta-\beta_1}\gt|_{\sigma}|\omega(\alpha,\beta)h|_{\sigma}dy\nnm\\
\leq&\lambda\epsilon^{a-1}\|\omega(\alpha,\beta)h\|^2_{\sigma}+C\lambda^{-1}\mathcal{E}(\tau)\mathcal{D}(\tau).
\ema
In summary, plugging \eqref{QGG1-1}-\eqref{QGG1-3} into \eqref{QGG1} gives \eqref{QGG3}.
\end{proof}
		
\begin{lem}\label{GGF}
Under the same assumptions as in Proposition \ref{priori3}, and let $\frac{\epsilon^a}{\delta}$ small enough and $\omega=\omega(\alpha,0)$, then for any $\alpha$ with $|\alpha|=2$, it holds that
\bma\label{GGF1}					&\epsilon^{a-1}\left|\(\p^{\alpha}\Gamma\(\frac{G_1}{\sqrt{\mu}},\frac{G_1}{\sqrt{\mu}}\),\omega^2\frac{\p^{\alpha}F_1}{\sqrt{\mu}}\)\right|
 +\epsilon^{a-1}\left|\(\p^{\alpha}\Gamma\(\frac{G_2}{\sqrt{\mu}},\frac{G_1}{\sqrt{\mu}}\),\omega^2\frac{\p^{\alpha}F_2}{\sqrt{\mu}}\)\right|\nnm\\
\leq& C\lambda^{-1}\(\sigma^2\frac{\epsilon^2}{\delta^2}+\epsilon^{2a-2}\mathcal{E}(\tau)\)\mathcal{D}(\tau)
  +C\lambda^{-1}\frac{\epsilon^{3+4a}}{\delta^3(\frac{\delta}{\sigma}+\epsilon^a\tau)^4}
  +\lambda\epsilon^{a-1}\|\mu^{-\frac12}\p^{\alpha}(F_1,F_2)\|^2_{\sigma,\omega},
\ema
where $\lambda>0$ is a small constant.
\end{lem}
\begin{proof}
We  only consider the first term on the left hand of \eqref{GGF1} because the other terms can be treated  similarly.
Since $G_1=\bar{G}_1+\sqrt{\mu}\gt$, the first term on the left-hand side of \eqref{GGF1} is equivalent to
\begin{equation}\label{GGF2}				
\epsilon^{a-1}\left|\(\p^\alpha\Gamma\(\frac{\bar{G}_1}{\sqrt{\mu}},\frac{\bar{G}_1}{\sqrt{\mu}}\)
  +\p^\alpha\Gamma\(\frac{\bar{G}_1}{\sqrt{\mu}},\gt\)
  +\p^\alpha\Gamma\(\gt,\frac{\bar{G}_1}{\sqrt{\mu}}\)
  +\p^\alpha\Gamma\(\gt,\gt\),\omega^2\frac{\p^{\alpha}F_1}{\sqrt{\mu}}\)\right|.
\end{equation}
For the first term of \eqref{GGF2}, by using  Lemma \ref{Gamma}, we arrive at for any small constant ${\varepsilon_0}>0$,
\allowdisplaybreaks\begin{align*}
&\epsilon^{a-1}\left|\(\p^{\alpha}\Gamma\(\frac{\bar{G}_1}{\sqrt{\mu}},\frac{\bar{G}_1}{\sqrt{\mu}}\),\omega^2\frac{\p^{\alpha}F_1}{\sqrt{\mu}}\)\right|\\
\leq& C\epsilon^{a-1} \sum_{\alpha_1 \leq \alpha}
  \int_{\R}\bigg|\mu^{\varepsilon_0}\p^{\alpha_1}\(\frac{\bar{G}_1}{\sqrt{\mu}}\)\bigg|_2
	 \bigg|\omega\p^{\alpha-\alpha_1}\(\frac{\bar{G}_1}{\sqrt{\mu}}\)\bigg|_\sigma
     \bigg|\omega \frac{\p^\alpha F_1}{\sqrt{\mu}}\bigg|_\sigma dy\nnm\\
\leq&C\lambda^{-1}\epsilon^{a-1}\sum_{\alpha_1\leq\alpha}\int_{\R}
     \bigg|\mu^{\varepsilon_0}\p^{\alpha_1}\(\frac{\bar{G}_1}{\sqrt{\mu}}\)\bigg|^2_{2}
	 \bigg|\omega\p^{\alpha-\alpha_1}\(\frac{\bar{G}_1}{\sqrt{\mu}}\)\bigg|_{\sigma}^2dy
  +\lambda\epsilon^{a-1}\|\mu^{-\frac12}\p^{\alpha}F_1\|^2_{\sigma,\omega}\nnm\\
:=&C\lambda^{-1}\sum_{\alpha_1 \leq \alpha}K^1_g
  +\lambda\epsilon^{a-1}\|\mu^{-\frac12}\p^{\alpha}F_1\|^2_{\sigma,\omega}.
\nnm
\end{align*}
For $|\alpha_1|=0$ and $|\alpha_1|=2$, by \eqref{barG7-1}, \eqref{Gb3} and Lemma \ref{rarefaction4}, we have
\bma
K^1_g
=&\epsilon^{a-1}
 \int_{\R}\bigg|\mu^{\varepsilon_0}\(\frac{\bar{G}_1}{\sqrt{\mu}}\)\bigg|^2_2\bigg|\omega\p^\alpha\(\frac{\bar{G}_1}{\sqrt{\mu}}\)\bigg|_{\sigma}^2
  +\bigg|\mu^{\varepsilon_0}\p^\alpha\(\frac{\bar{G}_1}{\sqrt{\mu}}\)\bigg|^2_2\bigg|\omega\(\frac{\bar{G}_1}{\sqrt{\mu}}\)\bigg|_{\sigma}^2dy\nnm\\
\leq& \epsilon^{1-a}\|\dy(\ub,\thetab)\|^2_{L^\infty_y}
 \(\left\|\omega \frac{\p^\alpha\bar{G}_1}{\sqrt{\mu}}\right\|_{\sigma}^2+\left\|\frac{\p^\alpha\bar{G}_1}{\sqrt{\mu}}\right\|^2\)\nnm\\
\leq& C\frac{\epsilon^{3+4a}}{\delta^3(\frac{\delta}{\sigma}+\epsilon^a\tau)^4}
  +\sigma^4\frac{\epsilon^{2+2a}}{\delta^4}\mathcal{D}(\tau).\nnm
\ema
For $|\alpha_1|=1$, by \eqref{barG7-2},  Lemma \ref{rarefaction4} and the Sobolev imbedding inequality, we have
\bma
K^1_g
&=\epsilon^{a-1}\int_{\R}\bigg|\mu^{\varepsilon_0}\p^{\alpha_1}\(\frac{\bar{G}_1}{\sqrt{\mu}}\)\bigg|^2_{2}
				\bigg|\omega\p^{\alpha-\alpha_1}\(\frac{\bar{G}_1}{\sqrt{\mu}}\)\bigg|_{\sigma}^2\,dy\nnm\\
&\leq C\epsilon^{3-3a}\sum_{|\alpha_1|=1}\int_{\R}\(|\p^{\alpha_1}\dy(\bar{u}_{1},\bar{\theta})|
       +|\dy(\bar{u}_{1},\bar{\theta})||\p^{\alpha_1}(\rho,u,\theta)|\)^4dy\nnm\\
&\leq C\frac{\epsilon^{3+4a}}{\delta^3(\frac{\delta}{\sigma}+\epsilon^a\tau)^4}
  +C\sigma^4\frac{\epsilon^{2+2a}}{\delta^4}\mathcal{D}(\tau).\nnm
\ema
Then we have
\bma\label{GGF2-1}
&\epsilon^{a-1}\left|\(\p^{\alpha}\Gamma\(\frac{\bar{G}_1}{\sqrt{\mu}},\frac{\bar{G}_1}{\sqrt{\mu}}\),\omega^2\frac{\p^{\alpha}F_1}{\sqrt{\mu}}\)\right|\nnm\\
\leq&\, C\lambda^{-1}\frac{\epsilon^{3+4a}}{\delta^3(\frac{\delta}{\sigma}+\epsilon^a\tau)^4}
 +C\lambda^{-1}\sigma^2\frac{\epsilon^{2 }}{\delta^2}\mathcal{D}(\tau)
 +\lambda\epsilon^{a-1}\|\mu^{-\frac12}\p^{\alpha}F_1\|^2_{\sigma,\omega},
\ema
where we have used $\frac{\epsilon^a}{\delta}\leq1$. Next, for the second term of \eqref{GGF2}, by using  Lemma \ref{Gamma}, we have
\begin{align*}
&\epsilon^{a-1}\left|\(\p^{\alpha}\Gamma\(\frac{\bar{G}_1}{\sqrt{\mu}},\gt\),\omega^2\frac{\p^{\alpha}F_1}{\sqrt{\mu}}\)\right|\\
\leq& C\epsilon^{a-1} \sum_{\alpha_1 \leq \alpha}
  \int_{\R}\bigg|\mu^{\varepsilon_0}\p^{\alpha_1}\(\frac{\bar{G}_1}{\sqrt{\mu}}\)\bigg|_2
	 \big|\omega\p^{\alpha-\alpha_1}\gt\big|_\sigma
     \bigg|\omega \frac{\p^\alpha F_1}{\sqrt{\mu}}\bigg|_\sigma dy\nnm\\
\leq&C\lambda^{-1}\epsilon^{a-1}\sum_{\alpha_1\leq\alpha}
   \int_{\R}\bigg|\mu^{\varepsilon_0}\p^{\alpha_1}\(\frac{\bar{G}_1}{\sqrt{\mu}}\)\bigg|^2_{2}
	 \big|\omega\p^{\alpha-\alpha_1}\gt\big|_{\sigma}^2dy
  +\lambda\epsilon^{a-1}\|\mu^{-\frac12}\p^{\alpha}F_1\|^2_{\sigma,\omega}\nnm\\
:=&C\lambda^{-1}\sum_{\alpha_1 \leq \alpha}K^2_g
  +\lambda\epsilon^{a-1}\|\mu^{-\frac12}\p^{\alpha}F_1\|^2_{\sigma,\omega}.\nnm
\end{align*}
For $|\alpha_1|=0$, by \eqref{barG7-1} and Lemma \ref{rarefaction4}, we have
\be
K^2_g
=\epsilon^{a-1}\int_{\R}\bigg|\mu^{\varepsilon_0}\(\frac{\bar{G}_1}{\sqrt{\mu}}\)\bigg|^2_2 \big|\omega\p^{\alpha}\gt\big|_{\sigma}^2dy
\leq C\sigma^2\frac{\epsilon^2}{\delta^2}\mathcal{D}(\tau).\nnm
\ee
For $|\alpha_{1}|=1$, by using $\omega(\alpha,0)\leq \omega(\alpha-\alpha_1,0)$, \eqref{barG7-2} and  Lemma \ref{rarefaction4}, it holds that
\allowdisplaybreaks\begin{align*}
K^2_g
&\leq C\epsilon^{a-1}\bigg\|\bigg|\mu^{\varepsilon_0}\p^{\alpha_1}\(\frac{\bar{G}_1}{\sqrt{\mu}}\)\bigg|_2\bigg\|^2_{L^\infty_y}
		 \big\|\omega\p^{\alpha-\alpha_{1}}\gt \big\|^2_{\sigma}\nnm\\
&\leq C\epsilon^{2-2a}\Big\|\(|\p^{\alpha_1}\dy(\bar{u}_{1},\bar{\theta})|
   +|\dy(\bar{u}_{1},\bar{\theta})||\p^{\alpha_1}(\rho,u,\theta)|\)^2\Big\|_{L^\infty_y}\mathcal{D}(\tau)\nnm\\
&\leq C\sigma^2\frac{\epsilon^{2+2a}}{\delta^4}\mathcal{D}(\tau)
	+C\sigma^2\frac{\epsilon^{1+a}}{\delta^2}\mathcal{E}(\tau)\mathcal{D}(\tau).
\end{align*}
For $|\alpha_{1}|=2$, by using $\omega(\alpha,0)\leq \omega(0,0)$ and \eqref{Gbsd}, we have
\allowdisplaybreaks\begin{align*}
K^2_g&=\epsilon^{a-1}\int_{\R}\bigg|\mu^{\varepsilon_0}\p^{\alpha}\(\frac{\bar{G}_1}{\sqrt{\mu}}\)\bigg|^2_{2}
	 \big|\omega\gt\big|_{\sigma}^2dy
  \leq C\epsilon^{a-1}\bigg\|\mu^{\varepsilon_0}\p^{\alpha}\(\frac{\bar{G}_1}{\sqrt{\mu}}\)\bigg\|^2
     \big\||\omega(0,0)\gt|_{\sigma}\big\|^2_{L^\infty_y}\nnm\\
&\leq C\(\frac{\epsilon^{2+3a}}{\delta^3(\frac{\delta}{\sigma}+\epsilon^a\tau)^2}
  +\frac{\epsilon^{2+2a}}{\delta^2(\frac{\delta}{\sigma}+\epsilon^a\tau)^2}\sum_{|\alpha_1|=1}\|\p^{\alpha_1}(\phi,\psi,\zeta)\|^2_{L^2_y}
  +\frac{\epsilon^2}{(\frac{\delta}{\sigma}+\epsilon^a\tau)^2} \|\p^{\alpha}(\phi,\psi,\zeta)\|^2_{L^2_y}\)\mathcal{D}(\tau)\nnm\\
&\leq C\sigma^2\frac{\epsilon^{2+3a}}{\delta^5}\mathcal{D}(\tau)
  +C\sigma^2\frac{\epsilon^{2a}}{\delta^2}\mathcal{E}(\tau)\mathcal{D}(\tau).
\end{align*}
Combining the above estimates yields
\bma\label{GGF2-2}
&\epsilon^{a-1}\bigg|\(\p^{\alpha}\Gamma\(\frac{\bar{G}_1}{\sqrt{\mu}},\gt\),\omega^2\frac{\p^{\alpha}F_1}{\sqrt{\mu}}\)\bigg|\nnm\\
\leq&\, C\lambda^{-1}\sigma^2\frac{\epsilon^2}{\delta^2}\mathcal{D}(\tau)
  +C\lambda^{-1}\sigma^2 \mathcal{E}(\tau)\mathcal{D}(\tau)
  +\lambda\epsilon^{a-1}\|\mu^{-\frac12}\p^{\alpha}F_1\|^2_{\sigma,\omega}.
\ema
where we have used $\frac{\epsilon^a}{\delta}\leq1$.
Similar estimates also hold for the third term on the left-hand side of \eqref{GGF2}. For the last term in \eqref{GGF2}, by  Lemma \ref{Gamma}, we have for any small constant ${\varepsilon_0}>0$,
\bma
&\epsilon^{a-1}\bigg|\(\p^{\alpha}\Gamma(\gt,\gt),\omega^2\frac{\p^{\alpha}F_1}{\sqrt{\mu}}\)\bigg|\nnm\\
\leq&C\lambda^{-1}\epsilon^{a-1}\sum_{\alpha_1\leq\alpha}
   \int_{\R}|\mu^{\varepsilon_0}\p^{\alpha_1}\gt|^2_2
	 |\omega\p^{\alpha-\alpha_1}\gt|_{\sigma}^2dy
  +\lambda\epsilon^{a-1}\|\mu^{-\frac12}\p^{\alpha}F_1\|^2_{\sigma,\omega}\nnm\\
:=&C\lambda^{-1}\sum_{\alpha_1 \leq \alpha}K^3_g
  +\lambda\epsilon^{a-1}\|\mu^{-\frac12}\p^{\alpha}F_1\|^2_{\sigma,\omega}.\nnm
\ema
For $|\alpha_{1}|=0$ and $|\alpha_1|=2$, by using the Sobolev imbedding, we have
\allowdisplaybreaks\begin{align*}
K^3_g &\leq C\epsilon^{a-1}\||\mu^{\varepsilon_0}\gt|_2\|^2_{L_y^\infty}\|\p^{\alpha}\gt\|^2_{\sigma,\omega}+C\epsilon^{a-1}\|\mu^{\varepsilon_0}\p^{\alpha}\gt\|^2
				\||\gt|_{\sigma,\omega}\|^2_{L^\infty_y}\nnm\\
 &\leq C\epsilon^{a-1}\|\mu^{\varepsilon_0}\gt\| \|\mu^{\varepsilon_0}\dy\gt\|\|\p^{\alpha}\gt\|^2_{\sigma,\omega}+C\epsilon^{a-1}\|\mu^{\varepsilon_0}\p^{\alpha}\gt\|^2\|\gt\|_{\sigma,\omega}\|\dy\gt\|_{\sigma,\omega}\nnm\\
&\leq C\epsilon^{2a-2}\mathcal{E}(\tau)\mathcal{D}(\tau),
\end{align*}
and for $|\alpha_{1}|=1$,
\allowdisplaybreaks\begin{align*}
K^3_g
 &\leq C\epsilon^{a-1}\sum_{|\alpha_1|=1}\||\mu^{\varepsilon_0}\p^{\alpha_1}\gt|_2\|^2_{L_y^\infty}\|\p^{\alpha-\alpha_1}\gt\|^2_{\sigma,\omega}\nnm\\
 &\leq C\epsilon^{a-1}\sum_{|\alpha_1|=1}\|\mu^{\varepsilon_0}\p^{\alpha_1}\gt\| \|\mu^{\varepsilon_0}\p^{\alpha_1}\dy\gt\|  \|\p^{\alpha-\alpha_1}\gt\|^2_{\sigma,\omega}\nnm\\
&\leq C\epsilon^{a-1}\mathcal{E}(\tau)\mathcal{D}(\tau).
\end{align*}

From the above estimates on $K^3_g$, we obtain	
\begin{equation}\label{GGF2-3}
\epsilon^{a-1}\bigg|\(\p^{\alpha}\Gamma(\gt,\gt),\omega^2\frac{\p^{\alpha}F_1}{\sqrt{\mu}}\)\bigg|
\leq C\lambda^{-1}\epsilon^{2a-2}\mathcal{E}(\tau)\mathcal{D}(\tau)
  +\lambda\epsilon^{a-1}\|\mu^{-\frac12}\p^{\alpha}F_1\|^2_{\sigma,\omega}.
\end{equation}
Hence, plugging all the above estimates \eqref{GGF2-1}-\eqref{GGF2-3} into \eqref{GGF2} gives
\bma					
&\epsilon^{a-1}\left|\(\p^\alpha\Gamma\(\frac{G_1}{\sqrt{\mu}},\frac{G_1}{\sqrt{\mu}}\),\omega^2\frac{\p^{\alpha}F_1}{\sqrt{\mu}}\)\right|\nnm\\
\leq& C\lambda^{-1}\Big(\sigma^2\frac{\epsilon^2}{\delta^2}+\epsilon^{2a-2}\mathcal{E}(\tau)\Big)\mathcal{D}(\tau)
+C\lambda^{-1}\frac{\epsilon^{3+4a}}{\delta^3(\frac{\delta}{\sigma}+\epsilon^a\tau)^4}
  +\lambda\epsilon^{a-1}\|\mu^{-\frac12}\p^{\alpha}F_1\|^2_{\sigma,\omega},
\ema
which gives  estimate \eqref{GGF1} and  then  the Lemma \ref{GGF} follows.
\end{proof}

\medskip
\noindent {\bf Acknowledgements:} Mingying Zhong's research is  supported by the special foundation for Guangxi Ba Gui Scholars,  the National Natural Science Foundation of China  grants No. 12171104 and Innovation Project of Guangxi Graduate Education YCBZ2024004. Tong Yang's research is supported by a fellowship award from the Research Grants Council of the Hong Kong Special Administrative Region, China (Proiect no. SRFS2021-1S01). Tong Yang would also like to thank the Kuok foundation for its generous support and the Center for Nonlinear Analysis in The Hong Kong Polytechnic University.

\end{document}